\documentclass[a4paper,10pt,reqno]{amsart}
\usepackage{amsmath,amsfonts,amsthm,amssymb,color}
\usepackage[T1]{fontenc}
\usepackage{pdfsync}
\usepackage{pstricks}
\usepackage{lmodern}
\usepackage[normalem]{ulem}
\usepackage{mathrsfs}
\usepackage[extra]{tipa}
\usepackage{graphicx}
\usepackage{csquotes}

\numberwithin{equation}{section}

\newcommand{\fouri}[1]{\mathcal{F}(#1)}

\newcommand{\vertiii}[1]{{\left\vert\kern-0.25ex\left\vert\kern-0.25ex\left\vert #1 
    \right\vert\kern-0.25ex\right\vert\kern-0.25ex\right\vert}}

\newcommand{\R}{\mathbb R}

\newcommand{\Z}{\mathbb Z}

\newcommand{\be}{\beta}

\newcommand{\1}{{\bf 1}}

\newcommand{\cb}{\mathcal B}
\newcommand{\cac}{\mathcal C}
\newcommand{\cd}{\mathcal D}

\newcommand{\cf}{\mathcal F}
\newcommand{\ch}{\mathcal H}
\newcommand{\ci}{\mathcal I}

\newcommand{\cl}{\mathcal L}

\newcommand{\cq}{\mathcal Q}
\newcommand{\cs}{\mathcal S}

\newcommand{\struc}{\mathscr T}
\newcommand{\scal}{\mathfrak{s}}
\newcommand{\compac}{\mathfrak{K}}

\newcommand{\al}{\alpha}

\newcommand{\gga}{\Gamma}

\newcommand{\ep}{\varepsilon}

\newcommand{\la}{\lambda}

\newcommand{\vp}{\varphi}

\newcommand{\lln}{\left|}
\newcommand{\rrn}{\right|}

\newtheorem{theorem}{Theorem}[section]

\newtheorem{corollary}[theorem]{Corollary}

\newtheorem{definition}[theorem]{Definition}

\newtheorem{lemma}[theorem]{Lemma}

\newtheorem{proposition}[theorem]{Proposition}

\theoremstyle{remark}
\newtheorem{remark}[theorem]{Remark}
\theoremstyle{remark}

\newcommand{\bean}{\begin{eqnarray*}}
\newcommand{\eean}{\end{eqnarray*}}
\newcommand{\ben}{\begin{enumerate}}
\newcommand{\een}{\end{enumerate}}
\newcommand{\beq}{\begin{equation}}
\newcommand{\eeq}{\end{equation}}

\begin{document}

\date{\today}



\begin{center}
{\large\textbf{
Construction and Skorohod representation of a fractional $K$-rough path
}}\\~\\
Aur\'elien Deya\footnote{Institut \'Elie Cartan, Universit\' e de Lorraine, BP 70239, 54506 Vandoeuvre-l\`es-Nancy, France. Email: {\tt aurelien.deya@univ-lorraine.fr}}
\end{center}

\bigskip

{\small \noindent {\bf Abstract:} We go ahead with the study initiated in \cite{deya} about a heat-equation model with non-linear perturbation driven by a space-time fractional noise. Using general results from Hairer's theory of regularity structures, the analysis reduces to the construction of a so-called $K$-rough path (above the noise), a notion we introduce here as a compromise between regularity structures formalism and rough paths theory. The exhibition of such a $K$-rough path at order three allows us to cover the whole roughness domain that extends up to the standard space-time white noise situation. We also provide a representation of this abstract $K$-rough path in terms of Skorohod stochastic integrals.

\bigskip


\bigskip

}

\small

\section{Introduction}

The aim of this paper is to go a few steps further into the analysis of the SPDE model introduced in \cite{deya}, namely the equation
\begin{equation}\label{equation-base}
(\partial_t Y)(t,x)=(\partial^2_x Y)(t,x)+F(x,Y(t,x)) \, (\partial_t\partial_x B)(t,x) \quad , \quad Y(0,x)=\Psi(x) \ , \ t\in [0,T] \ , \ x\in \R \ ,
\end{equation}
where $F:\R \times \R \to \R$ is a quite general vector fields, $\Psi$ is a continuous function, and $\partial_t\partial_x B$ stands for a space-time fractional noise. To be more specific, $B$ is here a fractional sheet with Hurst indexes $H_1,H_2$, that is a centered Gaussian field with covariance function given by the formula
$$\mathbb{E}\big[ B(s,x) B(t,y)\big]=R_{H_1}(s,t) R_{H_2}(x,y) \ , \ \text{where} \ R_H(s,t):=\frac12 \{|s|^{2H}+|t|^{2H}- |t-s|^{2H}\} \ .$$

\smallskip

At this point, let us recall that the whole difficulty raised by this equation (at least when $H_1\neq \frac12$) lies in the fact that $B$ is not a martingale process, which rules out the possibility to study this model within the classical SPDE framework of \cite{daprato-zabczyk} or \cite{walsh}. It is then natural to turn to pathwise methods, and in fact, this equation provides us with an interesting example to test the flexibility of the theory of regularity structures - RS in the sequel - recently introduced by M. Hairer in \cite{hai-14}. The machinery has already proved to be a very powerful tool to study stochastic parabolic dynamics, as a flourishing literature can easily testify. To mention but a few applications, we can quote for instance \cite{chandra-shen,berglund-kuehn,hairer-labbe,hairer-pardoux,hairer-weber}.

\smallskip

In this context, our objective behind the study of (\ref{equation-base}) is actually manifold:

\smallskip

$\bullet$  RS theory is essentially built upon a sophisticated extension of concepts from Lyons' rough paths (RP) theory. However, based on the introductory paper \cite{hai-14}, the fundamental analogies between RP and RS theories may not be obvious to a non-initiated reader. Therefore, and in the continuity of \cite{deya}, we here propose to somehow go one step back into the formulation, by highlighting the role of an object whose definition and properties look very much like those a classical rough path: the so-called $K$-rough path (see Definition \ref{defi:k-rough-path}). Of course, as the RS-expert reader will soon realize, restricting the analysis to this sole concept of a $K$-rough path reduces the possible scope of application of RS results, in comparison with the general abstract formalism settled in \cite{hai-14}. In brief, $K$-rough paths are specifically designed for the dynamics of (\ref{equation-base}), and their definition must be reshaped when turning to other models such as KPZ or $\Phi^4_3$ equations. This being said, we think, or at least we hope, that this more straightforward presentation may help the reader to catch the very \enquote{rough-path} essence of RS theory, on a non-trivial SPDE example.

\smallskip

$\bullet$  Still focusing on the notion of a $K$-rough path, the analysis will give us the opportunity to recall how the technical tools used to study rough paths (Hölder topologies, Garsia-Rodemich-Rumsey lemma,...) extend to the parabolic framework. We will also see that the renormalization procedures, one of the main achievements of RS theory, can be very conveniently expressed in this setting.

\smallskip

$\bullet$ In \cite{deya}, the above ideas were implemented for a second-order analysis, which in fact corresponds to the situation where $2H_1+H_2 >\frac53$. We will here go one step further and consider the study of the model up to third order, which covers the case $2H_1+H_2>\frac32$. This extension will therefore give us the opportunity to go deeper into the exploration of the concept of a $K$-rough path. What also makes this work important to us is that it makes the link with the classical space-time white noise situation, for which $H_1=H_2=\frac12$ (and accordingly $2H_1+H_2=\frac32$). In brief, thanks to the subsequent results, we are now able to cover the whole domain that extends up to - but does not include - the standard space-time white noise (see the comments of Figure \ref{figure-h1-h2} for more details). It is worth mentioning that in the (very) particular martingale situation where $H_1=H_2=\frac12$, the RS machinery has been implemented by M. Hairer and E. Pardoux in \cite{hairer-pardoux}, leading the authors to a Wong-Zakaï-type property similar to our forthcoming Corollary \ref{coro:equation}.

\smallskip

$\bullet$ The exhibition of a third-order $K$-rough path above the fractional noise turns out to be a quite technical task. The construction differs from those in the white-noise situation, due to some fractional kernel to be dragged throughout the computations. Our analysis relies on Fourier techniques inherited from the harmonizable (or Fourier) representation of the fractional sheet, that is the formula
\begin{equation}\label{representation-sheet}
B(t,x)=c_{H_1,H_2}\int_{\R^2} \fouri{W}(d\xi,d\eta) \, \frac{e^{\imath t\xi}-1}{|\xi|^{H_1+\frac12}} \frac{e^{\imath x\eta}-1}{|\eta|^{H_2+\frac12}} \ , 
\end{equation}
for some appropriate constant $c_{H_1,H_2} >0$ and where $\fouri{W}$ stands for the Fourier transform of a space-time white noise in $\R^2$. Our main result, namely the existence of a $K$-rough path above $\partial_t\partial_x B$, can then be seen as a parabolic version of the results of Coutin and Qian about fractional rough paths (see \cite{coutin-qian}). Just as in the latter reference, we will also be able to provide a decomposition of our fractional $K$-rough path in terms of Skorohod integrals, to be compared with the formulas in \cite[Theorem 4]{coutin-qian}. We consider these chaos-decomposition formulas another substantial improvement with respect to the study in \cite{deya}: we are here able to describe our abstract $K$-rough path in terms of some \enquote{pre-existing} stochastic tools. In the - very - particular situation of a white-in-time noise, the decomposition reduces to the sole \enquote{Itô} $K$-rough path, an identification that can then be transposed on the level of the equation itself (see the last statement in Corollary \ref{coro:equation}).

\

One of the main ideas in both RP and RS theories can be summed up as follows - in a loose manner, of course: in order to interpret and solve the noisy differential equation under consideration (standard differential equations for RP theory, parabolic equations for RS theory), and therefore give a sense to the implicitly-defined solution, we only need to study a finite number of objects that are explicitly defined in terms of the noise only. In other words, all the successive operations involved in the equation, i.e., composition with a smooth vector fields, multiplication with the noise, integration and even the fixed-point argument, nicely combine around these few explicit objects, called the rough path in RP theory and the $K$-rough path in our setting.
\smallskip

In this paper about Equation (\ref{equation-base}), and for the sake of conciseness, we will not come back to the description of the sophisticated machinery that associates a $K$-rough path (or a \enquote{model} along Hairer's terminology) with a solution of the equation. The details of this sophisticated deterministic procedure can be found in \cite[Sections 4 to 7]{hai-14}, as well as in \cite[Section 2]{deya} for a shorter version applying specifically to the dynamics of (\ref{equation-base}). 

\smallskip

Thus, in what follows, we will only stick to the problem of constructing a $K$-rough path above the fractional noise. A natural way to initiate this construction is to start from the so-called canonical $K$-rough path associated with a given smooth approximation $B^n$ of the rough fractional noise $B$. This object has to be seen as the parabolic counterpart of the iterated integrals - or \enquote{canonical rough path} - of RP theory: just as its one-parameter model, the canonical $K$-rough is indeed derived from a Taylor expansion of the classical equation associated with the smooth path $B^n$. A specific description of this object, that we will denote by $\mathbf{B}^n$ in the sequel, will be given in Definition \ref{defi:canonical}.

\smallskip

Showing the convergence of $\mathbf{B}^n$ would then immediately provide us with a $K$-rough path above $\partial_t\partial_x B$, as desired - again, the situation can for instance be compared with the RP example treated in \cite{coutin-qian}. Unfortunately, as soon as $2H_1+H_2 \leq 2$, such a convergence happens to fail, which forces us to turn to renormalization tricks and exploit the flexibility of the definition of a $K$-rough path, as illustrated by Lemma \ref{lem:transfo-m}. This divergence phenomenon and the need for an appropriate renormalization were already observed at second order in \cite{deya}, and can be compared with the well-known divergence properties of KPZ or parabolic Anderson models. Couterbalancing the explosion will prove to be an intricate task at third order, with correction terms inspired by the chaos expansion of the canonical $K$-rough path.

\smallskip

Throughout the study,  we will consider the approximation $B^n$ of $B$ given by the formula: 
\begin{equation}\label{approx-noise}
B^n(s,x)=c_{H_1,H_2}\int_{D_n} \fouri{W}(d\xi,d\eta) \, \frac{e^{\imath s\xi}-1}{|\xi|^{H_1+\frac12}} \frac{e^{\imath x\eta}-1}{|\eta|^{H_2+\frac12}} \ , 
\end{equation}
where $D_n:=[-2^{2n},2^{2n}] \times [-2^n,2^n]$, and $c_{H_1,H_2},\fouri{W}$ are defined just as in (\ref{representation-sheet}). Thanks to the isometry properties satisfied by $W$ (or $\mathcal{F}(W)$), such an approximation readily yields explicit and manageable formulas (in terms of the fractional kernel) when computating related moments. Another advantage of the representation is that the Malliavin calculus with respect to $B^n$ (one of the keys of our analysis) can easily be connected with the standard Malliavin calculus for $W$, or $\mathcal{F}(W)$, as we will see it in Section \ref{sec:mall-cal}. This being said, we are pretty sure that the consideration of a mollifying procedure (just as in \cite{hai-14}) would lead to very similar constructions and results. Consider indeed a mollifying sequence $\rho_n(s,x):=2^{3n} \rho(2^{2n}s,2^n x)$ on $\R^2$ and denote temporarily
$$L_{(s,x)}(\xi,\eta):=c_{H_1,H_2} \, \frac{e^{\imath s\xi}-1}{|\xi|^{H_1+\frac12}} \frac{e^{\imath x\eta}-1}{|\eta|^{H_2+\frac12}} \ .$$
Then using the representation (\ref{representation-sheet}) of the fractional sheet and applying Fubbini theorem - at least formally - allow us to write
$$(\rho_n \ast B)(s,x)= \int_{\R^2} \fouri{W}(d\xi,d\eta) \, (\rho_n \ast L_{(s,x)})(\xi,\eta) \ ,$$
which points out the strong similarity with the above approximation 
$$B^n(s,x):=\int_{\R^2} \fouri{W}(d\xi,d\eta) \, (\1_{D_n} \cdot L_{(s,x)})(\xi,\eta)\ .$$
In the same vein, there is no doubt to us that the subsequent constructions could easily be extended to a wider class of fractional noises, provided one can exhibit appropriate bounds on the Fourier transform of their covariance function (this will indeed be the quantity at the core of the computations regarding the noise part). For the sake of conciseness, we have stuck to the prototype fractional-sheet example though.

\

The paper is organized as follows. In Section \ref{sec:main-results}, we introduce the notion of a $K$-rough path, which corresponds to the central object of our analysis and - hopefully - offers a clear link between the formalisms of RP and RS theories. We shall also state a few basic properties satisfied by $K$-rough paths, together with analytical tools suited to these objects. This will put us in a position to state our main result, namely the existence of a $K$-rough path above the fractional sheet, as well as its important consequences on Equation (\ref{equation-base}). Section \ref{sec:mall-cal} and \ref{sec:proof-main-results} are then devoted to the details of this construction. Section \ref{sec:mall-cal} actually consists in a Malliavin-chaos expansion of the components of the (renormalized) canonical $K$-rough path associated with the smooth approximation $B^n$, while Section \ref{sec:proof-main-results} focuses on the extension of these formulas above the rough process $B$, by means of technical moments controls. Finally, in Appendix \ref{append}, we have collected a few useful (deterministic) estimates related to the interactions between $K$-rough paths and the heat kernel, at first and second orders.

\section{$K$-rough paths and main results}\label{sec:main-results}

The general machinery of RS theory, as introduced in \cite{hai-14}, leans on a combination of a high number of sophisticated objects, gathered under the names of models and regularity structures. However, when specializing the analysis to Equation (\ref{equation-base}) and focusing on the essential information within those models/regularity structures, a relatively simple object naturally arises. We will call a $K$-rough path for its similarity with a classical RP. Just as with classical RP, the definition of a $K$-rough path highly depends on the roughness of the driver - here, the almost sure roughness of $\dot{B}:=\partial_t \partial_x B$ -, seen as a distribution. In order to quantify this roughness, we will use the (local) Besov-type topologies introduced in \cite{hai-14}, as detailled below. 

\smallskip

\textbf{Notations.} We will denote by $\cd'(\R^2)$ the set of general distributions and by $\cd_m'(\R^2)$ ($m \geq 0$) the dual of the space $\cac^m(\R^2)$ of $m$-times differentiable and compactly-supported test-functions. Let us also consider the usual parabolic scaling $\scal:=(2,1)$ and the related balls 
$$\cb_{\scal}(x,R):=\{y=(y_0,y_1)\in \R^2: \, \sqrt{|y_0-x_0|+|y_1-x_1|^2}\leq R\} \ ,$$
for every $x=(x_0,x_1)\in \R^2$ and $R>0$. Given $\vp:\R^2 \to \R$, $x=(x_0,x_1)\in \R^2$ and $\ell \geq 0$, we will denote by $\vp_x^\ell$ the $x$-centered and $2^{-\ell}$-scaled version of $\vp$, that is
$$\vp_x^\ell(y):=2^{3\ell} \vp(2^{2\ell}(y_0-x_0),2^\ell(y_1-x_1)) \ , \ \text{for all} \ y=(y_0,y_1)\in \R^2 \ .$$
Finally, we will denote by $\cb$ the set of smooth functions on $\R^2$ with compact support included in $\cb_{\scal}(x,R)$ and derivatives uniformly bounded by $1$ up to order $4$.

\begin{definition}\label{def:besov-space}
For every $\al <0$ and every set $\compac\subset \R^2$, we say that a distribution $X\in \cd'(\R^2)$ belongs to $\cac^{\al}(\compac)$ if it belongs to $\cd'_{4}(\R^2)$ and if the quantity
$$\lVert X\rVert_{\al;\compac}:=\sup_{x\in \compac,\vp\in \cb,\ell \geq 0}\, 2^{\al \ell}|\langle X,\vp_x^\ell \rangle |$$
is finite. In the sequel, we denote by $\cac^\al_c(\R^2)$ the set of distributions $X\in \cd'(\R^2)$ such that $X\in \cac^\al(\compac)$ for every compact set $\compac$.
\end{definition}

\smallskip

\begin{definition}\label{def:besov-space-2}
For every $\al <0$ and every set $\compac\subset \R^2$, we say that a map $X:\R^2 \to  \cd'(\R^2)$ belongs to $\pmb{\cac}^\al(\compac)$ if for every $x\in \R^2$, $X_x$ belongs to $\cd'_{4}(\R^2)$  and if the quantity
$$\lVert X\rVert_{\al;\compac}:=\sup_{x\in \compac,\vp\in \cb,\ell \geq 0}\, 2^{\al \ell}|\langle X_x,\vp_x^\ell \rangle |$$
is finite. We denote by  $\pmb{\cac}^\al_c(\R^2)$ the set of maps $X:\R^2 \to  \cd'(\R^2)$ such that $X\in \pmb{\cac}^\al(\compac)$ for every compact set $\compac$.
\end{definition}

\smallskip

\begin{proposition}[\cite{deya}]\label{prop:regu-brow-sh}
Let $B$ denote a fractional sheet of Hurst index $(H_1,H_2)$, defined on some complete probability space $(\Omega,\mathfrak{F},\mathbb{P})$, and consider its derivative $\dot{B}:=\partial_t\partial_x B$, understood in the sense of distributions. Then, almost surely, $\dot{B}$ belongs to $\cac_c^\al(\R^2)$, for every $\al<-3+2H_1+H_2$.
\end{proposition}

\smallskip

As we already mentionned it in the introduction, the definition of the canonical $K$-rough above a smooth approximation $B^n$ (and by extension the definition of a $K$-rough path above a rough function $B$) is derived from the space-time expansion of Equation (\ref{equation-base}), that can also be written as
\begin{equation}\label{mild-formul}
Y(s,x)=(G(s,.)\ast \Psi)(x)+\int_0^sdt\int_{\R}dy\,  G(s-t,x-y) F(y,Y(t,y))X(t,y) \ ,
\end{equation}   
where $G$ stands for the usual heat kernel on $\R$ and we have set $X:=\partial_t\partial_x B^n$. Therefore, it is not a surprise that the description of the key elements of the dynamics, which together will form the canonical $K$-rough path, should appeal to the heat kernel. As we are only dealing with local behaviours, it actually suffices to focus on the kernel around its singularity, that is around $0$, which gives birth to the following definition, more suited to the topology under consideration:

\begin{definition}\label{defi:loc-heat-ker}
We call a \emph{localized heat kernel} any function $K:\R^2 \backslash \{(0,0)\}\to \R$ satisfying the following conditions:

\smallskip

\noindent
$(i)$ $K(x)=0$ for every $x=(x_0,x_1) \in \R^2$ such that $x_0\leq 0$.

\smallskip

\noindent
$(ii)$ There exists a smooth function $K_0:\R^2 \to \R$ with support in $[-1,1]^2$ such that for every non-zero $x=(x_0,x_1) \in \R^2$, one has 
\begin{equation}\label{decompo-k}
K(x)=\sum_{\ell\geq 0} 2^{\ell} K_0(2^{2\ell}x_0,2^\ell x_1) \ .
\end{equation}
\end{definition}

Again: a localized heat kernel $K$ is nothing but a local representation, around $0$, of the heat kernel $G$ (see \cite[Lemma 5.5]{hai-14} for more details). The following bound on the Fourier transform (denoted by $\mathcal{F}$) of such a localized heat kernel will be extensively used in the computations of Section \ref{sec:proof-main-results}. Its proof is an immediate consequence of the decomposition (\ref{decompo-k}), as the reader can easily check it:

\begin{lemma}\label{lem:estim-k-hat}
Let $K$ be a localized heat kernel, in the sense of Definition \ref{defi:loc-heat-ker}. Then for all $a,b\in [0,1)$ satisfying $a+b<1$, there exists a constant $c_{K,a,b}$ such that for all non-zero $\xi,\eta\in \R$,
$$\big| \fouri{K}(\xi,\eta)\big| \leq \frac{c_{K,a,b}}{|\xi|^{a} |\eta|^{2b}} \ .$$
\end{lemma}

Finally, space-time expansion of (\ref{mild-formul}) will naturally lead us to consider space-time expansions of the heat kernel (or its localized version). Let us label those quantities for further use: given a localized heat kernel $K$ and for all $x=(x_0,x_1),y=(y_0,y_1),z=(z_0,z_1)\in \R^2$, we set
\begin{equation}\label{defi-k-un}
K^{(1)}_{x,y}(z):=K(y-z)-K(x-z) \ ,
\end{equation}
\begin{equation}\label{defi-k-deux}
K^{(2)}_{x,y}(z):=K(y-z)-K(x-z)-(y_1-x_1) (D^{(0,1)}K)(x-z) \ .
\end{equation}

\subsection{$K$-rough paths}\label{subsec:k-rp}
We are now ready to introduce our central notion of a $K$-rough path above a given deterministic distribution $X\in \cac_c^\al(\R^2)$, at least for $\al<0$ large enough. This definition already appeared in \cite{deya} for $\al\in (-\frac43,0)$, which morally corresponds to an expansion of the equation up to second order - with the consideration of some $K$-Lévy area term. We would like to go one step further here and handle the situation where $\al\in (-\frac32,\frac43]$, which forces us to introduce third-order elements in the analysis.

\smallskip

Thus, from now on and for the rest of Section \ref{subsec:k-rp}, we fix such a coefficient 
$$\boxed{-\frac32<\al \leq -\frac43 \ .}$$
Let us start with the description of the canonical $K$-rough path in this case:

\begin{definition}[Canonical $K$-rough path]\label{defi:canonical}
Given a localized heat kernel $K$ and a continuous function $X:\R^2\to \R$, we call the \emph{canonical} $K$-rough path (of order $\al$) above $X$ the $4$-uplet
$$\mathbf{X}=(\mathbf{X}^{\mathbf{1}},\mathbf{X}^{\mathbf{2}},\mathbf{X}^{\mathbf{3.1}},\mathbf{X}^{\mathbf{3.2}})$$
defined along the following iterative formulas: for every $x=(x_0,x_1),y=(y_0,y_1)\in \R^2$,
\begin{equation}\label{defi-can-rp-1}
\mathbf{X}^{\mathbf{1}}:=X  \quad , \quad \mathbf{X}^{\mathbf{2}}_x(y):= \langle \mathbf{X}^{\mathbf{1}}, K^{(1)}_{x,y} \rangle \, X(y) \quad ,
\end{equation}
\begin{equation}\label{defi-can-rp-2}
\mathbf{X}^{\mathbf{3.1}}_x(y):=\langle \mathbf{X}^{\mathbf{2}}_x, K^{(2)}_{x,y} \rangle\, X(y) \quad , \quad \mathbf{X}^{\mathbf{3.2}}_x(y):=\langle \mathbf{X}^{\mathbf{1}}, K^{(1)}_{x,y} \rangle^2\, X(y)  \ .
\end{equation}
\end{definition}
Let us insist one more time on the fact that this definition is motivated by the space-time Taylor expansion of the standard PDE (\ref{mild-formul}). Then, just as in RP theory, it turns out that this notion of a canonical $K$-rough path - and the whole integration machinery built upon it - can be lifted at some more abstract level above a rougher distribution $X$, which gives rise to the following general definition:

\begin{definition}[$K$-rough path]\label{defi:k-rough-path}
Given a localized heat kernel $K$ and a distribution $X\in \cac_c^\al(\R^2)$, we call a $K$-rough path (of order $\al$) above $X$ any $4$-uplet 
$$\mathbf{X}=(\mathbf{X}^{\mathbf{1}},\mathbf{X}^{\mathbf{2}},\mathbf{X}^{\mathbf{3.1}},\mathbf{X}^{\mathbf{3.2}})\in \cac_c^\al(\R^2)\times  \pmb{\cac}_c^{2\al+2}(\R^2) \times \pmb{\cac}_c^{3\al+4}(\R^2) \times \pmb{\cac}_c^{3\al+4}(\R^2)$$
such that $\mathbf{X}^{\mathbf{1}} =X$ and the following \enquote{$K$-Chen} relations hold: for every $x=(x_0,x_1),y=(y_0,y_1)\in \R^2$,
\begin{equation}\label{k-chen-relation}
\mathbf{X}_{x}^{\mathbf{2}}-\mathbf{X}_{y}^{\mathbf{2}}= \mathbf{X}^{\mathbf{1}}\big( K^{(1)}_{x,y}\big) \cdot \mathbf{X}^{\mathbf{1}} \quad , \quad\mathbf{X}_{x}^{\mathbf{3.2}}-\mathbf{X}_{y}^{\mathbf{3.2}}= \mathbf{X}^{\mathbf{1}}\big( K^{(1)}_{x,y}\big) \cdot \big\{ \mathbf{X}_{x}^{\mathbf{2}}+\mathbf{X}_{y}^{\mathbf{2}} \big\}
\end{equation}
and
\begin{multline}\label{k-chen-relation-third-order}
\mathbf{X}_{x}^{\mathbf{3.1}}-\mathbf{X}_{y}^{\mathbf{3.1}}=\\ \mathbf{X}^{\mathbf{2}}_{x}\big( K^{(2)}_{x,y}\big) \cdot \mathbf{X}^{\mathbf{1}}+\mathbf{X}^{\mathbf{1}}\big( K^{(1)}_{x,y}\big) \cdot \mathbf{X}^{\mathbf{2}}_y+\big\{\mathbf{X}^{\mathbf{2}}_{x}\big( (D^{(0,1)}K)(x-.)\big)- \mathbf{X}^{\mathbf{2}}_{y}\big( (D^{(0,1)}K)(y-.)\big) \big\} \cdot (y_1-.) \cdot \mathbf{X}^{\mathbf{1}} \ .
\end{multline}
For a fixed localized heat kernel $K$ and given two $K$-rough paths $\mathbf{X},\mathbf{Y}$ (above possibly different distributions $X,Y$), we denote, for every compact set $\compac \subset \R^2$,
\begin{equation}\label{norm-rp}
 \lVert \mathbf{X};\mathbf{Y}\rVert_{\al;\compac}:= \lVert \mathbf{X}^{\mathbf{1}}-\mathbf{Y}^{\mathbf{1}}\rVert_{\al;\compac}+\lVert \mathbf{X}^{\mathbf{2}}-\mathbf{Y}^{\mathbf{2}}\rVert_{2\al+2;\compac}+\lVert \mathbf{X}^{\mathbf{3.1}}-\mathbf{Y}^{\mathbf{3.1}}\rVert_{3\al+4;\compac}+\lVert \mathbf{X}^{\mathbf{3.2}}-\mathbf{Y}^{\mathbf{3.2}}\rVert_{3\al+4;\compac}\ ,
\end{equation}
and $\lVert \mathbf{X}\rVert_{\al;\compac}:=\lVert \mathbf{X};0\rVert_{\al;\compac}$. In the sequel, we will denote by $\mathcal{E}_{K,\al}$ the set of $K$-rough paths of order $\al$.
\end{definition}

\smallskip

\begin{remark}
As its name suggests it, the canonical $K$-rough path above a continuous function $X$ is a particular example of $K$-rough path. Relations (\ref{k-chen-relation}) and (\ref{k-chen-relation-third-order}) can indeed be easily checked from the explicit formulas in (\ref{defi-can-rp-1})-(\ref{defi-can-rp-2}). 
\end{remark}

\begin{remark}
With expansion (\ref{decompo-k}) in mind and for every $n\geq 0$, let us denote by $K^n$ the smooth compactly-supported function obtained as the finite sum $K^n(s,x):=\sum_{0\leq \ell\leq n} 2^{3\ell} K_0(2^{2\ell}s,2^\ell x)$. Then, in (\ref{k-chen-relation}) and (\ref{k-chen-relation-third-order}), and although $K$ itself is not a smooth test-function (due to its singularity at $0$), the coefficient $\mathbf{X}^{\mathbf{i}}_{x}\big( K^{(i)}_{x,y}\big)$ ($i\in \{1,2\}$), resp. $\mathbf{X}^{\mathbf{2}}_{x}\big( (D^{(0,1)}K)(x-.)\big)$, is well defined as the limit of the sequence $\mathbf{X}^{\mathbf{i}}_{x}\big( K^{(i),n}_{x,y}\big)$ ($i\in \{1,2\}$), resp. $\mathbf{X}^{\mathbf{2}}_{x}\big( (D^{(0,1)}K^n)(x-.)\big)$, where $K^{(i),n}$ is derived from $K^{(i)}$ by replacing each occurence of $K$ with $K^n$. The proof of this assertion can be easily deduced from the properties in Appendix \ref{append}, which also contains a few regularity results related to these quantities. 
\end{remark}

\begin{remark}
Relations (\ref{k-chen-relation}) and (\ref{k-chen-relation-third-order}) can legitimately be considered as a parabolic analog of the classical Chen's relations for a third-order rough path $\mathbf{x}=(\mathbf{x}^{\mathbf{1}}=\delta x,\mathbf{x}^{\mathbf{2}},\mathbf{x}^{\mathbf{3}})$ (see \cite{lyons-book} for a detailled definition). To emphasize this analogy, introduce the dual path $\tilde{\mathbf{x}}=(\tilde{\mathbf{x}}^{\mathbf{1}},\tilde{\mathbf{x}}^{\mathbf{2}},\tilde{\mathbf{x}}^{\mathbf{3}})$ defined for every test-function $\vp$ as $\tilde{\mathbf{x}}^{\mathbf{i}}_t(\vp):=\int \vp(s) \, d_s(\mathbf{x}^{\mathbf{i}}_{s,t})$. With these notations, the classical Chen's relations can also be written as
$$\tilde{\mathbf{x}}^{\mathbf{2}}_s-\tilde{\mathbf{x}}^{\mathbf{2}}_t=\tilde{\mathbf{x}}^{\mathbf{1}}(\1_{[s,t]})\, \tilde{\mathbf{x}}^{\mathbf{1}} \quad , \quad \tilde{\mathbf{x}}^{\mathbf{3}}_s-\tilde{\mathbf{x}}^{\mathbf{3}}_t=\tilde{\mathbf{x}}^{\mathbf{2}}_s(\1_{[s,t]})\, \tilde{\mathbf{x}}^{\mathbf{1}}+\tilde{\mathbf{x}}^{\mathbf{1}}(\1_{[s,t]})\, \tilde{\mathbf{x}}^{\mathbf{2}}_t\ ,$$
which makes the similarity with (\ref{k-chen-relation})-(\ref{k-chen-relation-third-order}) obvious. Obseerve however that the higher complexity of this two-parameter setting forces us to consider more sophisticated structures exhibiting two third-order components (instead of one for RP theory). Also, the implicit presence of the (regularizing) heat kernel in the definition of a $K$-rough path echoes in a natural way on the choice of the roughness assumptions, that is on the choice of the successive combinations $\al$, $2\al+2$, $3\al+4$, as Formulas (\ref{defi-can-rp-1})-(\ref{defi-can-rp-2}) and the regularity properties of Appendix \ref{append} should convince the reader.
\end{remark}

\begin{remark}\label{rk:d-al}
The set $\mathcal{E}_{K,\al}$ can easily be turned into a complete metric space by considering the distance
$$d_\al(\mathbf{X},\mathbf{Y}):=\sum_{k\geq 0} 2^{-k} \frac{\Vert \mathbf{X};\mathbf{Y}\rVert_{\al;R_k}}{1+\lVert \mathbf{X};\mathbf{Y}\rVert_{\al;R_k}} \ ,$$
where $\lVert \mathbf{X};\mathbf{Y}\rVert_{\al;\compac}$ is defined by (\ref{norm-rp}) and we have set $R_k:=[-k,k]^2$.
The completeness of $(\mathcal{E}_{K,\al},d_\al)$ can indeed be shown along the arguments of the proof of \cite[Proposition 3.1]{deya}.
\end{remark}

\smallskip

The following basic property, the proof of which is immediate, illustrates the flexibility of the definition of a $K$-rough path. It shows in particular that, provided it exists, a $K$-rough path above a given distribution $X\in \cac_c^\al(\R^2)$ is not unique at all. Here again, the idea is strongly reminiscent of the well-known non-uniqueness property of classical rough paths. In the sequel, we shall rely on this \enquote{renormalization} trick to overcome the diverging issue raised by the canonical $K$-rough path above the approximation $B^n$.  

\begin{lemma}[Renormalization]\label{lem:transfo-m}
Fix a localized heat kernel $K$ and a path $X\in \cac_c^\al(\R^2)$. Given a $K$-rough path $\mathbf{X}$ above $X$ and a constant $c\in \R$, consider the path 
$$\widehat{\mathbf{X}}=\text{Renorm}\,(\mathbf{X},c):=(\widehat{\mathbf{X}}^{\mathbf{1}},\widehat{\mathbf{X}}^{\mathbf{2}},\widehat{\mathbf{X}}^{\mathbf{3.1}},\widehat{\mathbf{X}}^{\mathbf{3.2}})$$
defined along the following formulas:
$$\widehat{\mathbf{X}}^{\mathbf{1}}:=\mathbf{X}^{\mathbf{1}} \quad , \quad \widehat{\mathbf{X}}^{\mathbf{2}}_x:=\mathbf{X}^{\mathbf{2}}_x-c \, \mathbf{Id} \ ,$$
$$\widehat{\mathbf{X}}^{\mathbf{3.1}}_x:=\mathbf{X}^{\mathbf{3.1}}_x-c \, \mathbf{X}^{\mathbf{1}}\big( K^{(1)}_{x,.}\big)   \quad , \quad \widehat{\mathbf{X}}^{\mathbf{3.2}}_x:=\mathbf{X}^{\mathbf{3.2}}_x-2c \, \mathbf{X}^{\mathbf{1}}\big( K^{(1)}_{x,.}\big) \ .$$
Then $\widehat{\mathbf{X}}$ is a $K$-rough path above $X$ as well.
\end{lemma}

\smallskip

In RP theory, the so-called Garsia-Rodemich-Rumsey Lemma and its extensions (see e.g. \cite[Section 6]{gubi-controlling}) provide a very efficient tool to study the roughness of the processes under consideration, and therefore represent one of the keys of the RP analysis. Using sophisticated wavelets arguments, M. Hairer succeeded in the exhibition of similar tools in the multiparameter setting. The following statement offers a possible simple way to account for these results. In particular, it should be clear to the reader that the central condition in this statement is directly related to (not to say that it perfectly fits) the structure of the $K$-Chen relations from Definition \ref{defi:k-rough-path} (compare (\ref{condit-grr}) with (\ref{k-chen-relation})-(\ref{k-chen-relation-third-order})). Let us also recall that the notation $\cb$ has been introduced at the beginning of this section.

\begin{lemma}[Multiparameter G-R-R Lemma] \label{GRR-gene}
Fix $\la\in [0,-\al)$. Let $X: \R^2 \to \cd'_2(\R^2)$ be a map with increments of the form
\begin{equation}\label{condit-grr}
X_x-X_y=\sum_{i=1,\ldots,r} \theta^i(x,y) \cdot X^{\sharp,i}_y
\end{equation}
where $X^{\sharp,i}\in \pmb{\cac}^{\al_i}_c(\R^2)$ ($\al_i\in [\al,\al+\la] $) and $\theta^i:\R^2\times \R^2\to \R$ is such that for every compact set $\compac \subset \R^2$ and every $x,y\in \compac$, one has
$$|\theta^i(x,y)| \leq C_{\theta^i;\compac} \, \lVert x-y\rVert_\scal^{\al+\la-\al_i} \ ,$$
for some constant $C_{\theta^i;\compac} \geq 0$. Then there exists a finite set $\cb_0 \subset \cb$ such that
\begin{equation}\label{bound-GRR-gene}
\lVert X\rVert_{\al+\la;R_k}\lesssim \sup_{\psi\in \cb_0} \sup_{\ell\geq 0} \sup_{x\in \Lambda_\scal^\ell \cap R_{k+1}} 2^{\ell(\al+\la)} |\langle X_x, \psi_x^\ell \rangle |+\sum_{i=1,\ldots,r} C_{\theta^i;R_{k+1}}\lVert X^{\sharp,i}\rVert_{\al_i;R_{k+1}} \ ,
\end{equation}
where we have set $\Lambda_\scal^\ell:=\{(2^{-2\ell} k_1,2^{-\ell}k_2), \ k_1,k_2\in \Z\}$ and $R_k:=[-k,k]^2$.
\end{lemma}

\begin{proof}
It is a straightforward generalization of the arguments in the proof of \cite[Lemma 3.2]{deya}.
\end{proof}

\smallskip

As we evoke it from the beginning, what makes this $K$-rough-path structure so interesting (beyond its clear analogy with the classical rough-path structure) is the fact that it can be readily injected into the machinery of \cite{hai-14} so as to deduce numerous striking results about the equation driven by $X$, that is
\begin{equation}\label{equation-base-x}
(\partial_t Y)(t,x)=(\partial^2_x Y)(t,x)+F(x,Y(t,x)) \, X(t,x) \quad , \quad Y(0,x)=\Psi(x) \ .
\end{equation}
The following general statement sums up these important consequences. At this point, let us recall that we have fixed a parameter $\al\in (-\frac32,-\frac43]$ for the whole Section \ref{subsec:k-rp}, and that the map $\text{Renorm}$ has been defined through Lemma \ref{lem:transfo-m}, while the distance $d_\al$ has been introduced in Remark \ref{rk:d-al}.

\begin{proposition}[Solution map]\label{prop:link-model}
Fix an arbitrary time horizon $T>0$ as well as a vector field $F\in \cac^\infty(\R^2;\R)$ such that $\lVert D^{(k_1,k_2)}F\rVert_{L^\infty(\R^2)} < \infty$ for all $k_1,k_2\geq 0$ and $F(x,y)=0$ for all $(x,y)\in (\R \backslash \compac_F) \times \R$,
for some compact set $\compac_F\subset \R$. Then there exists a localized heat kernel $K$ and a \enquote{solution} map
$$\Phi^K_{T,F} : \mathcal{E}_{K,\al} \times L^\infty(\R) \longrightarrow\ (0,T] \times L^\infty([0,T] \times \R)$$
such that the following properties are satisfied:

\smallskip

\noindent
$(i)$ If $X:\R^2 \to \R$ is a continuous function and $\Psi\in L^\infty(\R)$, then denoting by $\mathbf{X}$ the canonical $K$-rough path above $X$, one has $\Phi^K_{T,F}(\mathbf{X},\Psi)=(T,Y)$, where $Y$ is the classical solution on $[0,T]$ of Equation (\ref{equation-base-x}).


\noindent
$(ii)$ If $X:\R^2 \to \R$ is a continuous function, $\Psi\in L^\infty(\R)$ and $c\in \R$, then denoting by $\mathbf{X}$ the canonical $K$-rough path above $X$, one has $\Phi^K_{T,F}(\text{Renorm}(\mathbf{X},c),\Psi)=(T,\widehat{Y})$, where $\widehat{Y}$ is the classical solution on $[0,T]$ of the equation 
$$
\left\{\begin{array}{ccl}
\partial_t \widehat{Y} (t,x) &= &\partial^2_x \widehat{Y}(t,x)+F(x,\widehat{Y}(t,x)) \,X(t,x)-c\, F(x,\widehat{Y}(t,x))\, \partial_2F(x,\widehat{Y}(t,x))\ ,\\
Y(0,x) &= &\Psi(x) \ .
\end{array}\right.
$$

\smallskip

\noindent
$(iii)$ Let $(\mathbf{X},\Psi)\in \mathcal{E}_{K,\al} \times L^\infty(\R)$ and $(\mathbf{X}^n,\Psi^n)$ be a sequence in $\mathcal{E}_{K,\al} \times L^\infty(\R)$ such that
$$d_\al\big(\mathbf{X}^n,\mathbf{X}\big) \to 0 \quad , \quad \lVert \Psi^n-\Psi\rVert_{L^\infty(\R)} \to 0 \quad \text{and} \quad \Phi^K_{T,F}(\mathbf{X}^n)=(T^n,Y^n) \ ,$$
for some sequence $(T^n,Y^n) \in  (0,T] \times L^\infty([0,T] \times \R)$. Then $Y^n \to Y$ in $L^\infty([0,T_1]\times \R)$ for every $T_1<(T_0 \wedge \inf_{n} T^n)$, where $(T_0,Y):=\Phi^K_{T,F}(\mathbf{X},\Psi)$.

\smallskip

\noindent
$(iv)$ If $\Phi^K_{T,F}(\mathbf{X})=(T_0,Y)$ with $T_0<T$, then one has $\lim\limits_{\substack{t \to T_0 \\ t<T_0}}\, \lVert Y_t\rVert_{L^\infty(\R)} =\infty$ and $Y_t=0$ for $t\geq T_0$.

\smallskip

Based on the above properties and setting $(T_0,Y):=\Phi^K_{T,F}(\mathbf{X},\Psi)$, we call $Y$ the (maximal) solution on $[0,T_0]$, in the sense of the $K$-rough paths, of the equation
$$
\partial_t Y=\partial^2_x Y+F(.,Y)\, \mathbf{X}\quad, \quad Y(0,x) = \Psi(x) \ .
$$
\end{proposition}

\begin{proof}
The four properties $(i)$-$(iv)$ are derived from a careful examination of the analysis carried out in \cite{hai-14} (see also \cite{deya} for a detailled version of this analysis in the special case given by the dynamics of Equation (\ref{equation-base})). The only ingredient to specify here is how the $K$-rough-path structure can be related to the regularity structure/model terminology used in \cite{hai-14}. A part of this connection has already been exhibited in \cite[Section 2.1]{deya} for $\al \in (-\frac43,-1]$. Taking up the notations $(A,\mathscr{T})$ of the latter reference and following the general procedure described in \cite[Section 8]{hai-14}, the regularity structure $(\bar{A},\bar{\mathscr{T}})$ to be considered in our situation is given by
$$\bar{A}:=A \cup \{3\al+4,2\al+4\} \quad , \quad \bar{\mathscr{T}}:=\mathscr{T} \oplus \mathscr{T}_{3\al+4} \oplus \mathscr{T}_{2\al+4}$$
$$\mathscr{T}_{2\al+4}:=\text{Span}\{\ci(\Xi \ci(\Xi))\} \quad , \quad \mathscr{T}_{3\al+4}:=\text{Span}\{\Xi\ci(\Xi \ci(\Xi)),\Xi\ci(\Xi)^2\} \ .$$
Then, given a $K$-rough path $\mathbf{X}$ above a distribution $X\in \cac^\al_c(\R^2)$, we define the two maps
$$\gga^{\mathbf{X}}: \R^2 \times \R^2 \to \cl(\bar{\mathscr{T}}) \quad \text{and} \quad \Pi^{\mathbf{X}}: \R^2 \to \cl(\bar{\mathscr{T}},\cs'(\R^2)) $$
as follows. For any $\beta\in A$ and $\tau\in \mathscr{T}_\be$, we define both $\gga^{\mathbf{X}}_{xy}(\tau)$ and $\Pi^{\mathbf{X}}_x(\tau)$ ($x,y\in \R^2$) just as in \cite[Section 2.1]{deya}. Besides, we set, for all $x=(x_0,x_1),y=(y_0,y_1)\in \R^2$,
$$\Pi^{\mathbf{X}}_x(\Xi \ci(\Xi \ci(\Xi))):=\mathbf{X}^{\mathbf{3.1}}_x \quad , \quad \Pi^{\mathbf{X}}_x(\Xi \ci(\Xi)^2):=\mathbf{X}^{\mathbf{3.2}}_x \quad , \quad \Pi^{\mathbf{X}}_x(\ci(\Xi \ci(\Xi)))(y):=\mathbf{X}^{\mathbf{2}}_x(K^{(2)}_{x,y}) \ ,$$
and
\begin{equation}\label{new-gga}
\gga^{\mathbf{X}}_{xy}(\ci(\Xi \ci(\Xi))):=\ci(\Xi \ci(\Xi))-a_{xy} \, Y_1-b_{xy} \, \ci(\Xi)-c_{xy} \, \1 \ ,
\end{equation}
with
$$a_{xy}:=\mathbf{X}^{\mathbf{2}}_y((D^{(0,1)}K)(y-.))-\mathbf{X}^{\mathbf{2}}_x((D^{(0,1)}K)(x-.)) \quad , \quad b_{xy}:=\mathbf{X}^{\mathbf{1}}(K^{(1)}_{xy})$$
$$c_{xy}:=\mathbf{X}^{\mathbf{2}}_x(K^{(2)}_{x,y})-\mathbf{X}^{\mathbf{1}}(K^{(1)}_{x,y})\, b_{xy}-(y_1-x_1) \, a_{xy} \ .$$
In (\ref{new-gga}), the notation $Y_1$ refers to the abstract symbol for the first-order monomial in the $x_1$-variable (for instance, $\Pi^{\mathbf{X}}_x(Y_1)(y)=y_1-x_1$). Finally, set 
$$\gga^{\mathbf{X}}_{xy}(\Xi \ci(\Xi \ci(\Xi))):=\gga^{\mathbf{X}}_{xy}(\Xi)\star \gga^{\mathbf{X}}_{xy}(\Xi \ci(\Xi \ci(\Xi))) \quad \text{and} \quad \gga^{\mathbf{X}}_{xy}(\Xi \ci(\Xi)^2):=\gga^{\mathbf{X}}_{xy}(\Xi)\star \gga^{\mathbf{X}}_{xy}(\ci(\Xi))\star \gga^{\mathbf{X}}_{xy}(\ci(\Xi)) \ ,$$
where the notation $\star$ refers to the product of the regularity structure (see \cite[Section 2.1]{deya}). 

\smallskip

Using the properties contained in the very definition of a $K$-rough path, one can now check that for all $x,y\in \R^2$, the identity $\Pi^{\mathbf{X}}_y=\Pi^{\mathbf{X}}_x \circ \gga^{\mathbf{X}}_{xy}$ holds true, while for every compact set $\compac\subset \R^2$, it holds, with the notations of \cite[Section 2.3]{hai-14}, 
$$\lVert (\Pi^{\mathbf{X}},\gga^{\mathbf{X}})\rVert_{\al;\compac}\lesssim \lVert \mathbf{X}\rVert_{\al;\bar{\compac}} \quad , \quad \lVert (\Pi^{\mathbf{X}},\gga^{\mathbf{X}});(\Pi^{\mathbf{Y}},\gga^{\mathbf{Y}})\rVert_{\al;\compac}\lesssim \lVert \mathbf{X};\mathbf{Y}\rVert_{\al;\bar{\compac}} \ ,$$
for some compact set $\bar{\compac} \supset \compac$. 

\smallskip

This leads us to the desired conclusion: the above-described pair $(\Pi^{\mathbf{X}},\gga^{\mathbf{X}})$ does define a \emph{model} for the regularity structure $\bar{\struc}$, in the sense of \cite[Definition 2.17]{hai-14}. The construction has actually been designed in such a way that if $X$ is a smooth path and $\mathbf{X}$ stands for the canonical $(\al,K)$-rough path above $X$ (along Definition \ref{defi:canonical}), then the associated model $(\Pi^{\mathbf{X}},\gga^{\mathbf{X}})$ coincides with the \emph{canonical model} described in \cite[Section 8.2]{hai-14}.

\smallskip

Based on these observations, the properties $(i)$-$(iv)$ of our statement are now mere consequences of general results from \cite{hai-14} (see also \cite[Section 2]{deya} for a shorter presentation applying specifically to Equation (\ref{equation-base-x})). For the sake of conciseness, we will not return to the details of this sophisticated analysis.  

\end{proof}

\subsection{Main results}\label{subsec:main-results}
Now endowed with the above preliminary - deterministic - material, we can turn to the detailled presentation of our main result, namely the exhibition of a $K$-rough path above the stochastic fractional sheet $B$. This construction will be based on a chaos decomposition of the canonical $K$-rough path $B^n$, a strategy that will naturally lead us to the consideration of trace-type terms. The procedure will more specifically involve a series of operators $L,\Theta^{(i)},L^{i,X}$ that we propose to introduce right now. Let us recall that, throughout the paper, we denote by $\mathcal{F}$ the usual Fourier transformation, that is
$$\fouri{\vp}(\xi,\eta):=\int_{\R^2} dx_0 dx_1 \, e^{-\imath \xi x_0} e^{-\imath \eta x_1} \vp(x_0,x_1) \ .$$

\smallskip

First, for every $(H_1,H_2)$ such that $2H_1+H_2 >\frac32$, every $x=(x_0,x_1)\in \R^2$ and every smooth compactly-supported function $\psi$ on $\R^2$, we set
\begin{equation}
L_{x}(\psi):=-c_{H_1,H_2}^2 \iint_{\R^2}d\xi d\eta \, \frac{e^{\imath \xi x_0} e^{\imath \eta x_1}}{|\xi|^{2H_1-1} |\eta|^{2H_2-1}}\mathcal{F}(K)(\xi,\eta)\, \mathcal{F}(\psi)(\xi,\eta) \ .
\end{equation}
where $c_{H_1,H_2}$ refers to the constant in the representation (\ref{representation-sheet}) of the fractional sheet. Then, for every $X\in \cac^\al_c(\R^2)$ (with $\al\in (-\frac32,-\frac43]$) and every $x=(x_0,x_1),y=(y_0,y_1)\in \R^2$, we define successively
$$\Theta^{(1)}_{x}(\psi,X)(y):=\psi(y) \, X\big( K^{(1)}_{x,y}\big) \quad , \quad \Theta^{(2)}_{x}(\psi,X)(y):=\psi(y) \, X\big( K^{(1)}_{x,y}\big)\, (y_1-x_1)$$
$$\Theta^{(3)}_{x}(\psi,X)(y):=\iint_{\R^2} dz \, \psi(y+z) \, X\big(K^{(1)}_{y+z,z}\big) \, K^{(2)}_{x,y+z}(z) \ ,$$
and finally
\begin{equation*}
L^{1,X}_{x}(\psi):= -c_{H_1,H_2}^2 \iint_{\R^2}d\xi d\eta\, \frac{e^{\imath \xi x_0} e^{\imath \eta x_1}}{|\xi|^{2H_1-1} |\eta|^{2H_2-1}}\mathcal{F}(K)(\xi,\eta)\,\mathcal{F}\big(\Theta^{(1)}_{x}(\psi,X)\big)(\xi,\eta)  \ , 
\end{equation*}
\begin{equation*}
L^{2,X}_{x}(\psi):=- c_{H_1,H_2}^2 \iint_{\R^2}d\xi d\eta\, \frac{e^{\imath \xi x_0} e^{\imath \eta x_1}}{|\xi|^{2H_1-1} |\eta|^{2H_2-1}}\mathcal{F}(D^{(0,1)}K)(\xi,\eta)\,\mathcal{F}\big(\Theta^{(2)}_{x}(\psi,X)\big)(\xi,\eta)  \ , 
\end{equation*}
\begin{equation*}
L^{3,X}_{x}(\psi):=c_{H_1,H_2}^2 \iint_{\R^2}\frac{d\xi d\eta}{|\xi|^{2H_1-1} |\eta|^{2H_2-1}}\mathcal{F}\big(\Theta^{(3)}_{x}(\psi,X)\big)(\xi,\eta)  \ . 
\end{equation*}

\smallskip

\begin{remark}
Checking the well-posedness of the above operators is one of the objectives behind the computations of Section \ref{sec:proof-main-results} and Appendix \ref{append}. Therefore, we refer the reader to these sections for a detailled analysis of $L$, $\Theta^{(i)}$ and $L^{i,X}$. In fact, the treatment of these trace operators (which are specific to the fractional situation, as emphasized by Proposition \ref{prop:ito}) will prove to be one of the most technical parts of our study.  
\end{remark}

\smallskip

\begin{theorem}[Existence of a fractional $K$-rough path]\label{theo:main}
 Fix $(H_1,H_2) \in (0,1)^2$ and $\al<0$ such that 
$$\frac32 <2H_1+H_2 \leq 2 \quad , \quad -\frac32<\al<\min(-\frac43,-3+2H_1+H_2) \ ,$$
and let $B$ be a $(H_1,H_2)$-fractional sheet defined on some complete probability space $(\Omega,\mathfrak{F},\mathbb{P})$, with representation (\ref{representation-sheet}). Let $B^n$ be the approximation of $B$ given by (\ref{approx-noise}) and, for a fixed localized heat kernel $K$, denote by $\mathbf{B}^n$ the canonical $K$-rough path (of order $\al$) above $\dot{B}^n:=\partial_t\partial_x B^n$, in the sense of Definition \ref{defi:canonical}. Then there exists a $K$-rough path $\widehat{\mathbf{B}}$ above $\dot{B}:=\partial_t\partial_x B$ such that if we set $\widehat{\mathbf{B}}^n:=\text{Renorm}\, (\mathbf{B}^n,c^n_{H_1,H_2})$ with
\begin{equation}\label{defi-seq-reno}
c^n_{H_1,H_2}:=c_{H_1,H_2}^2\iint_{\cd_n}\frac{d\xi \, d\eta  }{|\xi|^{2H_1-1}|\eta|^{2H_2-1}}\fouri{K}(\xi,\eta) \ ,
\end{equation}
one has almost surely
$$d_\al \big(\widehat{\mathbf{B}}^n,\widehat{\mathbf{B}} \big) \stackrel{n\to\infty}{\longrightarrow} 0 \quad .$$
Moreover, the following identification formulas hold true for the components of $\widehat{\mathbf{B}}$: for every $x\in \R^2$ and every smooth compactly-supported function $\psi$, one has almost surely
\begin{equation}\label{def:hat-b-order-1}
\widehat{\mathbf{B}}^{\mathbf{1}}(\psi)=\dot{B}(\psi) \ ,
\end{equation}
\begin{equation}\label{def:hat-b-order-2}
\widehat{\mathbf{B}}^{\mathbf{2}}_{x}(\psi)
=\delta^{B}\big( \psi \cdot \widehat{\mathbf{B}}^{\mathbf{1}}\big(K^{(1)}_{x,.}\big) \big)+L_{x}(\psi) \ ,
\end{equation}
\begin{equation}\label{def:hat-b-order-3}
\widehat{\mathbf{B}}^{\mathbf{3.1}}_{x}(\psi)
=\delta^{B}\big( \psi \cdot \widehat{\mathbf{B}}^{\mathbf{2}}_{x}\big(K^{(2)}_{x,.}\big) \big)+\big[ L^{1,\dot{B}}_{x}(\psi)+L^{2,\dot{B}}_{x}(\psi)+L^{3,\dot{B}}_{x}(\psi)\big] \ ,
\end{equation}
\begin{equation}\label{def:hat-b-order-3-1}
\widehat{\mathbf{B}}^{\mathbf{3.2}}_{x}(\psi)
=\delta^{B}\big( \psi \cdot \big(\widehat{\mathbf{B}}^{\mathbf{1}}\big(K^{(1)}_{x,.}\big)\big)^2 \big)+2\, L^{1,\dot{B}}_{x}(\psi) \ ,
\end{equation}
where the notation $\delta^B$ refers to the Skorohod integral with respect to $B$ (see Section \ref{sec:mall-cal}).
\end{theorem}

\smallskip

\begin{remark}
Using the Malliavin-calculus terminology, decompositions (\ref{def:hat-b-order-2})-(\ref{def:hat-b-order-3-1}) are of course nothing but the expansions of the components of $\widehat{\mathbf{B}}_{x}(\psi)$ into the chaoses associated with the fractional noise $\dot{B}$. For instance, in (\ref{def:hat-b-order-3}), $\delta^{B}( \psi \cdot \widehat{\mathbf{B}}^{\mathbf{2}}_{x}(K^{(2)}_{x,.}) )$ stands for the component in the third-order chaos, while $L^{1,\dot{B}}_{x}(\psi)$, $L^{2,\dot{B}}_{x}(\psi)$ and $L^{3,\dot{B}}_{x}(\psi)$ all belong to the first-order chaos.
\end{remark}

\begin{remark}
By following the arguments of the proof of \cite[Proposition 3.12]{deya}, it is easy to show that the quantity $c^n_{H_1,H_2}$ defined by (\ref{defi-seq-reno}) asymptotically behaves as
\begin{equation}\label{estim-cstt}
c^n_{H_1,H_2}\stackrel{n\to \infty}{\sim} 
\left\lbrace
\begin{array}{ll}
d^1_{H_1,H_2}\cdot  2^{2n(2-2H_1-H_2)}& \quad \text{if} \quad \frac32 < 2H_1+H_2 <2 \ ,\\
d^2_{H_1,H_2}\cdot  n & \quad \text{if} \quad 2H_1+H_2 =2 \ ,
\end{array}
\right.
\end{equation}
for some constants $d^1_{H_1,H_2},d^2_{H_1,H_2}$, which points out the actual divergence of the canonical $K$-rough path $\mathbf{B}^n$ as soon as $2H_1+H_2\leq 2$.
\end{remark}

\smallskip

In the very particular situation of a white-in-time noise, that is when $H_1=\frac12$ - and the whole Itô integration theory becomes available -, the symmetry in representation (\ref{representation-sheet}), combined with the vanishing properties of any (localized) heat kernel, offers a drastic simplification of both the analysis and the results:

\begin{proposition}[White-in-time noise]\label{prop:ito}
In the setting of Theorem \ref{theo:main}, assume that $H_1=\frac12$ and $H_2>\frac12$, that is $B$ is a white-in-time noise with fractional spatial regularity of order $H_2>\frac12$. Then, with the above notations, one has, for every test-function $\psi$ with support included in the set $\{x\in \R^2: \ x_0\geq 0\}$ and every $x\in \R^2$,
$$ L_{x}(\psi_{x})=L^{1,\dot{B}}_{x}(\psi_{x})=L^{2,\dot{B}}_{x}(\psi_{x})=L^{3,\dot{B}}_{x}(\psi_{x})=0 \ ,$$
where we have set $\psi_x(y):=\psi(y-x)$. Accordingly, in this case, and when applied to such test-functions, the $K$-rough path $\widehat{\mathbf{B}}$ exhibited in Theorem \ref{theo:main} reduces (almost surely) to the Itô $K$-rough path above $\dot{B}$, that is to the $K$-rough path extending Formulas (\ref{defi-can-rp-1})-(\ref{defi-can-rp-2}) by means of Itô integrals.
\end{proposition}

\begin{proof}
Since $H_1=\frac12$, we can lean on the basic isometry properties of the Fourier transform to assert that  
\begin{eqnarray*}
L_x(\psi_x)&=&-c_{H_1,H_2}^2\int_{\R}\frac{d\eta}{|\eta|^{2H_2-1}} e^{\imath \eta x_1}\int_{\R^2} dy_1dz_1 \, e^{-\imath \eta (y_1+z_1)}\int_{\R} d\xi \, e^{\imath \xi x_0} \, \mathcal{F}(K(.,y_1))(\xi) \mathcal{F}(\psi_x(.,z_1))(\xi)\\
&=&c\int_{\R}\frac{d\eta}{|\eta|^{2H_2-1}} e^{\imath \eta x_1}\int_{\R^2} dy_1dz_1 \, e^{-\imath \eta (y_1+z_1)}\int_{\R} dy_0 \, K(x_0-y_0,y_1) \psi(y_0-x_0,z_1-x_1) \ .
\end{eqnarray*}
It now suffices to observe that due to the vanishing assumption on $K$ and the support condition on $\psi$, the latter integral is necessarily equal to zero. The same combination of arguments (isometry property and support condition) can also be used to show that $L^{1,\dot{B}}_{x}(\psi_{x})=L^{2,\dot{B}}_{x}(\psi_{x})=L^{3,\dot{B}}_{x}(\psi_{x})=0$.

\smallskip

The identification with the Itô $K$-rough path immediately follows from Formulas (\ref{def:hat-b-order-1}) to (\ref{def:hat-b-order-3-1}) and the fact that Skorohod integrals are known to coincide with Itô integrals in this situation (see for instance \cite{grorud-pardoux}). 

\end{proof}

\smallskip

We can finally combine the above construction with the general results of \cite{hai-14} about Equation (\ref{equation-base}), as we summed them up through Proposition \ref{prop:link-model}:

\begin{corollary}[Application to the equation]\label{coro:equation}
Fix $(H_1,H_2) \in (0,1)^2$ such that $2H_1+H_2 >\frac32$, $\al\in (-\frac32,-3+2H_1+H_2)$, and let $B$ be a $(H_1,H_2)$-fractional sheet defined on some complete probability space $(\Omega,\mathfrak{F},\mathbb{P})$, with representation (\ref{representation-sheet}). Fix an arbitrary time horizon $T>0$ and let $F:\R^2\to \R$ be a vector field that satisfies the assumptions of Proposition \ref{prop:link-model}. Also, consider a sequence of bounded deterministic initial conditions $\Psi^n $ that converges in $L^\infty (\R)$ to some element $\Psi$, and set, with the notations of Proposition \ref{prop:link-model} and Theorem \ref{theo:main}, $(T_0,Y):=\Phi^K_{T,F}(\widehat{\mathbf{B}})$. 

\smallskip

Then for every $0<T_1<T_0$ and as $n$ tends to infinity, the sequence $Y^n$ of classical solutions of the (renormalized) equation
\begin{equation}\label{eq-base}
\left\{\begin{array}{ccl}
\partial_t Y^n (t,x) &= &\partial^2_x Y^n(t,x)+F(x,Y^n(t,x)) \,\partial_t\partial_x B^n (t,x)-c^n_{H_1,H_2}\, F(x,Y^n(t,x))\, \partial_2F(x,Y^n(t,x))\ ,\\
Y^n(0,x) &= &\Psi^n(x) \ , 
\end{array}\right.
\end{equation}
a.s. converges in $L^\infty([0,T_1]\times \R)$ to $Y$, that is to the solution on $[0,T_0]$ of the equation
\begin{equation}\label{regu-struc-equation}
\partial_t Y=\partial^2_x Y+F(.,Y) \, \widehat{\mathbf{B}} \quad , \quad Y_0=\Psi \ ,
\end{equation}
understood in the sense of the $K$-rough paths. When $H_1=\frac12$ and $H_2 > \frac12$, one has $T_0=T$ and the solution $Y$ a.s. coincides with the solution on $[0,T]$ of the equation
\begin{equation}\label{ito-equation}
\partial_t Y=\partial^2_x Y+F(.,Y) \, \partial_t\partial_x B \quad , \quad Y_0=\Psi \ ,
\end{equation}
understood in the classical Itô sense. 
\end{corollary}

\begin{proof}
The first assertion is a straightforward application of Proposition \ref{prop:link-model} and Theorem \ref{theo:main}. When $H_1=\frac12$, the identification of the solution with Itô's solution is then a consequence of the identification result of Proposition \ref{prop:ito} (on the level of the $K$-rough path): the details of this lifting procedure can be found in \cite[Section 6]{hairer-pardoux}.
\end{proof}

\smallskip

The following basic picture illustrates the domain covered by the combination of the above results with the results of \cite{deya}, regarding the pair $(H_1,H_2)$. Based on Proposition \ref{prop:regu-brow-sh}, the successive stages for the global roughness $\al$ of the noise turn into successive slices for the combination $2H_1+H_2$. The black, resp. red, slice corresponds to the first-order, resp. second-order, situation where $\al\in (-1,0)$, resp. $\al\in (-\frac43,-1]$, and was treated in \cite{deya}. The blue slice corresponds to the setting of the present paper, with $\al\in (-\frac32,-\frac43]$. Its border extends up to the standard space-time white-noise situation ($H_1=H_2=\frac12$), as we pointed it out earlier. 

\

\

\

\

\

\

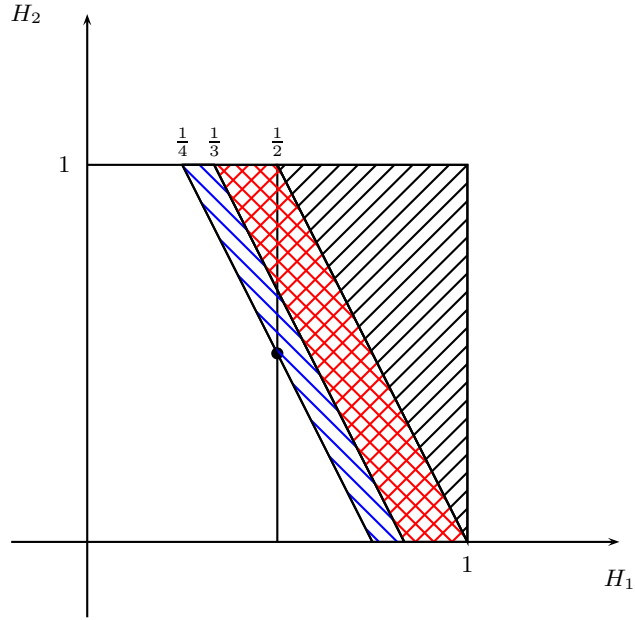
\begin{center}
\begin{figure}[ht]\label{figure-h1-h2}
\begin{pspicture}(0,0)(5,5)

\psline{->}(-1,0)(7,0) \psline{->}(0,-1)(0,7)

\psline(0,5)(5,5)(5,0)
\psline(2.5,0)(2.5,5)
\psline(2.5,5)(5,0) 
\psline(1.25,5)(3.75,0) 
\psline(1.67,5)(4.17,0) 
\pscircle*(2.5,2.5){0.08}

\pspolygon[fillstyle=vlines,hatchcolor=blue](1.25,5)(3.75,0)(4.17,0)(1.67,5)

\pspolygon[fillstyle=crosshatch,hatchcolor=red](1.67,5)(4.17,0)(5,0)(2.5,5)(1.25,5)

\pspolygon[fillstyle=hlines](2.5,5)(5,0)(5,5)(2.5,5)

\rput(7,-0.5){$H_1$} \rput(-0.8,7){$H_2$} 

 \rput(1.25,5.3){$\frac14$} 

 \rput(1.67,5.3){$\frac13$} 

 \rput(2.5,5.3){$\frac12$} 

 \rput(5,-0.3){$1$} \rput(-0.3,5){$1$}

\end{pspicture}
\end{figure}

\end{center}

\

\

The rest of the paper is now devoted to the proof of Theorem \ref{theo:main}. Therefore, from now on and until the end, we fix, on a complete probability space $(\Omega,\mathcal{F},\mathbb{P})$, a fractional sheet $B$ of Hurst index $(H_1,H_2)\in (0,1)^2$ satisfying 
$$\frac32< 2H_1+H_2\leq 2 \ ,$$
and with representation (\ref{representation-sheet}) with respect to some space-time white noise $W$. Also, we consider the smooth approximation $B^n$ defined by (\ref{approx-noise}) and, for some fixed localized heat kernel $K$, we denote by $\mathbf{B}^{n}$ the canonical $K$-rough path associated with $B^n$, in the sense of Definition \ref{defi:canonical}. Finally, we denote by
$$\widehat{\mathbf{B}}^{n}:=(\widehat{\mathbf{B}}^{\mathbf{1},n},\widehat{\mathbf{B}}^{\mathbf{2},n},\widehat{\mathbf{B}}^{\mathbf{3.1},n},\widehat{\mathbf{B}}^{\mathbf{3.2},n})$$
the renormalized $K$-rough path $\widehat{\mathbf{B}}^{n}:=\text{Renorm}(\mathbf{B}^{n},c^n_{H_1,H_2})$ (see Lemma \ref{lem:transfo-m}), where $c^n_{H_1,H_2}$ is defined by (\ref{defi-seq-reno}).

\smallskip

At this point, it must be recalled that the third-order results of Theorem \ref{theo:main} (covering the situation where $2H_1+H_2>\frac32$) are the continuation of the first and second-order results of \cite{deya} for the more restrictive case where $2H_1+H_2>\frac53$. In particular, the study of the first and second-order components of the renormalized canonical $K$-rough path (i.e., $\widehat{\mathbf{B}}^{\mathbf{1},n}$ and $\widehat{\mathbf{B}}^{\mathbf{2},n}$) can be done along the very same arguments and estimates as in \cite{deya}, which already provides us with the following preliminary statement:

\begin{lemma}(\cite[Corollaries 3.4 and 3.5]{deya}) \label{lem:results-orders-1-2}
For every $\al \in (-\frac32,-3+2H_1+H_2)$, there exist $\ep>0$ and a pair $(\widehat{\mathbf{B}}^{\mathbf{1}},\widehat{\mathbf{B}}^{\mathbf{2}}) \in \cac_c^\al(\R^2)\times  \pmb{\cac}_c^{2\al+2}(\R^2)$ such that
$$\widehat{\mathbf{B}}^{\mathbf{1}}=\partial_t\partial_\xi B \quad , \quad \widehat{\mathbf{B}}^{\mathbf{2}}_{x}-\widehat{\mathbf{B}}^{\mathbf{2}}_{y}= \widehat{\mathbf{B}}^{\mathbf{1}}\big( K^{(1)}_{x,y}\big) \cdot \widehat{\mathbf{B}}^{\mathbf{1}} \ , $$
and for all $n,k,p\geq 1$, 
\begin{equation}\label{estim-mom-x-1}
\mathbb{E}\big[\lVert \widehat{\mathbf{B}}^{\mathbf{1},n}\rVert_{\al;R_k}^p \big] \leq C_{p,\al} \, k^2 \quad , \quad \mathbb{E}\big[\lVert \widehat{\mathbf{B}}^{\mathbf{1},n}-\widehat{\mathbf{B}}^{\mathbf{1}}\rVert_{\al;R_k}^p \big] \leq C_{p,\al} \, k^2 \, 2^{-n\ep p}  \ ,
\end{equation}
\begin{equation}\label{estim-mom-x-2}
\mathbb{E}\big[\lVert \widehat{\mathbf{B}}^{\mathbf{2},n}\rVert_{2\al+2;R_k}^p \big] \leq C_{p,\al} \, k^2 \quad , \quad \mathbb{E}\big[\lVert \widehat{\mathbf{B}}^{\mathbf{2},n}-\widehat{\mathbf{B}}^{\mathbf{2}}\rVert_{2\al+2;R_k}^p \big] \leq C_{p,\al} \, k^2 \, 2^{-n\ep p}  \ ,
\end{equation}
for some constant $C_{p,\al}$, and where we have set $R_k:=[-k,k]^2$. 
\end{lemma}

\begin{remark}
We are aware that in \cite{deya}, the results of (\ref{estim-mom-x-2}) are only stated under the assumption that $\frac53 < 2H_1+H_2\leq 2$ (the bounds in (\ref{estim-mom-x-1}) are actually true for every $(H_1,H_2)\in (0,1)^2$). However, a close examination of the technical details in the latter reference would show to the - patient - reader that the computations remain valid for $2H_1+H_2\in (\frac32,\frac53]$ as well. In any case, the forthcoming proof of Formula (\ref{def:hat-b-order-2}) for $\widehat{\mathbf{B}}^{\mathbf{2}}$ would easily allow us to recover (\ref{estim-mom-x-2}) (see for instance (\ref{rappel-estim-ordre-deux})), and these bounds at first and second-order orders turn out to be elementary to obtain in comparison with the subsequent third-order estimates.
\end{remark}

\section{Chaos decomposition of the fractional canonical $K$-rough path}\label{sec:mall-cal}
Before we go further, note that, for the sake of clarity in the subsequent computations, and in opposition with the formulation of Section \ref{sec:main-results}, we will henceforth go back to more standard notations regarding time, resp. space, variables, and denote them by $s,t,u$, resp. $x,y,z$.

\smallskip

Now, in the Gaussian setting under consideration, our strategy to prove the convergence of $\widehat{\mathbf{B}}^{n}$ is based on the following natural (albeit technical) two-step procedure:

\smallskip

\noindent
$(i)$ For fixed $n$, and using Malliavin-calculus tools, expand the components of $\widehat{\mathbf{B}}^{n}$ as a sum of Skorohod integrals with respect to $B^n$;

\smallskip

\noindent
$(ii)$ Show the convergence, as $n$ tends to infinity and for an appropriate topology, of each of the summands in these chaos-type decompositions.

\

The present section \ref{sec:mall-cal} is devoted to the proof of Step $(i)$, and therefore, some preliminary material on Malliavin calculus must be introduced. In fact, for the two-parameter processes we shall consider in the sequel, an exhaustive presentation of this material can be found in \cite[Sections 5 and 6]{chouk-tindel}, and accordingly we will not return to the definition of the classical objects therein introduced, namely the Hilbert space $\mathcal{H}_Z$, the Malliavin derivative $D^Z$ and the Skorohod integral $\delta^Z$ associated with any centered Gaussian field $\{Z(s,x); \ s,x\in \R\}$ defined on a complete probability space $(\Omega, \cf,\mathbb{P})$.

\smallskip

However, what must be underlined in this situation (i.e., when working with $B$ or $B^n$) is that thanks to the representation (\ref{representation-sheet}), resp. (\ref{approx-noise}), there exists a close link between the Malliavin calculus with respect to $B$, resp. $B^n$, and the Malliavin calculus with respect to $W$ or $\widehat{W}:=\mathcal{F}(W)$. To elaborate on these relations, let us introduce the family of operators $\cq_{\al_1,\al_2}$ ($\al_1,\al_2\in (0,1)$), resp. $\cq^n_{\al_1,\al_2}$, defined for every measurable, compactly-supported function $\vp$ and every $\xi,\eta \in \R$ as 
$$\cq_{\al_1,\al_2}(\vp)(\xi,\eta):=-c_{\al_1,\al_2} \frac{\xi\cdot \eta}{|\xi|^{\al_1+\frac12} |\eta|^{\al_2+\frac12}}\, \fouri{\vp}(-\xi,-\eta) \ ,$$
resp.
$$\cq^n_{\al_1,\al_2}(\vp)(\xi,\eta):=-c_{\al_1,\al_2}\1_{\{(\xi,\eta) \in \cd_n\}} \frac{\xi\cdot \eta}{|\xi|^{\al_1+\frac12} |\eta|^{\al_2+\frac12}}\, \fouri{\vp}(-\xi,-\eta) \ ,$$
where $c_{\al_1,\al_2}$ is the same constant as in the representation (\ref{representation-sheet}). We can then rely on the following identities: for every test-function $\vp$, every functional $F=F(B)$, resp. $F^n=F^n(B^n)$, smooth enough (in the sense of Malliavin calculus) and every $\ch_B$-valued, resp. $\ch_{B^n}$-valued, random variable $u$ in an appropriate domain, it holds that
\begin{equation}\label{identi-q}
\lVert \vp\rVert_{\ch_B}=\lVert \cq_{H_1,H_2}(\vp)\rVert_{L^2(\R^2)} \quad , \quad \langle u,D^BF\rangle_{\ch_B}=\langle \cq_{H_1,H_2}(u),D^{\widehat{W}}F \rangle_{L^2(\R)} \quad , \quad \delta^{B}(u)=\delta^{\widehat{W}}(\cq_{H_1,H_2}(u)) \ ,
\end{equation}
resp.
\begin{equation}\label{identi-q-n}
\lVert \vp\rVert_{\ch_{B^n}}=\lVert \cq^n_{H_1,H_2}(\vp)\rVert_{L^2(\R^2)} \quad , \quad  \langle u,D^{B^n}F^n\rangle_{\ch_{B^n}}=\langle \cq^n_{H_1,H_2}(u),D^{\widehat{W}}F^n \rangle_{L^2(\R)} \quad , \quad \delta^{B^n}(u)=\delta^{\widehat{W}}(\cq^n_{H_1,H_2}(u)) \ .
\end{equation}
Here again, we refer the reader to \cite{chouk-tindel} (and more specifically to \cite[Lemma 6.1]{chouk-tindel}) for a proof of these identities, as well as for further details regarding the specific assumptions on $F$, $F^n$ and $u$. Let us also recall the following general product rule satisfied by the Skorohod integral, for either $Z:=B$ or $Z:=B^n$ (see \cite{nualart}):
\begin{equation}\label{mall-product-rule}
\delta^Z(F\, u)=F \, \delta^Zu-\langle D^Z F,u\rangle_{\ch_Z} \ .
\end{equation}

\smallskip

Identities (\ref{identi-q}) and (\ref{identi-q-n}) point out the important role played by the operators $\cq_{\al_1,\al_2},\cq^n_{\al_1,\al_2}$ in this setting. The following related estimates will thus prove to be fundamental in the sequel:
\begin{lemma}\label{lem:sobolev}
Let $\al_1,\al_2 \in (0,1)$. For every smooth compactly-supported function $\vp:\R^2 \to \R$, it holds that
\begin{equation}\label{bou-q-n}
\lVert \cq^n_{\al_1,\al_2}(\vp)\Vert_{L^2(\R^2)} \leq d^n_{H_1,H_2}\lVert \vp \rVert_{L^1(\R^2)}
\end{equation}
with $d^n_{H_1,H_2}\to \infty$ as $n\to \infty$, and
\begin{equation}\label{bou-q}
\lVert \cq_{\al_1,\al_2}(\vp)\Vert_{L^2(\R^2)} \lesssim \lVert \vp \rVert_{\al_1,\al_2} \ ,
\end{equation}
where $\lVert \vp \rVert_{\al_1,\al_2}$ is defined along the following formulas:

\smallskip

\noindent
$\bullet$ If $\al_1,\al_2 \in (\frac12,1)$, then
\begin{equation}\label{sobol-cas-plus-grand}
\lVert \vp \rVert_{\al_1,\al_2}^2 :=
\iint_{\R^2} ds  dx \bigg( \iint_{\R^2}dt dy \, \frac{\vp(t,y)}{\lln t-s\rrn^{\frac32-\al_1} \lln y-x\rrn^{\frac32-\al_2}} \bigg)^2 \ .
\end{equation}

\smallskip

\noindent
$\bullet$ If $\al_1 \in (0,\frac12)$ and $\al_2\in (\frac12,0)$, then
\begin{equation}\label{sobolev-1}
\lVert \vp \rVert_{\al_1,\al_2}^2 :=
\int_{\R} ds \int_{\R} dx \bigg( \int_{\R} dy \, \frac{\vp(s,y)}{|x-y|^{\frac32-\al_2}} \bigg)^2+\iint_{\R^2} \frac{ds dt}{|s-t|^{2-2\al_1}}\int_{\R}dx \bigg( \int_{\R} dy \, \frac{\vp(s,y)-\vp(t,y)}{|x-y|^{\frac32 -\al_2}} \bigg)^2 \ .
\end{equation}

\smallskip

\noindent
$\bullet$ If $\al_1,\al_2 \in (0,\frac12)$, then
\begin{multline}\label{sobolev-2}
\lVert \vp \rVert_{\al_1,\al_2}^2 :=\int_{\R} ds \int_{\R} dx \, |\vp(s,x)|^2+\int_{\R} ds\int_{\R^2} \frac{dx dy}{|x-y|^{2-2\al_2}} |\vp(s,x)-\vp(s,y)|^2\\
+\int_{\R} dx\int_{\R^2} \frac{ds dt}{|s-t|^{2-2\al_1}} |\vp(s,x)-\vp(t,x)|^2+\int_{\R^2} dsdt \int_{\R^2}dxdy\, \frac{|\vp(s,x)-\vp(t,x)-\vp(s,y)+\vp(t,y)|^2}{|s-t|^{2-2\al_1} |x-y|^{2-2\al_2}} \ . 
\end{multline}
\end{lemma}

\begin{proof}
The bound (\ref{bou-q-n}) is of course immediate: setting $d^n_{H_1,H_2}:=\int_{\cd_n} \frac{d\xi d\eta}{|\xi|^{2H_1-1}|\eta|^{2H_2-1}}$, one has 
$$\int_{\cd_n} \frac{d\xi d\eta}{|\xi|^{2H_1-1}|\eta|^{2H_2-1}} |\mathcal{F}(\vp)(-\xi,-\eta)|^2\leq d^n_{H_1,H_2} \lVert \mathcal{F}(\vp)\rVert^2_{L^\infty(\R^2)}  \leq d^n_{H_1,H_2}\lVert \vp\rVert^2_{L^1(\R^2)} \ .$$
As for (\ref{bou-q}), it can easily be derived from the same Sobolev-embedding arguments as in \cite[Lemma 4.4]{deya}.
\end{proof}

We are now in a position to prove the desired decomposition formulas for the components of $\widehat{\mathbf{B}}^n$ (for each fixed $n$), namely:
\begin{proposition}
For every smooth compactly-supported function $\psi:\R^2 \to \R$ and every $(s,x)\in \R^2$, one has, in $L^2(\Omega)$,
\begin{equation}\label{terme-simpl-deux}
\widehat{\mathbf{B}}^{\mathbf{2},n}_{(s,x)}(\psi)=\delta^{B^n}\big( \psi \cdot \widehat{\mathbf{B}}^{\mathbf{1},n}\big(K^{(1)}_{(s,x),.}\big) \big)+L^n_{(s,x)}(\psi)
\end{equation}
\begin{equation}\label{terme-simpl-trois-un}
\widehat{\mathbf{B}}^{\mathbf{3.1},n}_{(s,x)}(\psi)
=\delta^{B^n}\big( \psi \cdot \widehat{\mathbf{B}}^{\mathbf{2},n}_{(s,x)}\big(K^{(2)}_{(s,x),.}\big) \big)+\big[ L^{1,n,\dot{B}^n}_{(s,x)}(\psi)+L^{2,n,\dot{B}^n}_{(s,x)}(\psi)+L^{3,n,\dot{B}^n}_{(s,x)}(\psi)\big] \ ,
\end{equation}
and
\begin{equation}\label{terme-simpl-trois}
\widehat{\mathbf{B}}^{\mathbf{3.2},n}_{(s,x)}(\psi)
=\delta^{B^n}\big( \psi \cdot \big(\widehat{\mathbf{B}}^{\mathbf{1},n}_{(s,x)}\big(K^{(1)}_{(s,x),.}\big)\big)^2 \big)+2\, L^{1,n,\dot{B}^n}_{(s,x)}(\psi) \ ,
\end{equation}
where, following the notations of Section \ref{subsec:main-results}, we have set, for every $X\in \cac_c^\al(\R^2)$, 
$$L^n_{(s,x)}(\psi):=-c_{H_1,H_2}^2 \iint_{\cd_n}d\xi d\eta \, \frac{e^{\imath \xi s} e^{\imath \eta x}}{|\xi|^{2H_1-1} |\eta|^{2H_2-1}}\mathcal{F}(K)(\xi,\eta)\, \mathcal{F}(\psi)(\xi,\eta) \ ,$$
\begin{equation*}
L^{1,n,X}_{(s,x)}(\psi):= -c_{H_1,H_2}^2 \iint_{\cd_n}d\xi d\eta\, \frac{e^{\imath \xi s} e^{\imath \eta x}}{|\xi|^{2H_1-1} |\eta|^{2H_2-1}}\mathcal{F}(K)(\xi,\eta)\,\fouri{\Theta^{(1)}_{(s,x)}(\psi,X)}(\xi,\eta)  \ , 
\end{equation*}
\begin{equation*}
L^{2,n,X}_{(s,x)}(\psi):= - c_{H_1,H_2}^2 \iint_{\cd_n}d\xi d\eta\, \frac{e^{\imath \xi s} e^{\imath \eta x}}{|\xi|^{2H_1-1} |\eta|^{2H_2-1}}\mathcal{F}(D^{(0,1)}K)(\xi,\eta)\,\fouri{\Theta^{(2)}_{(s,x)}(\psi,X)}(\xi,\eta)  \ , 
\end{equation*}
and
\begin{equation*}
L^{3,n,X}_{(s,x)}(\psi):=c_{H_1,H_2}^2\iint_{\cd_n}\frac{d\xi d\eta}{|\xi|^{2H_1-1} |\eta|^{2H_2-1}}\fouri{\Theta^{(3)}_{(s,x)}(\psi,X)}(\xi,\eta)  \ .
\end{equation*}
\end{proposition}

\begin{proof}
Let us first focus on the most intricate identity, namely (\ref{terme-simpl-trois-un}). To this end, we set
$$V^n_{(s,x)}(t,y):=\psi(t,y) \, \widehat{\mathbf{B}}^{\mathbf{2},n}_{(s,x)}\big(K^{(2)}_{(s,x),(t,y)}\big)$$
and consider two sequences of partitions $t_i=t_i^\ell:=\frac{i}{\ell}$, $y_j=y_j^m:=\frac{j}{m}$ ($\ell,m\in  \mathbb{N}$, $i,j\in \mathbb{Z}$). For all fixed $i,j,\ell,m$, we know by (\ref{mall-product-rule}) that, writing $\1_{\square_{ij}}$ for $\1_{[t_i,t_{i+1})\times [y_j,y_{j+1})}$, one has
\begin{equation}\label{appli-product-rule}
\delta^{B^n}\big( V^n_{(s,x)}(t_i,y_j) \, \1_{\square_{ij}} \big)=V^n_{(s,x)}(t_i,y_j) \,\delta^{B^n}\big( \1_{\square_{ij}} \big)-\langle D^{\widehat{W}}\big( V^n_{(s,x)}(t_i,y_j)\big),\cq^n_{H_1,H_2}\big(\1_{\square_{ij}}\big) \rangle_{L^2(\R^2)} \ .
\end{equation}
Now, since
\begin{align*}
&D^{\widehat{W}}\big( V^n_{(s,x)}(t_i,y_j)\big)(\xi,\eta)\\
&=-c_{H_1,H_2} \psi(t_i,y_j) \frac{\xi \cdot \eta}{|\xi|^{H_1+\frac12}|\eta|^{H_2+\frac12}} \iint_{\R^2} du dz \, e^{\imath u\xi} e^{\imath z \eta}\dot{B}^n\big(K^{(1)}_{(s,x),(u,z)} \big)\, K^{(2)}_{(s,x),(t_i,y_j)}(u,z) \ ,
\end{align*}
we can assert that
\begin{align*}
&\sum_{i,j\in \Z}\langle D^{\widehat{W}}\big( V^n_{(s,x)}(t_i,y_j)\big),\cq^n_{H_1,H_2}\big(\1_{\square_{ij}}\big) \rangle_{L^2(\R^2)}=c_{H_1,H_2}^2\iint_{\cd_n}\frac{d\xi d\eta}{|\xi|^{2H_1-1}|\eta|^{2H_2-1}}\\
&\hspace{2cm}\frac{1}{\ell\cdot m} \sum_{i,j\in \Z}\psi(t_i,y_j) \Big[1+\int_0^1 d\la \, \{e^{-\imath \la (t_{i+1}-t_i)\xi}-1\} \Big] \Big[1+\int_0^1 d\la \, \{e^{-\imath \la (y_{j+1}-y_j)\eta}-1\} \Big]\\
&\hspace{4cm}\iint_{\R^2} dudz \, e^{-\imath (t_i-u)\xi}e^{-\imath (y_j-z)\eta} \dot{B}^n\big(K^{(1)}_{(s,x),(u,z)} \big)\, K^{(2)}_{(s,x),(t_i,y_j)}(u,z) \ .
\end{align*}
From this expression, and with the help of (\ref{estim-x-1-1}) and (\ref{estim-mom-x-1}), we can easily justify that as $\ell,m\to\infty$,
\begin{align*}
&\sum_{i,j} \langle D^{\widehat{W}}\big( V^n_{(s,x)}(t_i,y_j)\big),\cq^n_{H_1,H_2}\big(\1_{\square_{ij}}\big) \rangle_{L^2(\R^2)}\\
& \stackrel{L^2(\Omega)}{\longrightarrow} c_{H_1,H_2}^2 \iint_{\cd_n} \frac{d\xi \, d\eta}{|\xi|^{2H_1-1}|\eta|^{2H_2-1}}\\
&\hspace{2cm} \iint_{\R^2} dt dy \, \psi(t,y) \iint_{\R^2} du dz \, e^{-\imath \xi (t-u)} e^{-\imath \eta (y-z)} \dot{B}^n\big(K^{(1)}_{(s,x),(u,z)} \big)\, K^{(2)}_{(s,x),(t,y)}(u,z) \\
&\hspace{1cm}= c_{H_1,H_2}^2\iint_{\cd_n} \frac{d\xi \, d\eta}{|\xi|^{2H_1-1}|\eta|^{2H_2-1}} \iint_{\R^2} dt dy \, \psi(t,y) \dot{B}^n\big(K^{(1)}_{(s,x),(t,y)}\big)\\
& \hspace{3cm}\iint_{\R^2} du dz \, e^{-\imath \xi (t-u)} e^{-\imath \eta (y-z)} K^{(2)}_{(s,x),(t,y)}(u,z)+L^{3,n,\dot{B}^n}_{(s,x)}(\psi) \ .
\end{align*}
At this point, observe that
\begin{align*}
&\iint_{\R^2} du dz \, e^{-\imath \xi (t-u)} e^{-\imath \eta (y-z)} K^{(2)}_{(s,x),(t,y)}(u,z)\\
&=\fouri{K}(\xi,\eta)-e^{-\imath \xi(t-s)} e^{-\imath \eta (y-x)}\big\{ \fouri{K}(\xi,\eta)+(y-x)\cdot\fouri{(D^{(0,1)}K)}(\xi,\eta)\big\}\ ,
\end{align*}
which immediately entails that
\begin{align*}
&c_{H_1,H_2}^2\iint_{\cd_n} \frac{d\xi \, d\eta}{|\xi|^{2H_1-1}|\eta|^{2H_2-1}} \iint_{\R^2} dt dy \, \psi(t,y) \dot{B}^n\big(K^{(1)}_{(s,x),(t,y)}\big)\iint_{\R^2} du dz \, e^{-\imath \xi (t-u)} e^{-\imath \eta (y-z)} K^{(2)}_{(s,x),(t,y)}\\
&\hspace{2cm}=c^n_{H_1,H_2} \dot{B}^n\big(K^{(1)}_{(s,x),.}\big)(\psi)+L^{1,n,\dot{B}^n}_{(s,x)}(\psi)+L^{2,n,\dot{B}^n}_{(s,x)}(\psi) \ .
\end{align*}

\

\noindent
Going back to (\ref{appli-product-rule}) and observing in addition that $\widehat{\mathbf{B}}^{\mathbf{2},n}_{(s,x)}\big(K^{(2)}_{(s,x),(t,y)}\big)=\mathbf{B}^{\mathbf{2},n}_{(s,x)}\big(K^{(2)}_{(s,x),(t,y)}\big)$, it remains us to prove that as $\ell,m\to \infty$, one has
\begin{equation}\label{reste-proof}
\delta^{B^n}\Big( \sum_{i,j\in \Z} V^n_{(s,x)}(t_i,y_j) \, \1_{\square_{ij}}\Big) \stackrel{L^2(\Omega)}{\longrightarrow} \delta^{B^n}(V^n_{(s,x)}) \quad \text{and} \quad \sum_{i,j\in \Z} V^n_{(s,x)}(t_i,y_j) \, \delta^{B^n}(\1_{\square_{ij}}) \stackrel{L^2(\Omega)}{\longrightarrow}\mathbf{B}^{\mathbf{3.1},n}_{(s,x)}(\psi) \ . 
\end{equation}
Using the basic properties of the Skorohod integral, the proof of the first convergence actually reduces to showing that
$$\sum_{i,j\in \Z} V^n_{(s,x)}(t_i,y_j) \, \1_{\square_{ij}}\to V^n_{(s,x)} \quad \text{in} \ L^2(\Omega;\ch_{B^n}) \ .$$
To this end, we can first invoke (\ref{identi-q-n}) and (\ref{bou-q-n}) to assert that
\begin{eqnarray*}
\lefteqn{\mathbb{E}\Big[ \big\lVert \sum_{i,j\in \Z} V^n_{(s,x)}(t_i,y_j) \, \1_{\square_{ij}}-V^n_{(s,x)}\big\rVert_{\ch_{B^n}}^2 \Big]}\\
&\leq &(d^n_{H_1,H_2})^2\, \mathbb{E}\Big[ \big\lVert \sum_{i,j\in \Z} V^n_{(s,x)}(t_i,y_j) \, \1_{\square_{ij}}-V^n_{(s,x)}\big\rVert_{L^1(\R^2)}^2 \Big]\\
&\leq& (d^n_{H_1,H_2})^2\, \mathbb{E}\Big[ \Big( \sum_{i,j\in \Z} \int_{t_i}^{t_{i+1}}dt\int_{y_j}^{y_{j+1}}dy \, |V^n_{(s,x)}(t,y)-V^n_{(s,x)}(t_i,y_j)|\Big)^2\Big]\\
&\leq& (d^n_{H_1,H_2})^2\, \Big( \sum_{i,j\in \Z} \int_{t_i}^{t_{i+1}}dt\int_{y_j}^{y_{j+1}}dy \, \mathbb{E}\big[ |V^n_{(s,x)}(t,y)-V^n_{(s,x)}(t_i,y_j)|^2\big]^{1/2}\Big)^2 \ .
\end{eqnarray*}
Going back to the definition of $V^n_{(s,x)}$, the conclusion now follows from the combination of (\ref{estim-x-2-1})-(\ref{estim-x-2-2-time})-(\ref{estim-x-2-2}) and (\ref{estim-mom-x-2}).

\smallskip

\noindent
As for the second convergence statement in (\ref{reste-proof}), observe that
\begin{eqnarray*}
\lefteqn{\mathbb{E}\Big[ \big| \sum_{i,j\in \Z} V^n_{(s,x)}(t_i,y_j) \, \delta^{B^n}(\1_{\square_{ij}})-\mathbf{B}^{\mathbf{3.1},n}_{(s,x)}(\psi)\big|^2\Big]}\\
&=&\mathbb{E}\Big[ \big| \sum_{i,j\in \Z} \int_{t_i}^{t_{i+1}} dt \int_{y_j}^{y_{j+1}}dy \, \big\{V^n_{(s,x)}(t,y)-V^n_{(s,x)}(t_i,y_j)\big\} \, \dot{B}^n(t,y) \big|^2\Big]\\
&\leq & \Big( \sum_{i,j\in \Z} \int_{t_i}^{t_{i+1}} dt \int_{y_j}^{y_{j+1}}dy \, \mathbb{E}\big[ |V^n_{(s,x)}(t,y)-V^n_{(s,x)}(t_i,y_j)|^4\big]^{1/4}\mathbb{E}\big[ |\dot{B}^n(t,y)|^4\big]^{1/4}\Big)^2 \ .
\end{eqnarray*}
Just as above, we can now conclude by using (\ref{estim-x-2-1})-(\ref{estim-x-2-2-time})-(\ref{estim-x-2-2}) and (\ref{estim-mom-x-2}), together with the fact that
$$ \mathbb{E}\big[ |\dot{B}^n(t,y)|^4\big]^{1/4} \lesssim \mathbb{E}\big[ |\dot{B}^n(t,y)|^2\big]^{1/2} \lesssim \iint_{\cd_n} \frac{d\xi d\eta}{|\xi|^{2H_1-1}|\eta|^{2H_2-1}} \ .$$
This achieves the proof of (\ref{terme-simpl-trois-un}).

\smallskip

The - less sophisticated - identities (\ref{terme-simpl-deux}) and (\ref{terme-simpl-trois}) can then be shown along the very same arguments, and therefore we leave their proofs to the reader as an exercise.
\end{proof}

\section{Convergence of the decomposition}\label{sec:proof-main-results}
We turn here to the second - and final - step of the strategy sketched out at the beginning of the previous section. Thus, in brief, our aim now is to prove the convergence of $\widehat{\mathbf{B}}^n$ by showing the convergence of each of the summands in the decompositions (\ref{terme-simpl-deux})-(\ref{terme-simpl-trois}). Identities (\ref{def:hat-b-order-2})-(\ref{def:hat-b-order-3-1}) will then be obtained as immediate consequences of this extension.

\smallskip

Throughout the section, and just as in Lemma \ref{lem:results-orders-1-2}, we will use the notation $R_k:=[-k,k]^2$, for $k\geq 0$. Besides, we recall that given $\psi:\R^2 \to \R$, $(s,x)\in \R^2$ and $\ell \geq 0$, we denote by $\psi^\ell_{(s,x)}$ the rescaled function
$$\psi_{(s,x)}^\ell(t,y):=2^{3\ell} \psi(2^{2\ell}(t-s),2^\ell(y-x)) \ , \ \text{for all} \ (t,y)\in \R^2 \ .$$
For a clear statement of our result, let us introduce the processes $\widetilde{\mathbf{B}}^{\mathbf{2}}$, $\widetilde{\mathbf{B}}^{\mathbf{3.1}}$ and $\widetilde{\mathbf{B}}^{\mathbf{3.2}}$ defined by the right-hand sides of (\ref{def:hat-b-order-2}), (\ref{def:hat-b-order-3}) and (\ref{def:hat-b-order-3-1}), that is
$$
\widetilde{\mathbf{B}}^{\mathbf{2}}_{(s,x)}(\psi)
:=\delta^{B}\Big( \psi \cdot \widehat{\mathbf{B}}^{\mathbf{1}}\big(K^{(1)}_{(s,x),.}\big) \Big)+L_{(s,x)}(\psi) \ ,
$$
$$\widetilde{\mathbf{B}}^{\mathbf{3.1}}_{(s,x)}(\psi):=\delta^{B}\Big( \psi\cdot \widehat{\mathbf{B}}^{\mathbf{2}}_{(s,x)}\big(K^{(2)}_{(s,x),.}\big) \Big)+\big[ L^{1,\dot{B}}_{(s,x)}(\psi)+L^{2,\dot{B}}_{(s,x)}(\psi)+L^{3,\dot{B}}_{(s,x)}(\psi)\big]$$
and
$$\widetilde{\mathbf{B}}^{\mathbf{3.2}}_{(s,x)}(\psi):=\delta^{B}\Big( \psi \cdot \big(\widehat{\mathbf{B}}^{\mathbf{1}}\big(K^{(1)}_{(s,x),.}\big)\big)^2 \Big)+2\, L^{1,\dot{B}}_{(s,x)}(\psi) \ ,$$
where the processes $\widehat{\mathbf{B}}^{\mathbf{1}},\widehat{\mathbf{B}}^{\mathbf{2}}$ have been introduced through the preliminary Lemma \ref{lem:results-orders-1-2}.

\smallskip

The result of the main technical step of our analysis now reads as follows (we recall that the set $\cb$ of test-functions has been introduced at the beginning of Section \ref{sec:main-results}):
\begin{proposition}\label{main-bound}
For every $\al \in (-\frac32,-3+2H_1+H_2)$, there exists $\ep>0$ such that for all $n,\ell\geq 0$, $k\geq 1$, $(s,x)\in R_k$ and $\psi\in \cb$, one has
\begin{equation}\label{rappel-estim-ordre-deux}
\mathbb{E}\Big[\Big|\left\langle \widehat{\mathbf{B}}^{\mathbf{2},n}_{(s,x)}-\widetilde{\mathbf{B}}^{\mathbf{2}}_{(s,x)},\psi^\ell_{(s,x)} \right\rangle\Big|^2 \Big] \lesssim  k^2 2^{-n\ep} 2^{-2\ell (2\al+2)} \ ,
\end{equation}
\begin{equation}\label{mai-estim-ordre-trois-un}
\mathbb{E}\Big[\Big|\left\langle \widehat{\mathbf{B}}^{\mathbf{3.1},n}_{(s,x)}-\widetilde{\mathbf{B}}^{\mathbf{3.1}}_{(s,x)},\psi^\ell_{(s,x)} \right\rangle\Big|^2 \Big] \lesssim k^2 2^{-n\ep} 2^{-2\ell (3\al+4)} \ ,
\end{equation}
and
\begin{equation}\label{mai-estim-ordre-trois-deux}
\mathbb{E}\Big[\Big|\left\langle \widehat{\mathbf{B}}^{\mathbf{3.2},n}_{(s,x)}-\widetilde{\mathbf{B}}^{\mathbf{3.2}}_{(s,x)},\psi^\ell_{(s,x)} \right\rangle\Big|^2 \Big] \lesssim k^2 2^{-n\ep} 2^{-2\ell (3\al+4)} \ ,
\end{equation}
where the proportional constants are independent of $(n,k,\ell)$ and $\psi$.
\end{proposition}

\begin{proof}
Starting from the decompositions (\ref{terme-simpl-deux})-(\ref{terme-simpl-trois}), these bounds are straightforward consequences of the estimates of the next subsections. To be more specific, (\ref{rappel-estim-ordre-deux}) follows from the combination of (\ref{second-chaos}) and (\ref{estim-l}), (\ref{mai-estim-ordre-trois-deux}) from the combination of (\ref{third-chaos-2}) and (\ref{estim-l-x-1-appli}), while (\ref{mai-estim-ordre-trois-un}) follows from (\ref{third-chaos-1}), (\ref{estim-l-x-1-appli}), (\ref{estim-l-x-2-appli}) and (\ref{estim-l-x-3-appli}).
\end{proof}

Before we turn to the proof of the technical estimates behind Proposition \ref{main-bound}, let us see how the latter result can be used in order to derive our main Theorem \ref{theo:main}.

\begin{proof}[Proof of Theorem \ref{theo:main}]
Observe first that, due to the standard hypercontractivity properties of Gaussian chaoses, the bounds (\ref{rappel-estim-ordre-deux})-(\ref{mai-estim-ordre-trois-deux}) can be readily turned into $L^{2p}(\Omega)$-estimates, for every $p\geq 1$. In other words, for every $\al\in (-\frac32,-3+2H_1+H_2)$ and every $p\geq 1$, there exists a constant $\ep>0$ such that for
all $n,\ell\geq 0$, $k\geq 1$ and $(s,x)\in R_k$, one has
\begin{equation}\label{rappel-estim-ordre-deux-p}
\mathbb{E}\Big[\Big|\left\langle \widehat{\mathbf{B}}^{\mathbf{2},n}_{(s,x)}-\widetilde{\mathbf{B}}^{\mathbf{2}}_{(s,x)},\psi^\ell_{(s,x)} \right\rangle\Big|^{2p} \Big] \lesssim k^{2p} 2^{-n\ep} 2^{-2\ell p(2\al+2)} \ ,
\end{equation}
\begin{equation}\label{mai-estim-ordre-trois-un-p}
\mathbb{E}\Big[\Big|\left\langle \widehat{\mathbf{B}}^{\mathbf{3.1},n}_{(s,x)}-\widetilde{\mathbf{B}}^{\mathbf{3.1}}_{(s,x)},\psi^\ell_{(s,x)} \right\rangle\Big|^{2p} \Big] \lesssim k^{2p} 2^{-n\ep} 2^{-2\ell p (3\al+4)} \ ,
\end{equation}
and
\begin{equation}\label{mai-estim-ordre-trois-deux-p}
\mathbb{E}\Big[\Big|\left\langle \widehat{\mathbf{B}}^{\mathbf{3.2},n}_{(s,x)}-\widetilde{\mathbf{B}}^{\mathbf{3.2}}_{(s,x)},\psi^\ell_{(s,x)} \right\rangle\Big|^{2p} \Big] \lesssim k^{2p} 2^{-n\ep} 2^{-2\ell p (3\al+4)} \ ,
\end{equation}
where the proportional constants are independent of $(n,k,\ell)$ (but depend on $p$).

\smallskip

\noindent
Now, fix $\al\in (-\frac32,-3+2H_1+H_2)$ and let $m\geq n$ be two positive integers. By applying Lemma \ref{GRR-gene} and using the $K$-chen relations satisfied by $\widehat{\mathbf{B}}^n,\widehat{\mathbf{B}}^m$, we easily obtain that for $\mathbf{o}\in \{\mathbf{3.1},\mathbf{3.2}\}$,
\begin{equation}\label{appli-grr-ordre-trois}
\lVert \widehat{\mathbf{B}}^{\mathbf{o},n}- \widehat{\mathbf{B}}^{\mathbf{o},m}\rVert_{3\al+4;R_k}\lesssim \sup_{\psi\in \cb_0} \sup_{\ell\geq 0} \sup_{x\in \Lambda_\scal^\ell \cap R_{k+1}} 2^{\ell(3\al+4)} |\langle \widehat{\mathbf{B}}^{\mathbf{o},n}_x- \widehat{\mathbf{B}}^{\mathbf{o},m}_x, \psi_x^\ell \rangle |+M_{n,m,k}
\end{equation}
with
\begin{align*}
&M_{n,m,k}:=\lVert \widehat{\mathbf{B}}^{\mathbf{1},n}- \widehat{\mathbf{B}}^{\mathbf{1},m}\rVert_{\al;R_{k+1}}\, \{1+\lVert \widehat{\mathbf{B}}^{\mathbf{1},n}\rVert_{\al;R_{k+1}}^2+\lVert \widehat{\mathbf{B}}^{\mathbf{1},m}\rVert_{\al;R_{k+1}}^2+\lVert \widehat{\mathbf{B}}^{\mathbf{2},n}\rVert_{2\al+2;R_{k+1}}+\lVert \widehat{\mathbf{B}}^{\mathbf{2},m}\rVert_{2\al+2;R_{k+1}}\}\\
&\hspace{4cm}+\{\lVert \widehat{\mathbf{B}}^{\mathbf{1},n}\rVert_{\al;R_{k+1}}+\lVert \widehat{\mathbf{B}}^{\mathbf{1},m}\rVert_{\al;R_{k+1}} \}\, \lVert \widehat{\mathbf{B}}^{\mathbf{2},n}-\widehat{\mathbf{B}}^{\mathbf{2},m}\rVert_{2\al+2;R_{k+1}} \ .
\end{align*}
At this point, let us pick $\al'\in (\al,-3+2H_1+H_2)$ and use (\ref{rappel-estim-ordre-deux-p})-(\ref{mai-estim-ordre-trois-deux-p}) to get that, for $\mathbf{o}\in \{\mathbf{3.1},\mathbf{3.2}\}$,
\begin{align*}
&\mathbb{E}\bigg[ \bigg( \sup_{\psi\in \cb_0} \sup_{\ell\geq 0} \sup_{x\in \Lambda_\scal^\ell \cap R_{k+1}} 2^{\ell(3\al+4)} |\langle \widehat{\mathbf{B}}^{\mathbf{o},n}_x- \widehat{\mathbf{B}}^{\mathbf{o},m}_x, \psi_x^\ell \rangle |\bigg)^{2p}\bigg]\\
&\lesssim k^{2p}\sum_{\ell \geq 0} |\Lambda_\scal^\ell \cap R_{k+1}| \, 2^{-6\ell p(\al'-\al)}2^{-n\ep} \lesssim k^{2p+2} 2^{-n\ep} \sum_{\ell\geq 0} 2^{-3\ell (2p(\al'-\al)-1)}\ ,
\end{align*}
and of course $\sum_{\ell\geq 0} 2^{-3\ell (2p(\al'-\al)-1)}<\infty$ provided $p$ is chosen large enough. Therefore, going back to (\ref{appli-grr-ordre-trois}) and using also (\ref{estim-mom-x-1})-(\ref{estim-mom-x-2}), we can assert that for every $p$ large enough, one has, for $\mathbf{o}\in \{\mathbf{3.1},\mathbf{3.2}\}$,
\begin{equation}\label{conv-ordr-trois}
\mathbb{E}\big[\lVert \widehat{\mathbf{B}}^{\mathbf{o},n}- \widehat{\mathbf{B}}^{\mathbf{o},m}\rVert_{3\al+4;R_k}^{2p}\big] \lesssim k^{2p+2} 2^{-n\ep} \ .
\end{equation}
Then, since
$$\mathbb{E}\big[ d_\al(\widehat{\mathbf{B}}^n,\widehat{\mathbf{B}}^m)^{2p}\big] \lesssim \sum_{k\geq 0} 2^{-k}\mathbb{E}\big[ \lVert \widehat{\mathbf{B}}^n;\widehat{\mathbf{B}}^m\rVert_{\al;R_k}^{2p}\big] \ ,$$
we can combine (\ref{conv-ordr-trois}) with (\ref{estim-mom-x-1})-(\ref{estim-mom-x-2}) to conclude that every $p$ large enough, $(\widehat{\mathbf{B}}^n)_{n\geq 0}$ is a Cauchy sequence in the complete space $L^{2p}(\Omega;(\mathcal{E}_{K,\al},d_\al))$ (see Remark \ref{rk:d-al}). Its limit provides us with the desired $K$-rough path $\widehat{\mathbf{B}}$ above $\dot{B}$, and we know in addition that $\mathbb{E}\big[ d_\al(\widehat{\mathbf{B}}^n,\widehat{\mathbf{B}})^{2p}\big] \lesssim 2^{-n\ep}$. The almost-sure convergence of $\widehat{\mathbf{B}}^n$ to $\widehat{\mathbf{B}}$ in $(\mathcal{E}_{K,\al},d_\al)$ follows immediately.

\smallskip

\noindent
As far as the identities (\ref{def:hat-b-order-2})-(\ref{def:hat-b-order-3-1}) are concerned, it now suffices to write, for every $n,k$ and $x\in R_k$,
\begin{eqnarray*}
\big| \langle \widehat{\mathbf{B}}^{\mathbf{2}}_x-\widetilde{\mathbf{B}}^{\mathbf{2}}_x,\psi \rangle \big| &\leq & \big| \langle \widehat{\mathbf{B}}^{\mathbf{2}}_x-\widehat{\mathbf{B}}^{\mathbf{2},n}_x,\psi \rangle \big|+\big| \langle \widehat{\mathbf{B}}^{\mathbf{2},n}_x-\widetilde{\mathbf{B}}^{\mathbf{2}}_x,\psi \rangle \big|\\
&\lesssim&\lVert \widehat{\mathbf{B}}^{\mathbf{2}}-\widehat{\mathbf{B}}^{\mathbf{2},n}\rVert_{3\al+4;R_k} +\big| \langle \widehat{\mathbf{B}}^{\mathbf{2},n}_x-\widetilde{\mathbf{B}}^{\mathbf{2}}_x,\psi \rangle \big|
\end{eqnarray*}
and use (\ref{rappel-estim-ordre-deux-p}) again while letting $n$ tend to infinity. The same argument obviously holds for $\widehat{\mathbf{B}}^{\mathbf{3.1}}$ and $\widehat{\mathbf{B}}^{\mathbf{3.2}}$ as well, which achieves the proof of the theorem.
 
\end{proof}

The rest of the section is devoted to the proof of (\ref{rappel-estim-ordre-deux})-(\ref{mai-estim-ordre-trois-deux}). Based on decompositions (\ref{terme-simpl-deux})-(\ref{terme-simpl-trois}), the strategy reduces to controlling the convergence of each summand in these formulas. To this end, our arguments will strongly rely on the general deterministic bounds collected in Appendix \ref{append}, and therefore we are rather confident about the fact that the subsequent computations could easily be extended to a more general class of fractional noises.

\subsection{Convergence in the third chaos}

\begin{proposition}\label{prop:conv-third-chaos}
For every $\al \in (-\frac32,-3+2H_1+H_2)$, there exists $\ep>0$ such that for all $n,\ell\geq 0$, $k\geq 1$, $(s,x)\in R_k$ and $\psi\in \cb$, one has
\begin{equation}\label{third-chaos-1}
\mathbb{E}\Big[\Big|\delta^{B^n}\Big( \psi^\ell_{(s,x)} \cdot \widehat{\mathbf{B}}^{\mathbf{2},n}_{(s,x)}\big(K^{(2)}_{(s,x),.}\big) \Big)-\delta^{B}\Big( \psi^\ell_{(s,x)} \cdot \widehat{\mathbf{B}}^{\mathbf{2}}_{(s,x)}\big(K^{(2)}_{(s,x),.}\big) \Big)\Big|^2 \Big] \lesssim k^2 2^{-n\ep} 2^{-2\ell (3\al+4)} 
\end{equation}
and
\begin{equation}\label{third-chaos-2}
\mathbb{E}\Big[\Big|\delta^{B^n}\Big( \psi^\ell_{(s,x)} \cdot \big(\widehat{\mathbf{B}}^{\mathbf{1},n}_{(s,x)}\big(K^{(1)}_{(s,x),.}\big)\big)^2 \Big)-\delta^{B}\Big( \psi^\ell_{(s,x)} \cdot \big(\widehat{\mathbf{B}}^{\mathbf{1}}_{(s,x)}\big(K^{(1)}_{(s,x),.}\big)\big)^2 \Big)\Big|^2 \Big] \lesssim k^2 2^{-n\ep} 2^{-2\ell (3\al+4)} \ ,
\end{equation}
where the proportional constants are independent of $(n,k,\ell)$ and $\psi$.
\end{proposition}

\begin{proof}
Let us set
$$
V^{n,\ell}_{(s,x)}(t,y) := \psi^\ell_{(s,x)}(t,y) \, \widehat{\mathbf{B}}^{\mathbf{2},n}_{(s,x)}\big(K^{(2)}_{(s,x),(t,y)}\big) \quad \text{and} \quad V^{\ell}_{(s,x)}(t,y) := \psi^\ell_{(s,x)}(t,y) \, \widehat{\mathbf{B}}^{\mathbf{2}}_{(s,x)}\big(K^{(2)}_{(s,x),(t,y)}\big) \ .
$$
Then by (\ref{identi-q}) and (\ref{identi-q-n}), and for $\ep >0$ small enough, we have that
\begin{eqnarray}
\lefteqn{\mathbb{E}\Big[\Big|\delta^{B^n}\big(V^{n,\ell}_{(s,x)}  \big)-\delta^{B}\big( V^{\ell}_{(s,x)} \big)\Big|^2 \Big] \ = \ \mathbb{E}\Big[\Big|\delta^{\widehat{W}}\big(\cq^n_{H_1,H_2}\big( V^{n,\ell}_{(s,x)}\big)  \big)-\delta^{\widehat{W}}\big( \cq_{H_1,H_2}\big(V^{\ell}_{(s,x)}\big) \big)\Big|^2 \Big]}\nonumber\\
&=& \mathbb{E}\Big[\big\lVert \cq^n_{H_1,H_2}\big( V^{n,\ell}_{(s,x)}\big)  - \cq_{H_1,H_2}\big(V^{\ell}_{(s,x)}\big) \big\rVert^2_{L^2(\R^2)} \Big]\nonumber\\
&\lesssim& \mathbb{E}\Big[\big\lVert \big\{ \cq^n_{H_1,H_2}-\cq_{H_1,H_2}\big\}\big( V^{n,\ell}_{(s,x)}\big) \big\rVert^2_{L^2(\R^2)} \Big]+\mathbb{E}\Big[\big\lVert \cq_{H_1,H_2}\big( V^{n,\ell}_{(s,x)}-V^{\ell}_{(s,x)}\big) \big\rVert^2_{L^2(\R^2)}  \Big]\nonumber\\
&\lesssim & 2^{-n\ep} \Big\{  \mathbb{E}\Big[\big\lVert \cq_{H_1-\ep,H_2}\big( V^{n,\ell}_{(s,x)}\big) \big\rVert^2_{L^2(\R^2)} \Big]+ \mathbb{E}\Big[\big\lVert\cq_{H_1,H_2-\ep}\big( V^{n,\ell}_{(s,x)}\big) \big\rVert^2_{L^2(\R^2)} \Big] \Big\}\nonumber\\
& & \hspace{4cm}+\mathbb{E}\Big[\big\lVert \cq_{H_1,H_2}\big( V^{n,\ell}_{(s,x)}-V^{\ell}_{(s,x)}\big) \big\rVert^2_{L^2(\R^2)}  \Big] \ .\label{pr-third}
\end{eqnarray}
At this point, let us turn to the estimates of Lemma \ref{lem:sobolev}, and focus first on the situation where $H_1 >\frac12$ and $H_2 < \frac12$. We get in this case that
\begin{eqnarray}\label{proof-third-chaos}
\lefteqn{\big\lVert \cq_{H_1,H_2}\big( V^{n,\ell}_{(s,x)}-V^{\ell}_{(s,x)}\big) \big\rVert^2_{L^2(\R^2)}}\nonumber\\
&\lesssim& 2^{-2\ell (2H_1-\frac52)} \int_{\R} dy \int_{\R} dt\, \bigg( \int_{\R} du \, \frac{P^{n,\ell}_{(s,x)}(u,y)}{|t-u|^{\frac32-H_1}} \bigg)^2\nonumber\\
& &+2^{-2\ell (-3+2H_1+H_2)} \int_{\R^2} \frac{dy dz}{|y-z|^{2-2H_2}} \int_{\R} dt\, \bigg( \int_{\R} du \, \frac{P^{n,\ell}_{(s,x)}(u,y)-P^{n,\ell}_{(s,x)}(u,z)}{|t-u|^{\frac32-H_1}}  \bigg)^2 \ ,\label{proof-third-chaos-1}
\end{eqnarray}
with
$$P^{n,\ell}_{(s,x)}(u,y):=\psi(u,y) \, \big\{\widehat{\mathbf{B}}^{\mathbf{2},n}_{(s,x)}-\widehat{\mathbf{B}}^{\mathbf{2}}_{(s,x)} \big\} \big( K^{(2)}_{(s,x),(s+2^{-2\ell} u,x+2^{-\ell} y)} \big) \ .$$
Now we can combine the two estimates (\ref{estim-x-2-1-diff}) and (\ref{estim-x-2-2-diff}) to retrieve the following bounds:
\begin{equation}\label{p-n-1}
\big| P^{n,\ell}_{(s,x)}(u,y) \big| \lesssim 2^{-\ell(2\al+4)} \cdot C_{n,k}\cdot \1_{\{ |u|\leq 1,|y|\leq 1\}} 
\end{equation}
and 
\begin{equation}\label{p-n-2}
\big| P^{n,\ell}_{(s,x)}(u,y)-P^{n,\ell}_{(s,x)}(u,z) \big| \lesssim 2^{-\ell(2\al+4)} \cdot C_{n,k}\cdot\1_{\{ |u|\leq 1\}} \cdot \big[ \1_{\{ |y|\leq 1,|z|\leq 1\}} \cdot |y-z|^{\al+2}+\1_{\{ |y|\leq 1,|z|\geq 1\}}+\1_{\{ |y|\geq 1,|z|\leq 1\}} \big] \ ,
\end{equation}
where we have set
$$C_{n,k}:=\lVert \widehat{\mathbf{B}}^{\mathbf{2},n}-\widehat{\mathbf{B}}^{\mathbf{2}}\rVert_{2\al+2;R_{k+1}}+\lVert \widehat{\mathbf{B}}^{\mathbf{1},n}-\widehat{\mathbf{B}}^{\mathbf{1}}\rVert_{\al;R_{k+1}} \lVert \widehat{\mathbf{B}}^{\mathbf{1},n} \rVert_{\al;R_{k+1}}+\lVert  \widehat{\mathbf{B}}^{\mathbf{1}}\rVert_{\al;R_{k+1}} \lVert \widehat{\mathbf{B}}^{\mathbf{1},n}-\widehat{\mathbf{B}}^{\mathbf{1}} \rVert_{\al;R_{k+1}} \ .$$
Injecting (\ref{p-n-1})-(\ref{p-n-2}) into (\ref{proof-third-chaos-1}) easily leads us to the estimate
$$\mathbb{E}\Big[\big\lVert \cq_{H_1,H_2}\big( V^{n,\ell}_{(s,x)}-V^{\ell}_{(s,x)}\big) \big\rVert^2_{L^2(\R^2)}  \Big] \lesssim 2^{-2\ell(3\al+4)} \, \mathbb{E}\big[C_{n,k}^2 \big]\lesssim k^22^{-2\ell(3\al+4)}2^{-n\ep} \ ,$$
where we have used (\ref{estim-mom-x-1}) and (\ref{estim-mom-x-2}) to deduce the last inequality. 

\smallskip

\noindent
By following similar arguments (just replace the use of (\ref{estim-x-2-1-diff}) and (\ref{estim-x-2-2-diff}) with (\ref{estim-x-2-1}) and (\ref{estim-x-2-2})), we also get that 
$$\mathbb{E}\Big[\big\lVert \cq_{H_1-\ep,H_2}\big( V^{n,\ell}_{(s,x)}\big) \big\rVert^2_{L^2(\R^2)} \Big]\lesssim k^2 2^{-2\ell(3\al+4)} \quad \text{and}\quad \mathbb{E}\Big[\big\lVert \cq_{H_1,H_2-\ep}\big( V^{n,\ell}_{(s,x)}\big) \big\rVert^2_{L^2(\R^2)} \Big]\lesssim k^2 2^{-2\ell(3\al+4)} \ ,$$
provided $\ep>0$ is picked small enough. Going back to (\ref{pr-third}), we have thus shown (\ref{third-chaos-1}) in the situation where $H_1>\frac12$ and $H_2<\frac12$.

\smallskip

The other situations can be dealt along the same procedure: if $(H_1,H_2) \in (\frac12,1)^2$, resp. $(H_1,H_2)\in (\frac12,1)\times (0,\frac12)$, we lean successively on (\ref{sobolev-2}) and (\ref{estim-x-2-1})-(\ref{estim-x-2-1-diff}), resp. (\ref{sobolev-1}) and (\ref{estim-x-2-2-time})-(\ref{estim-x-2-2-time-diff}). Finally, note that the condition $2H_1+H_2 > \frac32$ rules out the case where $(H_1,H_2)\in (0,\frac12)^2$.

\smallskip

It is then not hard to see that we can mimic the above arguments in order to prove (\ref{third-chaos-2}). In fact, it suffices to replace the above quantities $V^{n,\ell}_{(s,x)}(t,y)$, $V^{\ell}_{(s,x)}(t,y)$ and $P^{n,\ell}_{(s,x)}$ with
$$
V^{n,\ell}_{(s,x)}(t,y) := \psi^\ell_{(s,x)}(t,y) \, \big( \widehat{\mathbf{B}}^{\mathbf{1},n}\big(K^{(1)}_{(s,x),(t,y)}\big)\big)^2 \quad , \quad V^{\ell}_{(s,x)}(t,y) := \psi^\ell_{(s,x)}(t,y) \, \big(\widehat{\mathbf{B}}^{\mathbf{1}}\big(K^{(1)}_{(s,x),(t,y)}\big)\big)^2 \ ,
$$
$$P^{n,\ell}_{(s,x)}(u,y)=\psi(u,y) \, \big\{ \big( \widehat{\mathbf{B}}^{\mathbf{1},n}_{(s,x)}\big(K^{(1)}_{(s,x),(s+2^{-2\ell} u,+2^{-\ell} y)}\big)\big)^2-\big( \widehat{\mathbf{B}}^{\mathbf{1}}_{(s,x)}\big(K^{(1)}_{(s,x),(s+2^{-2\ell} u,+2^{-\ell} y}\big)\big)^2 \big\} \ ,$$
and check with the help of (\ref{estim-x-1-1}) that both estimates (\ref{p-n-1}) and (\ref{p-n-2}) still hold true in this case, by using also the basic relation 
$$K^{(1)}_{(s,x),(t,y)}-K^{(1)}_{(s,x),(u,z)}=K^{(1)}_{(u,z),(t,y)} \ .$$

\end{proof}

\subsection{Convergence in the second chaos}

\begin{proposition}
For every $\al \in (-\frac32,-3+2H_1+H_2)$, there exists $\ep>0$ such that for all $n,\ell\geq 0$, $k\geq 1$, $(s,x)\in R_k$ and $\psi\in \cb$, one has
\begin{equation}\label{second-chaos}
\mathbb{E}\Big[\Big|\delta^{B^n}\Big( \psi^\ell_{(s,x)} \cdot  \widehat{\mathbf{B}}^{\mathbf{1},n}\big(K^{(1)}_{(s,x),.}\big) \Big)-\delta^{B}\Big( \psi^\ell_{(s,x)} \cdot  \widehat{\mathbf{B}}^{\mathbf{1}}\big(K^{(1)}_{(s,x),.}\big) \Big)\Big|^2 \Big] \lesssim k^2 2^{-n\ep} 2^{-2\ell (2\al+2)} 
\end{equation}
where the proportional constant is independent of $(n,k,\ell)$ and $\psi$.
\end{proposition}

\begin{proof}
The argument is very similar as the one in the proof of Proposition \ref{prop:conv-third-chaos}, and therefore we will only focus on the main ideas. In fact, by setting
$$V^{\ell}_{(s,x)}(t,y) := \psi^\ell_{(s,x)}(t,y) \, \widehat{\mathbf{B}}^{\mathbf{1}}\big(K^{(1)}_{(s,x),(t,y)}\big) \ ,$$
observe that if for instance $H_1>\frac12$ and $H_2<\frac12$, one has by (\ref{sobolev-1})
\begin{eqnarray*}
\lefteqn{\big\lVert \cq_{H_1,H_2}\big( V^{\ell}_{(s,x)}-V^{\ell}_{(s,x)}\big) \big\rVert^2_{L^2(\R^2)}}\\
&\lesssim& 2^{-2\ell (2H_1-\frac52)} \int_{\R} dy \int_{\R} dt\, \bigg( \int_{\R} du \, \frac{P^{\ell}_{(s,x)}(u,y)}{|t-u|^{\frac32-H_1}} \bigg)^2\\
& &+2^{-2\ell (-3+2H_1+H_2)} \int_{\R^2} \frac{dy dz}{|y-z|^{2-2H_2}} \int_{\R} dt\, \bigg( \int_{\R} du \, \frac{P^{\ell}_{(s,x)}(u,y)-P^{\ell}_{(s,x)}(u,z)}{|t-u|^{\frac32-H_1}}  \bigg)^2 \ ,
\end{eqnarray*}
with
$$P^{\ell}_{(s,x)}(u,y):=\psi(u,y) \, \widehat{\mathbf{B}}^{\mathbf{1}} \big( K^{(1)}_{(s,x),(s+2^{-2\ell} u,x+2^{-\ell} y)} \big) \ .$$
Besides, thanks to (\ref{estim-x-1-1}), we know that
$$
\big| P^{\ell}_{(s,x)}(u,y) \big| \lesssim 2^{-\ell(\al+2)} \cdot \lVert \widehat{\mathbf{B}}^{\mathbf{1}}\rVert_{\al,R_{k+1}}\cdot \1_{\{ |u|\leq 1,|y|\leq 1\}} 
$$
and 
$$
\big| P^{\ell}_{(s,x)}(u,y)-P^{\ell}_{(s,x)}(u,z) \big| \lesssim 2^{-\ell(\al+2)} \cdot \lVert \widehat{\mathbf{B}}^{\mathbf{1}}\rVert_{\al,R_{k+1}}\cdot\1_{\{ |u|\leq 1\}} \cdot \big[ \1_{\{ |y|\leq 1,|z|\leq 1\}} \cdot |y-z|^{\al+2}+\1_{\{ |y|\leq 1,|z|\geq 1\}}+\1_{\{ |y|\geq 1,|z|\leq 1\}} \big] \ .
$$
With these estimates in mind, we can easily mimic the proof of Proposition \ref{prop:conv-third-chaos} and conclude that, for $\ep >0$ small enough,
\begin{align*}
&\mathbb{E}\Big[\Big|\delta^{B^n}\Big( \psi^\ell_{(s,x)} \cdot  \widehat{\mathbf{B}}^{\mathbf{1},n}\big(K^{(1)}_{(s,x),.}\big) \Big)-\delta^{B}\Big( \psi^\ell_{(s,x)} \cdot  \widehat{\mathbf{B}}^{\mathbf{1}}\big(K^{(1)}_{(s,x),.}\big) \Big)\Big|^2 \Big]\\
& \lesssim 2^{-2\ell(2\al+2)} \big\{\mathbb{E}\big[\lVert \widehat{\mathbf{B}}^{\mathbf{1},n}-\widehat{\mathbf{B}}^{\mathbf{1}}\rVert_{\al,R_{k+1}}^2\big]+2^{-n\ep}\mathbb{E}\big[\lVert \widehat{\mathbf{B}}^{\mathbf{1}}\rVert_{\al,R_{k+1}}^2\big] \big\} \lesssim k^2 2^{-n\ep} 2^{-2\ell (2\al+2)}\ ,
\end{align*}
where we have used (\ref{estim-mom-x-1}) to derive the last inequality.
\end{proof}

\subsection{Convergence in the first chaos: case of $L^{1,n,\dot{B}^n}$}

\begin{proposition}\label{prop-l-x-1}
For every $\al \in (-\frac32,-3+2H_1+H_2)$, there exists $\ep>0$ such that for all $X,Y \in \cac_c^\al(\R^2)$, $n,\ell\geq 0$, $k\geq 1$, $(s,x)\in R_k$ and $\psi\in \cb$, one has
\begin{equation}\label{estim-l-x-1}
\big|\big\{ L^{1,n,X}_{(s,x)}-L^{1,Y}_{(s,x)}\big\}(\psi_{(s,x)}^\ell) \big| \lesssim 2^{-\ell(3\al+4)} \{2^{-n\ep} \lVert X \rVert_{\al;R_{k+1}}+ \lVert X-Y \rVert_{\al;R_{k+1}} \} \ ,
\end{equation}
where the proportional constant is independent of $(n,k,\ell)$ and $\psi$. As a consequence of (\ref{estim-l-x-1}) and (\ref{estim-mom-x-1}), one has in particular 
\begin{equation}\label{estim-l-x-1-appli}
\mathbb{E} \big[\big|\big\{ L^{1,n,\dot{B}^n}_{(s,x)}-L^{1,\dot{B}}_{(s,x)}\big\}(\psi_{(s,x)}^\ell) \big|^2\big] \lesssim k^2 2^{-2\ell(3\al+4)} 2^{-2n\ep}  \ .
\end{equation}
\end{proposition}

\begin{proof}
Let us start with the basic inequality
\begin{equation}\label{split-i-ii}
\big|\big\{ L^{1,n,X}_{(s,x)}-L^{1,Y}_{(s,x)}\big\}(\psi_{(s,x)}^\ell) \big| \lesssim |I^{\ell}_{(s,x)}|+|II^{n,\ell}_{(s,x)}| \ ,
\end{equation}
with
$$
I^{\ell}_{(s,x)}:=
\iint_{\R^2} \frac{d\xi d\eta}{|\xi|^{2H_1-1} |\eta|^{2H_2-1}}\Big|\fouri{K}(\xi,\eta)\iint_{\R^2} dt dy \, e^{-\imath \xi(t-s)} e^{-\imath \eta (y-x)}\psi^\ell_{(s,x)}(t,y) \{X-Y\}\big(K^{(1)}_{(s,x),(t,y)}\big) \Big|
$$
and
$$II^{n,\ell}_{(s,x)}:=\iint_{\R^2 \backslash \cd_n} \frac{d\xi d\eta}{|\xi|^{2H_1-1} |\eta|^{2H_2-1}}\Big|\fouri{K}(\xi,\eta)\iint_{\R^2} dt dy \, e^{-\imath \xi(t-s)} e^{-\imath \eta (y-x)}\psi^\ell_{(s,x)}(t,y) X\big(K^{(1)}_{(s,x),(t,y)}\big)\Big| \ .$$

\

\noindent
\textbf{Estimation of $I^{\ell}_{(s,x)}$.} Using elementary changes of variables, we can write $I^{\ell}_{(s,x)}$ as
\begin{equation}\label{defi-i-n-l}
I^{\ell}_{(s,x)} =\iint_{\R^2} \frac{d\xi d\eta}{|\xi|^{2H_1-1} |\eta|^{2H_2-1}}\fouri{K}(\xi,\eta) \fouri{F^{\ell}_{(s,x)}}(2^{-2\ell} \xi,2^{-\ell}\eta)
\end{equation}
with
\begin{equation}\label{f-n-l}
F^{\ell}_{(s,x)}(t,y):=\psi(t,y)\cdot \big\{X-Y\big\}\big(K^{(1)}_{(s,x),(s+2^{-2\ell}t,x+2^{-\ell}y)}\big) \ .
\end{equation}
At this point, pick $\la_1,\la_2,\al_0$ satisfying the following conditions:
\begin{equation}\label{param-cs-1}
\max(0,4H_1-3)< \la_1 < \min(4H_1-1,1) \quad , \quad \max(0,4H_2-3) < \la_2 <1 \ ,
\end{equation}
and
\begin{equation}\label{param-cs-2}
2\la_1+\la_2=7+4\al_0 \quad , \quad \al<\al_0<-3+2H_1+H_2 \ .
\end{equation}
Given our assumptions on $(H_1,H_2)$, the existence of such a triplet $(\la_1,\la_2,\al_0)$ is indeed easy to check. Now write
\begin{eqnarray}
\lefteqn{|I^{\ell}_{(s,x)}|}\nonumber\\
&\leq& \bigg(\iint_{\R^2} \frac{d\xi d\eta}{|\xi|^{4H_1-2-\la_1} |\eta|^{4H_2-2-\la_2}}|\fouri{K}(\xi,\eta)|^2 \bigg)^{\frac12} \bigg(\iint_{\R^2} \frac{d\xi d\eta}{|\xi|^{\la_1} |\eta|^{\la_2}}|\fouri{F^{\ell}_{(s,x)}}(2^{-2\ell} \xi,2^{-\ell}\eta)|^2 \bigg)^{\frac12} \nonumber\\
&\leq& 2^{-\ell (2\al+2)} \bigg(\iint_{\R^2} \frac{d\xi d\eta}{|\xi|^{4H_1-2-\la_1} |\eta|^{4H_2-2-\la_2}}|\fouri{K}(\xi,\eta)|^2 \bigg)^{\frac12} \bigg(\iint_{\R^2} \frac{d\xi d\eta}{|\xi|^{\la_1} |\eta|^{\la_2}}|\fouri{F^{\ell}_{(s,x)}}(\xi,\eta)|^2 \bigg)^{\frac12} \ .\label{cs-i-n-l}
\end{eqnarray}
Owing to (\ref{param-cs-1}), we know that $4H_1-2-\la_1<1$ and $4H_2-2-\la_2<1$. Besides, using (\ref{param-cs-1})-(\ref{param-cs-2}), it is easy to check that there exists $a,b\in [0,1)$ such that $a+b<1$ and
$$4H_1-2-\la_1+2a>1 \quad , \quad 4H_2-2-\la_2+4b>1 \ .$$
Thanks to Lemma \ref{lem:estim-k-hat}, we are therefore in a position to guarantee that
\begin{equation}\label{fourier-k-finie}
\iint_{\R^2} \frac{d\xi d\eta}{|\xi|^{4H_1-2-\la_1} |\eta|^{4H_2-2-\la_2}}|\fouri{K}(\xi,\eta)|^2 \ < \ \infty \ .
\end{equation}
As far as the second integral in (\ref{cs-i-n-l}) is concerned, we can use Lemma \ref{lem:sobolev} (more precisely the bound (\ref{sobol-cas-plus-grand})) and then (\ref{estim-x-1-1}) to get that
\begin{eqnarray}
\lefteqn{\iint_{\R^2} \frac{d\xi d\eta}{|\xi|^{\la_1} |\eta|^{\la_2}}|\fouri{F^{\ell}_{(s,x)}}(\xi,\eta)|^2}\nonumber\\
 & \lesssim & \iint_{\R^2} ds  dx \bigg( \iint_{\R^2}dt dy \, \frac{F^{\ell}_{(s,x)}(t,y)}{\lln t-s\rrn^{1-\frac{\la_1}{2}} \lln y-x\rrn^{1-\frac{\la_2}{2}}} \bigg)^2\nonumber\\
 & \lesssim &2^{-2\ell (\al+2)} \lVert X-Y\rVert_{\al;R_{k+1}}^2 \iint_{\R^2} ds  dx \bigg( \iint_{\R^2}dt dy \, \frac{|\psi(t,y)|}{\lln t-s\rrn^{1-\frac{\la_1}{2}} \lln y-x\rrn^{1-\frac{\la_2}{2}}} \bigg)^2\nonumber\\
& \lesssim &2^{-2\ell (\al+2)} \lVert X-Y\rVert_{\al;R_{k+1}}^2 \ .\label{second-ingred}
\end{eqnarray}
Going back to (\ref{cs-i-n-l}), we get the desired estimate, namely
$$|I^{\ell}_{(s,x)}| \lesssim 2^{-\ell (3\al+4)} \lVert X-Y\rVert_{\al;R_{k+1}} \ .$$

\

\noindent
\textbf{Estimation of $II^{n,\ell}_{(s,x)}$.} Observe that for any $\ep>0$, one has
\begin{equation}\label{decomp-ii-n-l}
| II^{n,\ell}_{(s,x)}| \lesssim 2^{-2n\ep} \big\{| \overline{II}^{\ell,1}_{(s,x)}|+| \overline{II}^{\ell,2}_{(s,x)}|\big\} \ ,
\end{equation}
where
$$\overline{II}^{\ell,1}_{(s,x)}:=\iint_{\R^2 } \frac{d\xi d\eta}{|\xi|^{2(H_1-\ep)-1} |\eta|^{2H_2-1}}\Big|\fouri{K}(\xi,\eta)\iint_{\R^2} dt dy \, e^{-\imath \xi(t-s)} e^{-\imath \eta (y-x)}\psi(t,y) X\big(K^{(1)}_{(s,x),(t,y)}\big)\Big|$$
and
$$\overline{II}^{\ell,2}_{(s,x)}:=\iint_{\R^2 } \frac{d\xi d\eta}{|\xi|^{2H_1-1} |\eta|^{2(H_2-\ep)-1}}\Big|\fouri{K}(\xi,\eta)\iint_{\R^2} dt dy \, e^{-\imath \xi(t-s)} e^{-\imath \eta (y-x)}\psi(t,y) X\big(K^{(1)}_{(s,x),(t,y)}\big)\Big| \ .$$
Then the estimation of both $\overline{II}^{\ell,1}_{(s,x)}$ and $\overline{II}^{\ell,2}_{(s,x)}$ can of course be done along the very same lines as the estimation of $I^{\ell}_{(s,x)}$ (by picking $\ep>0$ small enough), which leads us to the bound
$$|\overline{II}^{\ell,i}_{(s,x)}|\lesssim 2^{-\ell (3\al+4)} \lVert X\rVert_{\al;R_{k+1}} \ .$$
The proof of our assertion is therefore complete.

\end{proof}

\

\subsection{Convergence in the first chaos: case of $L^{2,n,\dot{B}^n}$}

\begin{proposition}\label{prop-l-2-x}
For every $\al \in (-\frac32,-3+2H_1+H_2)$, there exists $\ep>0$ such that for all $X,Y \in \cac_c^\al(\R^2)$, $n,\ell\geq 0$, $k\geq 1$, $(s,x)\in R_k$ and $\psi\in \cb$, one has
\begin{equation}\label{estim-l-x-2}
\big|\big\{ L^{2,n,X}_{(s,x)}-L^{2,Y}_{(s,x)}\big\}(\psi_{(s,x)}^\ell) \big| \lesssim 2^{-\ell(3\al+4)} \{2^{-n\ep} \lVert X \rVert_{\al;R_{k+1}}+ \lVert X-Y \rVert_{\al;R_{k+1}} \} \ ,
\end{equation}
where the proportional constant is independent of $(n,k,\ell)$ and $\psi$. As a consequence of (\ref{estim-l-x-2}) and (\ref{estim-mom-x-1}), one has in particular 
\begin{equation}\label{estim-l-x-2-appli}
\mathbb{E} \big[\big|\big\{ L^{2,n,\dot{B}^n}_{(s,x)}-L^{2,\dot{B}}_{(s,x)}\big\}(\psi_{(s,x)}^\ell) \big|^2\big] \lesssim k^2 2^{-2\ell(3\al+4)} 2^{-2n\ep}  \ .
\end{equation}
\end{proposition}

\begin{proof} 
It is based on the same strategy as the proof of Proposition \ref{prop-l-x-1}, that is an appropriate Cauchy-Schwartz argument. As above, let us start with the inequality
\begin{equation}\label{split-i-ii-bis}
\big|\big\{ L^{2,n,X}_{(s,x)}-L^{2,Y}_{(s,x)}\big\}(\psi_{(s,x)}^\ell) \big| \lesssim |I^{\ell}_{(s,x)}|+|II^{n,\ell}_{(s,x)}| \ ,
\end{equation}
with
\begin{multline*}
I^{\ell}_{(s,x)}:=\iint_{\R^2} \frac{d\xi d\eta}{|\xi|^{2H_1-1} |\eta|^{2H_2-2}}|\fouri{K}(\xi,\eta)|\cdot \\
 \bigg|\iint_{\R^2} dt dy \, e^{-\imath \xi(t-s)} e^{-\imath \eta (y-x)}(y-x)\psi^\ell_{(s,x)}(t,y) \{X-Y\}\big(K^{(1)}_{(s,x),(t,y)}\big)\bigg| 
\end{multline*}
and
\begin{multline*}
II^{n,\ell}_{(s,x)}:=\iint_{\R^2\backslash \cd_n} \frac{d\xi d\eta}{|\xi|^{2H_1-1} |\eta|^{2H_2-2}}|\fouri{K}(\xi,\eta)|\cdot \\
 \bigg|\iint_{\R^2} dt dy \, e^{-\imath \xi(t-s)} e^{-\imath \eta (y-x)}(y-x)\psi^\ell_{(s,x)}(t,y) X\big(K^{(1)}_{(s,x),(t,y)}\big)\bigg| 
\end{multline*}
We have here used the basic identity $(\fouri{D^{(0,1)}K})(\xi,\eta)=c \, \eta \, \fouri{K}(\xi,\eta)$.

\

\noindent
\textbf{Estimation of $I^{\ell}_{(s,x)}$.} Just as in the proof of Proposition \ref{prop-l-x-1}, write $I^{\ell}_{(s,x)}$ as
\begin{equation}\label{i-n-l-bis}
I^{\ell}_{(s,x)} =2^{-\ell}\iint_{\R^2} \frac{d\xi d\eta}{|\xi|^{2H_1-1} |\eta|^{2H_2-2}}\fouri{K}(\xi,\eta)\, \fouri{G^{\ell}_{(s,x)}}(2^{-2\ell} \xi,2^{-\ell}\eta)
\end{equation}
with
\begin{equation}\label{g-n-l}
G^{\ell}_{(s,x)}(t,y):=\psi(t,y)\cdot y\cdot \big\{X-Y\big\}\big(K^{(1)}_{(s,x),(s+2^{-2\ell}t,x+2^{-\ell}y)}\big) \ .
\end{equation}
Then note that given our assumptions on $(H_1,H_2)$, we can easily find $\la_1,\la_2,\al_0$ satisfying the following conditions:
\begin{equation}\label{param-cs-1-bis}
\max(0,4H_1-3)< \la_1 < \min(4H_1-1,1) \quad , \quad -1< \la_2 <\min(0,4H_2-1) \ ,
\end{equation}
\begin{equation}\label{param-cs-2-bis}
2\la_1+\la_2=5+4\al_0 \quad , \quad \al<\al_0<-3+2H_1+H_2 \ ,
\end{equation}
and write
\begin{eqnarray}
\lefteqn{|I^{\ell}_{(s,x)}|}\nonumber\\
&\leq& 2^{-\ell}\bigg(\iint_{\R^2} \frac{d\xi d\eta}{|\xi|^{4H_1-2-\la_1} |\eta|^{4H_2-4-\la_2}}|\fouri{K}(\xi,\eta)|^2 \bigg)^{\frac12} \bigg(\iint_{\R^2} \frac{d\xi d\eta}{|\xi|^{\la_1} |\eta|^{\la_2}}|\fouri{G^{\ell}_{(s,x)}}(2^{-2\ell} \xi,2^{-\ell}\eta)|^2 \bigg)^{\frac12} \nonumber\\
&\leq& 2^{-\ell (2\al+2)} \bigg(\iint_{\R^2} \frac{d\xi d\eta}{|\xi|^{4H_1-2-\la_1} |\eta|^{4H_2-4-\la_2}}|\fouri{K}(\xi,\eta)|^2 \bigg)^{\frac12} \bigg(\iint_{\R^2} \frac{d\xi d\eta}{|\xi|^{\la_1} |\eta|^{\la_2}}|\fouri{G^{\ell}_{(s,x)}}(\xi,\eta)|^2 \bigg)^{\frac12} \ .\label{cs-i-n-l-bis}
\end{eqnarray}
At this point, observe that $4H_1-2-\la_1<1$ and $4H_2-4-\la_2<1$. Moreover, with (\ref{param-cs-1-bis})-(\ref{param-cs-2-bis}) in mind, it is easy to exhibit $a,b\in [0,1)$ such that $a+b<1$ and
$$4H_1-2-\la_1+2a>1 \quad , \quad 4H_2-4-\la_2+4b>1 \ .$$
Consequently, we can apply Lemma \ref{lem:estim-k-hat} and assert that
$$\iint_{\R^2} \frac{d\xi d\eta}{|\xi|^{4H_1-2-\la_1} |\eta|^{4H_2-4-\la_2}}|\fouri{K}(\xi,\eta)|^2 \ < \ \infty \ .$$
As for the second integral in (\ref{cs-i-n-l-bis}), we have, by Lemma \ref{lem:sobolev} (and more precisely by (\ref{sobolev-1})),
\begin{eqnarray*}
\iint_{\R^2} \frac{d\xi d\eta}{|\xi|^{\la_1} |\eta|^{\la_2}}|\fouri{G^{\ell}_{(s,x)}}(\xi,\eta)|^2
&\lesssim &\int_{\R} dt \int_{\R} dy \bigg( \int_{\R} du \, \frac{G^{\ell}_{(s,x)}(u,y)}{|t-u|^{1 -\frac{\la_1}{2}}} \bigg)^2\\
&&+\iint_{\R^2} \frac{dy dz}{|y-z|^{1-\la_2}}\int_{\R}dt \bigg( \int_{\R} du \, \frac{G^{\ell}_{(s,x)}(u,y)-G^{\ell}_{(s,x)}(u,z)}{|t-u|^{1 -\frac{\la_1}{2}}} \bigg)^2 \ .
\end{eqnarray*}
Using (\ref{estim-x-1-1}), we can here rely on the following bounds:
$$\big|G^{\ell}_{(s,x)}(u,y)\big| \lesssim 2^{-\ell (\al+2)} \lVert X-Y \rVert_{\al;R_{k+1}}\cdot \1_{\{|u|\leq 1,|y|\leq 1\}} $$
and
$$\big| G^{\ell}_{(s,x)}(u,y)-G^{\ell}_{(s,x)}(u,z) \big|\lesssim  2^{-\ell (\al+2)} \lVert X-Y \rVert_{\al;R_{k+1}}\, \1_{\{ |u|\leq 1\}} \cdot \big[ \1_{\{ |y|\leq 1,|z|\leq 1\}} \cdot |y-z|^{\al+2}+\1_{\{ |y|\leq 1,|z|\geq 1\}}+\1_{\{ |y|\geq 1,|z|\leq 1\}} \big]  \ .$$
Combined with fact that $\la_2>-1>-2(\al+2)$, these estimates allow us to conclude that
\begin{equation}\label{second-ingre-bis}
\iint_{\R^2} \frac{d\xi d\eta}{|\xi|^{\la_1} |\eta|^{\la_2}}|\fouri{G^{\ell}_{(s,x)}}(\xi,\eta)|^2 \lesssim 2^{-2\ell (\al+2)} \lVert X-Y \rVert_{\al;R_{k+1}}^2 \ .
\end{equation}
Going back to (\ref{cs-i-n-l-bis}), we have thus proved the desired estimate, namely
$$|I^{\ell}_{(s,x)}| \lesssim 2^{-\ell (3\al+4)} \lVert X-Y \rVert_{\al;R_{k+1}} \ .$$
\

\noindent
\textbf{Estimation of $II^{n,\ell}_{(s,x)}$.} The argument is the same as in the proof of Proposition \ref{prop-l-x-1}: we first bound $|II^{n,\ell}_{(s,x)}|$ in a similar way as in (\ref{decomp-ii-n-l}) and then use the same Cauchy-Schwartz bound (\ref{cs-i-n-l-bis}) (by replacing $H_i$ with $H_i-\ep$, for $\ep>0$ small enough, and $X-Y$ with $X$). This leads us to the expected estimate, namely
$$|II^{n,\ell}_{(s,x)}| \lesssim 2^{-\ell (3\al+4)}2^{-n\ep} \lVert X \rVert_{\al;R_{k+1}} \ .$$

\end{proof}

\

\subsection{Convergence in the first chaos: case of $L^{3,n,\dot{B}^n}$}

\begin{proposition}
For every $\al \in (-\frac32,-3+2H_1+H_2)$, there exists $\ep>0$ such that for all $X,Y \in \cac_c^\al(\R^2)$, $n,\ell\geq 0$, $k \geq 1$, $(s,x)\in R_k$ and $\psi\in \cb$, one has
\begin{equation}\label{estim-l-x-3}
\big|\big\{ L^{3,n,X}_{(s,x)}-L^{3,Y}_{(s,x)}\big\}(\psi_{(s,x)}^\ell) \big| \lesssim 2^{-\ell(3\al+4)} \{2^{-n\ep} \lVert X \rVert_{\al;R_{k+1}}+ \lVert X-Y \rVert_{\al;R_{k+1}} \} \ ,
\end{equation}
where the proportional constant is independent of $(n,k,\ell)$ and $\psi$. As a consequence of (\ref{estim-l-x-3}) and (\ref{estim-mom-x-1}), one has in particular 
\begin{equation}\label{estim-l-x-3-appli}
\mathbb{E} \big[\big|\big\{ L^{3,n,\dot{B}^n}_{(s,x)}-L^{3,\dot{B}}_{(s,x)}\big\}(\psi_{(s,x)}^\ell) \big|^2\big] \lesssim k^2 2^{-2\ell(3\al+4)} 2^{-2n\ep}  \ .
\end{equation}
\end{proposition}

\begin{proof}
Let us set $X^\Delta:=X-Y$. One has
$$\big\{ L^{3,n,X}_{(s,x)}-L^{3,Y}_{(s,x)}\big\}(\psi_{(s,x)}^\ell)=c_2 \,  \big\{I^{\ell}_{(s,x)}+II^{n,\ell}_{(s,x)}\big\} \ ,$$
with
\begin{multline*}
I^{\ell}_{(s,x)}: =\iint_{\R^2} \frac{d\xi d\eta}{|\xi|^{2H_1-1} |\eta|^{2H_2-1}}\cdot\\
\iint_{\R^2} dt dy \, \psi_{(s,x)}^\ell(t,y) \iint_{\R^2} du dz \, e^{-\imath \xi(t-u)} e^{-\imath \eta (y-z)}  X^\Delta\big( K^{(1)}_{(t,y),(u,z)}\big) K^{(2)}_{(s,x),(t,y)}(u,z) \ ,
\end{multline*}
\begin{multline*}
II^{n,\ell}_{(s,x)}:= \iint_{\R^2 \backslash \cd_n} \frac{d\xi d\eta}{|\xi|^{2H_1-1} |\eta|^{2H_2-1}}\cdot\\
\iint_{\R^2} dt dy \, \psi_{(s,x)}^\ell(t,y) \iint_{\R^2} du dz \, e^{-\imath \xi(t-u)} e^{-\imath \eta (y-z)}  X\big( K^{(1)}_{(t,y),(u,z)}\big) K^{(2)}_{(s,x),(t,y)}(u,z) \ .
\end{multline*}
Just as in the proofs of Proposition \ref{prop-l-x-1} and Proposition \ref{prop-l-2-x}, it is easy to see that the estimations of $I^{\ell}_{(s,x)}$ and $II^{n,\ell}_{(s,x)}$ can be done along the very same steps (by relying on a similar bound as in (\ref{split-i-ii}) or (\ref{split-i-ii-bis})), and therefore we only focus on the estimation of $I^{\ell}_{(s,x)}$.

\smallskip

\noindent
First, using elementary changes of variables, let us write the latter expression as 
\begin{equation}\label{i-itilde}
I^{\ell}_{(s,x)}=2^{-\ell (4H_1+2H_2-3)} \tilde{I}^{\ell}_{(s,x)} \ ,
\end{equation}
where
\begin{multline}\label{i-tilde}
\tilde{I}^{\ell}_{(s,x)}: =\iint_{\R^2} \frac{d\xi d\eta}{|\xi|^{2H_1-1} |\eta|^{2H_2-1}}\iint_{\R^2} du dz \, e^{-\imath \xi u} e^{-\imath \eta z}\iint_{\R^2} dt dy \, \psi(t,y)\\
X^\Delta\big( K^{(1)}_{(s+2^{-2\ell}t,x+2^{-\ell}y),(s+2^{-2\ell}(t-u),x+2^{-\ell}(y-z))}\big) K^{(2)}_{(0,0),(2^{-2\ell}t,2^{-\ell}y)}(2^{-2\ell}(t-u),2^{-\ell}(y-z)) \ .
\end{multline}
Now, with expansion (\ref{decompo-k}) in mind, let us decompose $K^{(2)}_{(0,0),(2^{-2\ell}t,2^{-\ell}y)}(2^{-2\ell}(t-u),2^{-\ell}(y-z))$ as
\begin{multline}\label{expansion-k-2}
K^{(2)}_{(0,0),(2^{-2\ell}t,2^{-\ell}y)}(2^{-2\ell}(t-u),2^{-\ell}(y-z))\\
=\sum_{0\leq m< \ell} K^{(2),m}_{(0,0),(2^{-2\ell}t,2^{-\ell}y)}(2^{-2\ell}(t-u),2^{-\ell}(y-z)) + \sum_{m\geq \ell} K^{(2),m}_{(0,0),(2^{-2\ell}t,2^{-\ell}y)}(2^{-2\ell}(t-u),2^{-\ell}(y-z)) \ ,
\end{multline}
where 
\begin{multline*}
K^{(2),m}_{(0,0),(2^{-2\ell}t,2^{-\ell}y)}(2^{-2\ell}(t-u),2^{-\ell}(y-z)):=2^m K_0(2^{2(m-\ell)} u,2^{m-\ell} z)\\
-2^m K_0(2^{2(m-\ell)} (u-t),2^{m-\ell}(z-y))-2^{-\ell}\, y\, 2^{2m} (D^{0,1}K_0)(2^{2(m-\ell)} (u-t),2^{m-\ell}(z-y)) \ .
\end{multline*}
Injecting (\ref{expansion-k-2}) back into (\ref{i-tilde}) gives a decomposition we label as follows:
\begin{equation}\label{decompo-itilde}
\tilde{I}^{\ell}_{(s,x)} =\sum_{0\leq m< \ell} \tilde{I}^{\ell,m}_{(s,x)}+\sum_{m\geq \ell} \tilde{I}^{\ell,m}_{(s,x)} =:A^{\ell}_{(s,x)}+B^{\ell}_{(s,x)} \ .
\end{equation}

\

\noindent
\textbf{Estimation of $A^{\ell}_{(s,x)}$.} Let us here rely on the basic Taylor expansion
\begin{eqnarray*}
\lefteqn{\sum_{0\leq m< \ell} K^{(2),m}_{(0,0),(2^{-2\ell}t,2^{-\ell}y)}(2^{-2\ell}(t-u),2^{-\ell}(y-z))}\\
&=&2^\ell \sum_{0< m\leq \ell} 2^{-3m} t \int_0^1 dr \, (D^{0,1}K_0)(2^{-2m}(u-t+rt),2^{-m}z) \\
& & +2^\ell \sum_{0< m\leq \ell} 2^{-3m} y^2\int_0^1 dr \, r\int_0^1 dr' \, (D^{0,2}K_0)(2^{-2m}(u-t),2^{-m}(z-y+rr'y)) \ ,
\end{eqnarray*}
which gives rise to the decomposition $A^{\ell}_{(s,x)}=A^{1,\ell}_{(s,x)}+A^{2,\ell}_{(s,x)}$, where
\begin{multline}\label{a-1-n-l-1}
A^{1,\ell}_{(s,x)}:=2^\ell \sum_{0< m\leq \ell} 2^{-3m} \int_{\R^2}\frac{d\xi d\eta}{|\xi|^{2H_1-1} |\eta|^{2H_2-1}} \int_{\R^2} du dz \, e^{-\imath \xi u}e^{-\imath \eta z} \int_{\R^2} dt dy \, t \, \psi(t,y)\\
 X^\Delta\big(K^{(1)}_{(s+2^{-2\ell}t,x+2^{-\ell} y),(s+2^{-2\ell}(t-u),x+2^{-\ell} (y-z))} \big)\int_0^1 dr  \, (D^{0,1}K_0)(2^{-2m}(u-t+rt),2^{-m}z)
\end{multline}
and
\begin{multline*}
A^{2,\ell}_{(s,x)}:=2^\ell \sum_{0< m\leq \ell} 2^{-3m} \int_{\R^2}\frac{d\xi d\eta}{|\xi|^{2H_1-1} |\eta|^{2H_2-1}} \int_{\R^2} du dz \, e^{-\imath \xi u}e^{-\imath \eta z} \int_{\R^2} dt dy \, y^2 \, \psi(t,y)\\
 X^\Delta\big(K^{(1)}_{(s+2^{-2\ell}t,x+2^{-\ell} y),(s+2^{-2\ell}(t-u),x+2^{-\ell} (y-z))} \big)\int_0^1 dr \, r\int_0^1 dr' \, (D^{0,2}K_0)(2^{-2m}(u-t),2^{-m}(z-y+rr'y)) \ .
\end{multline*}
Let us first write $A^{1,\ell}_{(s,x)}$ as
\begin{equation}\label{a-1-n-l-2}
A^{1,\ell}_{(s,x)} =2^\ell \sum_{0< m\leq \ell} \int_{\R^2}\frac{d\xi d\eta}{|\xi|^{2H_1-1} |\eta|^{2H_2-1}}A^{1,\ell,m}_{(s,x)} (\xi,\eta) \ .
\end{equation}
For the sake of clarity, we have postponed the estimation of $|A^{1,\ell,m}_{(s,x)} (\xi,\eta)|$ to the subsequent Lemma \ref{lem:a-1-n}. Now, for some small fixed $\ep>0$, pick $a\in (0,2-2H_1)$, $b\in (0,2-2H_2)$ such that $a+b=1-\ep$. Note that this is made possible by the condition $2H_1+H_2\leq 2$, which entails that $(2-2H_1)+(2-2H_2) >1$. Using the bounds (\ref{bound-a-1-n-l-1})-(\ref{bound-a-1-n-l-2}), we derive that
\begin{eqnarray*}
\lefteqn{\big| A^{1,\ell}_{(s,x)} \big|\ \lesssim \ 2^{-\ell (\al+1)} \lVert X^\Delta\rVert_{\al;R_{k+1}}}\\
& & \sum_{0< m \leq \ell} \bigg\{ (2^{m(\al+2)})^\ep \int_{|\xi|\leq 1,|\eta|\leq 1}\frac{d\xi d\eta}{|\xi|^{2H_1-1+a} |\eta|^{2H_2-1+b}} +\int_{|\xi|\leq 1,|\eta|\geq 1}\frac{d\xi d\eta}{|\xi|^{2H_1-1} |\eta|^{2H_2+1}} \\
& &\hspace{2cm}+\int_{|\xi|\geq 1,|\eta|\leq 1}\frac{d\xi d\eta}{|\xi|^{2H_1+1} |\eta|^{2H_2-1}}+\int_{|\xi|\geq 1,|\eta|\geq 1}\frac{d\xi d\eta}{|\xi|^{2H_1+1} |\eta|^{2H_2+1}}  \bigg\}\\
&\lesssim & 2^{-\ell (\al+1)} \lVert X^\Delta\rVert_{\al;R_{k+1}}\, 2^{\ell (\al+2)\ep} \ .
\end{eqnarray*}
It is easy to see that the very same arguments apply to $A^{2,\ell}_{(s,x)}$ as well, leading us to the conclusion that
\begin{equation}\label{esti-a-n-l}
\big| A^{\ell}_{(s,x)} \big| \lesssim 2^{-\ell (\al+1)} \lVert X^\Delta\rVert_{\al;R_{k+1}}\, 2^{\ell (\al+2)\ep} \ .
\end{equation}

\

\noindent
\textbf{Estimation of $B^{\ell}_{(s,x)}$.} Decompose $B^{\ell}_{(s,x)}$ as $B^{\ell}_{(s,x)}=B^{1,\ell}_{(s,x)}-B^{2,\ell}_{(s,x)}-B^{3,\ell}_{(s,x)}$, with
\begin{multline*}
B^{1,\ell}_{(s,x)}:=\sum_{m\geq \ell} 2^m \iint_{\R^2} \frac{d\xi d\eta}{|\xi|^{2H_1-1} |\eta|^{2H_2-1}}\iint_{\R^2} du dz \, e^{-\imath \xi u} e^{-\imath \eta z}\iint_{\R^2} dt dy \, \psi(t,y)\\
X^\Delta\big( K^{(1)}_{(s+2^{-2\ell}t,x+2^{-\ell}y),(s+2^{-2\ell}(t-u),x+2^{-\ell}(y-z))}\big) K_0(2^{2(m-\ell)} u,2^{m-\ell} z) \ ,
\end{multline*}
\begin{multline*}
B^{2,\ell}_{(s,x)}:=\sum_{m\geq \ell} 2^m \iint_{\R^2} \frac{d\xi d\eta}{|\xi|^{2H_1-1} |\eta|^{2H_2-1}}\iint_{\R^2} du dz \, e^{-\imath \xi u} e^{-\imath \eta z}\iint_{\R^2} dt dy \, \psi(t,y)\\
X^\Delta\big( K^{(1)}_{(s+2^{-2\ell}t,x+2^{-\ell}y),(s+2^{-2\ell}(t-u),x+2^{-\ell}(y-z))}\big) K_0(2^{2(m-\ell)} (u-t),2^{m-\ell}(z-y)) \ ,
\end{multline*}
\begin{multline*}
B^{3,\ell}_{(s,x)}:=2^{-\ell}\sum_{m\geq \ell} 2^{2m} \iint_{\R^2} \frac{d\xi d\eta}{|\xi|^{2H_1-1} |\eta|^{2H_2-1}}\iint_{\R^2} du dz \, e^{-\imath \xi u} e^{-\imath \eta z}\iint_{\R^2} dt dy \, \psi(t,y)\\
X^\Delta\big( K^{(1)}_{(s+2^{-2\ell}t,x+2^{-\ell}y),(s+2^{-2\ell}(t-u),x+2^{-\ell}(y-z))}\big)\cdot y\cdot (D^{0,1}K_0)(2^{2(m-\ell)} (u-t),2^{m-\ell}(z-y)) \ .
\end{multline*}

\

\noindent
(i) \textit{Estimation of $B^{1,\ell}_{(s,x)}$.} One has
\begin{equation}\label{b-1-n-l}
B^{1,\ell}_{(s,x)}=2^\ell \sum_{m\geq 0} \iint_{\R^2} \frac{d\xi d\eta}{|\xi|^{2H_1-1} |\eta|^{2H_2-1}} B^{1,\ell,m}_{(s,x)}(\xi,\eta) \ ,
\end{equation}
where
$$B^{1,\ell,m}_{(s,x)}(\xi,\eta):=2^m \iint_{\R^2} du dz\, e^{-\imath \xi u} e^{-\imath \eta z} K_0(2^{2m}u,2^mz) N^{\ell}_{(s,x)}(u,z) \ ,$$
with
\begin{equation}\label{n-l-l}
N^{\ell}_{(s,x)}(u,z):=\iint_{\R^2} dt dy \, \psi(t,y)
X^\Delta\big( K^{(1)}_{(s+2^{-2\ell}t,x+2^{-\ell}y),(s+2^{-2\ell}(t-u),x+2^{-\ell}(y-z))}\big) \ .
\end{equation}
Using the subsequent Lemma \ref{lem:n-n-l}, it is easy to derive the following bound:
\begin{equation}\label{bound-b-1-n-l-k}
\big| B^{1,\ell,m}_{(s,x)}(\xi,\eta) \big| \lesssim 2^{-\ell (\al+2)} \lVert X^\Delta\rVert_{\al;R_{k+1}}\cdot \inf \bigg( 2^{-3m},\frac{2^m}{\xi^2},\frac{2^m}{\eta^4} \bigg) \ ,
\end{equation}
where the proportional constant is independent from $(\xi,\eta)$ and $(\ell,m)$. Observe indeed that
\begin{eqnarray*}
\big| B^{1,\ell,m}_{(s,x)}(\xi,\eta) \big| &\lesssim & 2^{-2m} \int_{\R^2} du dz \, |K_0(u,z)| \, |N_{(s,x)}^\ell(2^{-2m}u,2^{-m}z)| \\
&\lesssim & 2^{-3m} 2^{-\ell (\al+2)} \lVert X^\Delta\rVert_{\al;R_{k+1}}\int_{\R^2} du dz \, |K_0(u,z)| \{|u|+|z|\} \ ,
\end{eqnarray*}
while, for instance,
\begin{eqnarray*}
\big| B^{1,\ell,m}_{(s,x)}(\xi,\eta) \big| &\lesssim & \frac{2^m}{\xi^2} \int_{\R^2} du dz \, \big| \partial_u^2 \big( K_0(2^{2m}.,2^m.) N^{\ell}_{(s,x)}(.,.)\big)(u,z) \big|\\
&\lesssim& \frac{2^m}{\xi^2}2^{-\ell (\al+2)} \lVert X^\Delta\rVert_{\al;R_{k+1}}  \bigg\{ 2^{4m} \int_{\R^2} du dz \, \big| (D^{2,0} K_0)(2^{2m}u,2^m z)\big| \big\{|u|+|z|\big\}\\
& &+ 2^{2m} \int_{\R^2} du dz \, \big| (D^{1,0} K_0)(2^{2m}u,2^m z)\big|+ \int_{\R^2} du dz \, \big| K_0(2^{2m}u,2^m z)\big| \bigg\}\\
&\lesssim &\frac{2^m}{\xi^2}2^{-\ell (\al+2)} \lVert X^\Delta\rVert_{\al;R_{k+1}}  \ .
\end{eqnarray*}
Going back to (\ref{b-1-n-l}), the estimate (\ref{bound-b-1-n-l-k}) allows us to assert that
\begin{equation}\label{estim-b-1}
\big| B^{1,\ell}_{(s,x)} \big| \lesssim 2^{-\ell (\al+1)} \lVert X^\Delta\rVert_{\al;R_{k+1}} \ . 
\end{equation}
Indeed, observe that we can pick $a\in (1-H_1,\frac34)$, $b\in (\frac12 (1-H_2),\frac34)$ such that $a+b < \frac34$ (due to $2H_1+H_2 > \frac32$). Then write
\begin{eqnarray*}
\lefteqn{\big| B^{1,\ell}_{(s,x)} \big| \ \lesssim\ 2^{-\ell (\al+1)} \lVert X^\Delta\rVert_{\al;R_{k+1}} \cdot}\\
& & \sum_{m\geq 0} \bigg\{ 2^{-3m} \int_{|\xi| \leq 1,|\eta| \leq 1} \frac{d\xi d\eta}{|\xi|^{2H_1-1} |\eta|^{2H_2-1}}+2^{-m(3-4a)} \int_{|\xi|\geq 1,|\eta| \leq 1} \frac{d\xi d\eta}{|\xi|^{2H_1-1+2a}|\eta|^{2H_2-1}}\\
& &+2^{-m(3-4b)} \int_{|\xi|\leq 1,|\eta| \geq 1} \frac{d\xi d\eta}{|\xi|^{2H_1-1}|\eta|^{2H_2-1+4b}}+2^{-m(3-4(a+b))} \int_{|\xi|\geq 1,|\eta| \geq 1} \frac{d\xi d\eta}{|\xi|^{2H_1-1+2a}|\eta|^{2H_2-1+4b}}  \bigg\}\\
&\lesssim& 2^{-\ell (\al+1)} \lVert X^\Delta\rVert_{\al;R_{k+1}} \ .
\end{eqnarray*}

\

\noindent
(ii) \textit{Estimation of $B^{2,\ell}_{(s,x)}$.} Using basic changes of variables, we get that
$$B^{2,\ell}_{(s,x)}=\int_{\R^2}\frac{d\xi d\eta}{|\xi|^{2H_1-1} |\eta|^{2H_2-1}} B^{2,\ell}_{(s,x)}(\xi,\eta) \ ,$$
with
\begin{eqnarray*}
\lefteqn{B^{2,\ell}_{(s,x)}(\xi,\eta)}\\
&:= & 2^\ell \int_{\R^2} dt dy \, e^{-\imath \xi t}e^{-\imath \eta y} \psi(t,y) \int_{\R^2}du dz \, e^{-\imath \xi u}e^{-\imath \eta z}K(u,z) X^\Delta\big( K^{(1)}_{(s+2^{-2\ell}t,x+2^{-\ell}y),(s-2^{-2\ell}u,x-2^{-\ell}z)}\big)\\
&=&2^\ell\, \fouri{\psi}(\xi,\eta) \int_{\R^2}du dz \, e^{-\imath \xi u}e^{-\imath \eta z}K(u,z) X^\Delta\big( K^{(1)}_{(s,x),(s-2^{-2\ell}u,x-2^{-\ell}z)}\big)\\
& &-2^\ell \, \fouri{K}(\xi,\eta)  \int_{\R^2} dt dy \, e^{-\imath \xi t}e^{-\imath \eta y} \psi(t,y) X^\Delta\big( K^{(1)}_{(s,x),(s+2^{-2\ell}t,x+2^{-\ell}y)}\big) \\
& =: & B^{2,1,\ell}_{(s,x)}(\xi,\eta)-B^{2,2,\ell}_{(s,x)}(\xi,\eta) \ .
\end{eqnarray*}
Now, on the one hand, it holds by (\ref{estim-x-1-1}) that
$$|B^{2,1,\ell}_{(s,x)}(\xi,\eta)| \lesssim 2^{-\ell (\al+1)} \lVert X^\Delta\rVert_{\al;R_{k+1}} | \fouri{\psi}(\xi,\eta)| \int_{\R^2} du dz \, |K(u,z)| \lesssim 2^{-\ell (\al+1)} \lVert X^\Delta\rVert_{\al;R_{k+1}} | \fouri{\psi}(\xi,\eta)| \ ,$$
and accordingly
\begin{equation}\label{estim-b-2-1}
\int_{\R^2}\frac{d\xi d\eta}{|\xi|^{2H_1-1} |\eta|^{2H_2-1}} \big| B^{2,1,\ell}_{(s,x)}(\xi,\eta) \big| \lesssim 2^{-\ell (\al+1)} \lVert X^\Delta\rVert_{\al;R_{k+1}} \ .
\end{equation}
On the other hand, with the notations of the proof of Proposition \ref{prop-l-x-1} (see (\ref{f-n-l})), we have that
$$B^{2,2,\ell}_{(s,x)}(\xi,\eta)=2^\ell \, \fouri{K}(\xi,\eta) \, \fouri{F^{\ell}_{(s,x)}}(\xi,\eta) \ .$$
Therefore, we are dealing with a very similar quantity as the one defined by (\ref{defi-i-n-l}), and we can rely on the same Cauchy-Schwarz argument (that is, the combination of (\ref{cs-i-n-l}) and (\ref{second-ingred})) to assert that
\begin{equation}\label{estim-b-2-2}
\int_{\R^2}\frac{d\xi d\eta}{|\xi|^{2H_1-1} |\eta|^{2H_2-1}} \big| B^{2,2,\ell}_{(s,x)}(\xi,\eta) \big| \lesssim 2^{-\ell (\al+1)} \lVert X^\Delta\rVert_{\al;R_{k+1}} \ .
\end{equation}

\

\noindent
(ii) \textit{Estimation of $B^{3,\ell}_{(s,x)}$.} In the same spirit as with $B^{2,\ell}_{(s,x)}$, write, with the additional notation $\vp(t,y):=\psi(t,y) \cdot y$,
\begin{eqnarray*}
\lefteqn{B^{3,\ell}_{(s,x)}(\xi,\eta)}\\
&= & 2^\ell \int_{\R^2} dt dy \, e^{-\imath \xi t}e^{-\imath \eta y} \vp(t,y) \int_{\R^2}du dz \, e^{-\imath \xi u}e^{-\imath \eta z}(D^{0,1}K)(u,z) X^\Delta\big( K^{(1)}_{(s+2^{-2\ell}t,x+2^{-\ell}y),(s-2^{-2\ell}u,x-2^{-\ell}z)}\big)\\
&=&2^\ell\, \fouri{\vp}(\xi,\eta) \int_{\R^2}du dz \, e^{-\imath \xi u}e^{-\imath \eta z}(D^{0,1}K)(u,z) X^\Delta\big( K^{(1)}_{(s,x),(s-2^{-2\ell}u,x-2^{-\ell}z)}\big)\\
& &-2^\ell \, \fouri{(D^{0,1}K)}(\xi,\eta)  \int_{\R^2} dt dy \, e^{-\imath \xi t}e^{-\imath \eta y} \vp(t,y) X^\Delta\big( K^{(1)}_{(s,x),(s+2^{-2\ell}t,x+2^{-\ell}y)}\big) \\
& =: & B^{3,1,\ell}_{(s,x)}(\xi,\eta)-B^{3,2,\ell}_{(s,x)}(\xi,\eta) \ .
\end{eqnarray*}
Now,
$$\big|B^{3,1,\ell}_{(s,x)}(\xi,\eta)\big| \lesssim 2^{-\ell (\al+1)} \lVert X^\Delta\rVert_{\al;R_{k+1}} | \fouri{\vp}(\xi,\eta)| \int_{\R^2} du dz \, |(D^{0,1}K)(u,z)| \lesssim 2^{-\ell (\al+1)} \lVert X^\Delta\rVert_{\al;R_{k+1}} | \fouri{\vp}(\xi,\eta)| \ ,$$
and so
\begin{equation}\label{estim-b-3-1}
\int_{\R^2}\frac{d\xi d\eta}{|\xi|^{2H_1-1} |\eta|^{2H_2-1}} \big| B^{3,1,\ell}_{(s,x)}(\xi,\eta) \big| \lesssim 2^{-\ell (\al+1)} \lVert X^\Delta\rVert_{\al;R_{k+1}} \ .
\end{equation}
Besides, with the notations of the proof of Proposition \ref{prop-l-2-x} (see (\ref{g-n-l})), one has
$$B^{3,2,\ell}_{(s,x)}(\xi,\eta)=2^\ell \, \fouri{(D^{0,1}K)}(\xi,\eta) \, \fouri{G^{\ell}_{(s,x)}}(\xi,\eta) \ .$$
So, just as above, we can use the same argument as with the estimation of (\ref{i-n-l-bis}) (that is, the bounds (\ref{cs-i-n-l-bis}) and (\ref{second-ingre-bis})) to get that
\begin{equation}\label{estim-b-3-2}
\int_{\R^2}\frac{d\xi d\eta}{|\xi|^{2H_1-1} |\eta|^{2H_2-1}} \big| B^{3,2,\ell}_{(s,x)}(\xi,\eta) \big| \lesssim 2^{-\ell (\al+1)} \lVert X^\Delta\rVert_{\al;R_{k+1}} \ .
\end{equation}

\smallskip

\noindent
We can then combine the estimates (\ref{estim-b-1})-(\ref{estim-b-2-1})-(\ref{estim-b-2-2})-(\ref{estim-b-3-1})-(\ref{estim-b-3-2}) to conclude that
\begin{equation}\label{esti-b-n-l}
|B^{\ell}_{(s,x)}| \lesssim 2^{-\ell (\al+1)} \lVert X^\Delta\rVert_{\al;R_{k+1}} \ .
\end{equation}

\smallskip

\noindent
The desired bound for $I^{\ell}_{(s,x)}$ is now immediate: injecting (\ref{esti-a-n-l}) and (\ref{esti-b-n-l}) into (\ref{i-itilde}) and (\ref{decompo-itilde}) leads us to
$$\big|I^{\ell}_{(s,x)}\big| \lesssim \lVert X^\Delta\rVert_{\al;R_{k+1}} \, 2^{-\ell (\al+1)} 2^{-\ell (4H_1+2H_2-3-\ep(\al+2))} \ .$$
and by choosing $\ep>0$ small enough, we can finally assert that
$$\big|I^{\ell}_{(s,x)}\big| \lesssim \lVert X^\Delta\rVert_{\al;R_{k+1}} \, 2^{-\ell (3\al+4)} \ .$$

\end{proof}

\begin{lemma}\label{lem:a-1-n}
Let $A^{1,\ell,m}_{(s,x)} (\xi,\eta)$ be the quantity defined through (\ref{a-1-n-l-1})-(\ref{a-1-n-l-2}) (note in particular that $0< m\leq \ell$). Then it holds that
\begin{equation}\label{bound-a-1-n-l-1}
\big| A^{1,\ell,m}_{(s,x)} (\xi,\eta) \big| \lesssim 2^{-(\ell-m) (\al+2)} \lVert X^\Delta\rVert_{\al;R_{k+1}}  \ ,
\end{equation}
and for every $p,q\geq 1$,
\begin{equation}\label{bound-a-1-n-l-2}
\big| A^{1,\ell,m}_{(s,x)} (\xi,\eta) \big| \lesssim 2^{-\ell (\al+2)} \lVert X^\Delta\rVert_{\al;R_{k+1}} \, \inf \bigg( \frac{1}{|\xi|^p},\frac{1}{|\eta|^q}\bigg) \ .
\end{equation}
In both (\ref{bound-a-1-n-l-1}) and (\ref{bound-a-1-n-l-2}), the proportional constant does not depend on $\xi,\eta,\ell,m,k$ and $\psi$.
\end{lemma}

\begin{proof}
Set $\vp(t,y):=t\cdot \psi(t,y)$. For (\ref{bound-a-1-n-l-1}), observe that due to both (\ref{estim-x-1-1}) and support reasons (involving the supports of $\vp$ and $K$), one has
\begin{eqnarray}
\big| A^{1,\ell,m}_{(s,x)} (\xi,\eta) \big|&\lesssim &2^{-\ell (\al+2)}\lVert X^\Delta\rVert_{\al;R_{k+1}}\, 2^{-3m} \int_{|u|\lesssim 2^{2m}}du \int_{|z|\lesssim 2^m} dz \int_{\R^2} dt dy \, |\vp(t,y)| \, \{|u|+z^2\}^{\frac12 (\al+2)}\nonumber\\
& &\hspace{5cm} \int_0^1 dr \,  \big|(D^{0,1}K_0)(2^{-2m}(u-t+rt),2^{-m}z)\big|\nonumber\\
&\lesssim & 2^{-\ell (\al+2)}\lVert X^\Delta\rVert_{\al;R_{k+1}}\, 2^{m(\al+2)} \ .\label{support-arg}
\end{eqnarray}
As far as (\ref{bound-a-1-n-l-2}) is concerned, one has for instance, by setting $\theta^{\Delta}:=K\ast X^\Delta$,
\begin{multline*}
\big| A^{1,\ell,m}_{(s,x)} (\xi,\eta) \big| \lesssim \frac{2^{-3m}}{|\xi|} \int_{\R^2} du dz \int_0^1 dr \bigg| \int_{\R^2} dt dy \, \vp(t,y)\, \partial_u \Big( (D^{1,0}K_0)(2^{-2m}(.-t+rt),2^{-m}z) \\ \big\{ \theta^{\Delta}(s+2^{-2\ell} (t-.),x+2^{-\ell}(y-z))-\theta^{\Delta}(s+2^{-2\ell} t,x+2^{-\ell}y)\big\} \Big)(u,z) \bigg| \ .
\end{multline*}
Then note that the integral with respect to $(t,y)$ can be decomposed as
\begin{align*}
&2^{-2m} \int_{\R^2} dtdy \, \vp(t,y) \, (D^{2,0}K_0)(2^{-2m}(u-t+rt),2^{-m}z) \\
&\hspace{2cm} \big\{ \theta^{\Delta}(s+2^{-2\ell} (t-u),x+2^{-\ell}(y-z))-\theta^{\Delta}(s+2^{-2\ell} t,x+2^{-\ell}y)\big\}\\
& -\int_{\R^2} dtdy \, \vp(t,y) \, (D^{1,0}K_0)(2^{-2m}(u-t+rt),2^{-m}z) \\
&\hspace{2cm} \partial_t \big\{ \theta^{\Delta}(s+2^{-2\ell} (.-u),x+2^{-\ell}(y-z))-\theta^{\Delta}(s-2^{-2\ell} u,x+2^{-\ell}z)\big\}(t,y)\\
&=\ 2^{-2m} \int_{\R^2} dtdy \, \vp(t,y) \, (D^{2,0}K_0)(2^{-2m}(u-t+rt),2^{-m}z) \\
&\hspace{2cm} \big\{ \theta^{\Delta}(s+2^{-2\ell} (t-u),x+2^{-\ell}(y-z))-\theta^{\Delta}(s+2^{-2\ell} t,x+2^{-\ell}y)\big\}\\
& +\int_{\R^2} dtdy \, (\partial_t \vp)(t,y) \, (D^{1,0}K_0)(2^{-2m}(u-t+rt),2^{-m}z) \\
&\hspace{2cm}  \big\{ \theta^{\Delta}(s+2^{-2\ell} (t-u),x+2^{-\ell}(y-z))-\theta^{\Delta}(s-2^{-2\ell} u,x+2^{-\ell}z)\big\}\\
& +2^{-2m}(r-1)\int_{\R^2} dtdy \, \vp(t,y) \, (D^{2,0}K_0)(2^{-2m}(u-t+rt),2^{-m}z) \\
&\hspace{2cm}  \big\{ \theta^{\Delta}(s+2^{-2\ell} (t-u),x+2^{-\ell}(y-z))-\theta^{\Delta}(s-2^{-2\ell} u,x+2^{-\ell}z)\big\} \,\\
\end{align*}
which, by (\ref{estim-x-1-1}) and for the same support reasons as in (\ref{support-arg}), easily leads us to the expected bound:
$$\big| A^{1,\ell,m}_{(s,x)} (\xi,\eta) \big| \lesssim \frac{2^{-\ell (\al+2)}}{|\xi|} \lVert X^\Delta\rVert_{\al;R_{k+1}} \ .$$
The other estimates contained in (\ref{bound-a-1-n-l-2}) can then be derived from the same integration-by-parts arguments.

\end{proof}

\begin{lemma}\label{lem:n-n-l}
Let $N^{\ell}_{(s,x)}$ be the function defined by (\ref{n-l-l}). Then for every $p,q \geq 0$ such that $1\leq p+q\leq 4$, one has
$$\sup_{(u,z)\in [-1,1]^2} \big| \partial^p_u \partial_z^q\big( N^{\ell}_{(s,x)}\big)(u,z)\big| \lesssim 2^{-\ell (\al+2)} \lVert X^\Delta\rVert_{\al;R_{k+1}} \ ,$$
where the proportional constant does not depend on $\ell,k$ and $\psi$. In particular, as $ N^{\ell}_{(s,x)}(0,0)=0$, one has for every $(u,z)\in [-1,1]^2$,
$$\big| N^{\ell}_{(s,x)}(u,z)\big| \lesssim 2^{-\ell (\al+2)} \lVert X^\Delta\rVert_{\al;R_{k+1}}\, \{|u|+|z| \}  \ .$$
\end{lemma}

\begin{proof}
Set $\theta^{\Delta}:=K\ast X^\Delta$. Then we have
\begin{eqnarray*}
\lefteqn{\big| \partial^p_u \partial_z^q\big( N^{\ell}_{(s,x)}\big)(u,z)\big|}\\
&=& \bigg|\int_{\R^2} dt dy \, \psi(t,y)\, \partial_u^p\partial_z^q \big( \theta^{\Delta}(s+2^{-2\ell}(t-.),x+2^{-\ell} (y-.)-\theta^{\Delta}(s+2^{-2\ell}t,x+2^{-\ell} y)\big)(u,z)\bigg|\\
&=& \bigg| \int_{\R^2} dt dy \, \psi(t,y)\, \partial_t^p\partial_y^q \big( \theta^{\Delta}(s+2^{-2\ell}(.-u),x+2^{-\ell} (.-z)-\theta^{\Delta}(s-2^{-2\ell}u,x-2^{-\ell} z)\big)(t,y)\bigg|\\
&=& \bigg| \int_{\R^2} dt dy \, (\partial_t^p\partial_y^q \psi)(t,y)\,  \big\{ \theta^{\Delta}(s+2^{-2\ell}(t-u),x+2^{-\ell} (y-z))-\theta^{\Delta}(s-2^{-2\ell}u,x-2^{-\ell} z)\big\}\bigg|
\end{eqnarray*}
and so, by (\ref{estim-x-1-1}),
$$
\sup_{(u,z)\in [-1,1]^2}  \big| \partial^p_u \partial_z^q\big( N^{\ell}_{(s,x)}\big)(u,z)\big| \lesssim  2^{-\ell (\al+2)} \lVert \theta^{\Delta}\rVert_{\al+2;R_{k+1}} \int_{\R^2} dt dy \, \big| (\partial_t^p \partial_y^q \psi)(t,y)\big| \lesssim  2^{-\ell (\al+2)} \lVert X^\Delta\rVert_{\al;R_{k+1}} \ . 
$$
\end{proof}

\subsection{Convergence in the deterministic chaos}

\begin{proposition}
For every $\al \in (-\frac32,-3+2H_1+H_2)$, there exists $\ep>0$ such that for all $n,\ell\geq 0$, $k\geq 1$, $(s,x)\in R_k$ and $\psi\in \cb$, one has
\begin{equation}\label{estim-l}
\big|\big\{ L^{n}_{(s,x)}-L_{(s,x)}\big\}(\psi_{(s,x)}^\ell) \big| \lesssim 2^{-\ell(2\al+2)} 2^{-n\ep} \ ,
\end{equation}
where the proportional constant is independent of $(n,k,\ell)$ and $\psi$. 
\end{proposition}

\begin{proof}
It is again a basic Cauchy-Schwartz argument. In fact, pick $(\la_1,\la_2,\al_0)$ exactly as in (\ref{param-cs-1})-(\ref{param-cs-2}) and observe that with this choice, one has
\begin{eqnarray*}
\lefteqn{\big|L_{(s,x)}(\psi_{(s,x)}^\ell) \big|}\\
&\leq& \bigg(\iint_{\R^2} \frac{d\xi d\eta}{|\xi|^{4H_1-2-\la_1} |\eta|^{4H_2-2-\la_2}}|\fouri{K}(\xi,\eta)|^2 \bigg)^{\frac12} \bigg(\iint_{\R^2} \frac{d\xi d\eta}{|\xi|^{\la_1} |\eta|^{\la_2}}|\fouri{\psi_{(0,0)}^\ell}(\xi,\eta)|^2 \bigg)^{\frac12} \\
&\leq& 2^{-\ell (2\al+2)} \bigg(\iint_{\R^2} \frac{d\xi d\eta}{|\xi|^{4H_1-2-\la_1} |\eta|^{4H_2-2-\la_2}}|\fouri{K}(\xi,\eta)|^2 \bigg)^{\frac12} \bigg(\iint_{\R^2} \frac{d\xi d\eta}{|\xi|^{\la_1} |\eta|^{\la_2}}|\fouri{\psi}(\xi,\eta)|^2 \bigg)^{\frac12} \ .
\end{eqnarray*}
The argument showing that the latter bound is finite is then the same as for (\ref{fourier-k-finie}). With these estimations in mind, the proof of (\ref{estim-l}) (for $\ep >0$ small enough) is immediate. 
\end{proof}

\

\appendix
\section{A few deterministic estimates at first and second orders}\label{append}

We collect here a few useful estimates related to the interactions between a given $K$-rough path (more precisely, its first and second-order components) and a localized heat kernel (in the sense of Definition \ref{defi:loc-heat-ker}). Just as in Section \ref{sec:main-results}, we fix a coefficient $\al\in (-\frac32,-\frac43)$, as well as a localized heat kernel $K$. We recall that the notations $K^{(1)}$ and $K^{(2)}$ are then defined through (\ref{defi-k-un}) and (\ref{defi-k-deux}), respectively. Also, for every $R>0$, we set $I_R:=[-R,R]^2$.

\smallskip

Let us start with a basic result applying to the first-order component of a $K$-rough path, that is to a general $\cac^\al(\R^2)$ distribution (the proof of this estimate can be found for instance in \cite[Lemma 2.2]{deya}).
\begin{lemma}
Let $X \in \cac_c^\al(\R^2)$. Then, for every $R>0$ and every $x=(x_0,x_1),y=(y_0,y_1)\in I_R$ such that $x-y\in [-2,2]^2$, one has
\begin{equation}\label{estim-x-1-1}
\big| X\big(K^{(1)}_{x,y}\big)\big| \lesssim \lVert X\rVert_{\al;I_R}  \cdot \lVert x-y\rVert_{\scal}^{\al+2} \ .
\end{equation}
\end{lemma}

\smallskip

Let us now turn to second-order considerations and fix two paths
$$\mathbf{X}=(\mathbf{X}^{\mathbf{1}},\mathbf{X}^{\mathbf{2}}) \ , \ \mathbf{Y}=(\mathbf{Y}^{\mathbf{1}},\mathbf{Y}^{\mathbf{2}}) \ \in \cac_c^\al(\R^2) \times \pmb{\cac}_c^{2\al+2}(\R^2) $$
that satisfy the $K$-Chen relation (\ref{k-chen-relation}).

\begin{lemma}\label{lem:x-2-1}
For every $R>0$ and every $x=(x_0,x_1),y=(y_0,y_1)\in I_R$ such that $x-y\in [-2,2]^2$, one has
\begin{equation}\label{estim-x-2-1}
\big| \mathbf{X}_{x}^{\mathbf{2}}\big(K^{(2)}_{x,y}\big)\big| \lesssim \big\{ \lVert \mathbf{X}^{\mathbf{2}}\rVert_{2\al+2;I_R} +\lVert \mathbf{X}^{\mathbf{1}} \rVert_{\al;I_R}^2\big\} \cdot \lVert x-y\rVert_{\scal}^{2\al+4} \ ,
\end{equation}
and
\begin{multline}\label{estim-x-2-1-diff}
\big| \big\{ \mathbf{X}_{x}^{\mathbf{2}}-\mathbf{Y}_{x}^{\mathbf{2}}\big\}\big(K^{(2)}_{x,y}\big)\big|\\
 \lesssim \big\{ \lVert \mathbf{X}^{\mathbf{2}}-\mathbf{Y}^{\mathbf{2}}\rVert_{2\al+2;I_R}+\lVert \mathbf{X}^{\mathbf{1}}-\mathbf{Y}^{\mathbf{1}}\rVert_{\al;I_R} \lVert \mathbf{X}^{\mathbf{1}} \rVert_{\al;I_R}+\lVert \mathbf{Y}^{\mathbf{1}}\rVert_{\al;I_R} \lVert \mathbf{X}^{\mathbf{1}}-\mathbf{Y}^{\mathbf{1}}\rVert_{\al;I_R}\big\} \cdot \lVert x-y\rVert_{\scal}^{2\al+4} \ .
\end{multline}
Moreover,
\begin{equation}\label{estim-x-2-1-1}
\big| \mathbf{X}_{x}^{\mathbf{2}}\big( (D^{0,1}K)(x-.)\big)-\mathbf{X}_{y}^{\mathbf{2}}\big( (D^{0,1}K)(y-.)\big) \big| \lesssim  \big\{ \lVert \mathbf{X}^{\mathbf{2}}\rVert_{2\al+2;I_R} +\lVert \mathbf{X}^{\mathbf{1}} \rVert_{\al;I_R}^2\big\}  \cdot \lVert x-y \rVert_{\scal}^{2\al+3} \ .
\end{equation}
\end{lemma}

\begin{proof}
We follow a similar strategy as in the proof of \cite[Lemma 2.2]{deya}. Namely, with expansion (\ref{decompo-k}) in mind, we write
$$
\mathbf{X}^{\mathbf{2}}_{x}\big(K^{(2)}_{x,y}\big) = \sum_{\ell\geq 0}\mathbf{X}^{\mathbf{2}}_{x}\big(K^{(2),\ell}_{x,y}\big) \ ,
$$
where $K^{(2),\ell}_{x,y}$ stands for the expression obtained by replacing each occurence of $K$ with $K_\ell:=2^\ell K_0(2^{2\ell}.,2^\ell .)$ in $K^{(2)}_{x,y}$. Then pick $i\geq 0$ such that $3 \cdot 2^{-(i+1)} \leq \lVert x-y\rVert_\scal \leq 3\cdot 2^{-i}$. For $\ell>i$, we use (\ref{k-chen-relation}) to derive that
$$
\mathbf{X}^{\mathbf{2}}_{x}\big(K^{(2),\ell}_{x,y}\big)=\Big[ \mathbf{X}_{y}^{\mathbf{2}}\big(K_\ell(y-)\big)+\mathbf{X}^{\mathbf{1}}\big( K^{(1)}_{x,y}\big) \mathbf{X}^{\mathbf{1}}\big( K_\ell(y-.)\big) \Big]
-\mathbf{X}^{\mathbf{2}}_{x}\big( K_\ell(x-.)\big)-(y_1-x_1)\,\mathbf{X}^{\mathbf{2}}_{x}\big( (D^{0,1}K_\ell)(x-.)\big) \ .
$$
Therefore, by (\ref{estim-x-1-1}), we can assert that
$$
\big| \mathbf{X}^{\mathbf{2}}_{x}\big(K^{(2),\ell}_{x,y}\big) \big| \lesssim \big\{ \lVert \mathbf{X}^{\mathbf{2}}\rVert_{2\al+2;I_R} + \lVert \mathbf{X}^{\mathbf{1}} \rVert_{\al;I_R}^2\big\}\,
\big\{ 2^{-\ell(2\al+4)}+ |y_1-x_1|\cdot 2^{-\ell(2\al+3)}+\lVert x-y\rVert_\scal^{\al+2}\cdot 2^{-\ell(\al+2)} \big\} \ ,
$$
and as a result
$$
\sum_{\ell>i} \big| \mathbf{X}^{\mathbf{2}}_{x}\big(K^{(2),\ell}_{x,y}\big) \big| \lesssim \big\{ \lVert \mathbf{X}^{\mathbf{2}}\rVert_{2\al+2;I_R} + \lVert \mathbf{X}^{\mathbf{1}} \rVert_{\al;I_R}^2\big\} \, \lVert x-y\rVert_\scal^{2\al+4} \ .
$$
On the other hand, for $0\leq \ell \leq i$, write
\begin{multline*}
\mathbf{X}^{\mathbf{2}}_{x}\big(K^{(2),\ell}_{x,y}\big)=(y_0-x_0) \, \int_0^1 dr \, \mathbf{X}^{\mathbf{2}}_{x} \big( (D^{1,0}K_\ell)(x_0+r(y_0-x_0)-.,y_1-.)\big) \\
+(y_1-x_1)^2\, \int_0^1 dr \, r\int_0^1 dr' \, \, \mathbf{X}^{\mathbf{2}}_{x} \big( (D^{0,2}K_\ell)(x_0-.,x_1+rr'(y_1-x_1)-.)\big)\ ,
\end{multline*}
and so, using (\ref{k-chen-relation}) and (\ref{estim-x-1-1}) just as above, we get that
\begin{multline*}
\big| \mathbf{X}^{\mathbf{2}}_{x}\big(K^{(2),\ell}_{x,y}\big) \big| \lesssim \big\{ \lVert \mathbf{X}^{\mathbf{2}}\rVert_{2\al+2;I_R} + \lVert \mathbf{X}^{\mathbf{1}} \rVert_{\al;I_R}^2\big\}\cdot\\
\big\{ |y_0-x_0| \cdot 2^{-\ell(2\al+2)}+|y_0-x_0| \cdot \lVert x-y\rVert_\scal^{\al+2} \cdot 2^{-\ell\al}+|y_1-x_1|^2 \cdot 2^{-\ell(2\al+2)}+|y_1-x_1|^{\al+4} \cdot 2^{-\ell\al} \big\} \ .
\end{multline*}
Consequently,
$$
\sum_{0\leq \ell\leq i} \big| \mathbf{X}^{\mathbf{2}}_{x}\big(K^{(2),\ell}_{x,y}\big) \big| \lesssim \big\{ \lVert \mathbf{X}^{\mathbf{2}}\rVert_{2\al+2;I_R} + \lVert \mathbf{X}^{\mathbf{1}} \rVert_{\al;I_R}^2\big\} \, \lVert x-y \rVert_\scal^{2\al+4} \ ,
$$
which achieves the proof of (\ref{estim-x-2-1}). It is then clear that (\ref{estim-x-2-1-diff}) can be shown along the very same lines.

\smallskip

\noindent
The proof of (\ref{estim-x-2-1-1}) also relies on a similar strategy and therefore we omit it for the sake of conciseness. 
\end{proof}

\begin{lemma}\label{lem:x-2-2-time}
For all $R>0$, $x=(x_0,x_1)\in I_R$, $y=(y_0,y_1)\in [-1,1]^2$ and $z_0\in [-1,1]$, it holds that
\begin{multline}\label{estim-x-2-2-time}
\big| \mathbf{X}^{\mathbf{2}}_{x}\big(K^{(2)}_{x,(x_0+y_0,x_1+y_1)}-K^{(2)}_{x,(x_0+z_0,x_1+y_1)}\big)\big|\\
 \lesssim \big\{ \lVert \mathbf{X}^{\mathbf{2}}\rVert_{2\al+2;I_{R+1}} + \lVert \mathbf{X}^{\mathbf{1}} \rVert_{\al;I_{R+1}}^2\big\} \cdot \big\{ |y_0-z_0|^{\al+2}+\lVert y\rVert_\scal^{\al+2} \, |y_0-z_0|^{\frac12 (\al+2)} \big\} \ ,
\end{multline}
and
\begin{multline}\label{estim-x-2-2-time-diff}
\big| \big\{\mathbf{X}_{x}^{\mathbf{2}}-\mathbf{Y}_{x}^{\mathbf{2}}\big\}\big(K^{(2)}_{x,(x_0+y_0,x_1+y_1)}-K^{(2)}_{x,(x_0+z_0,x_1+y_1)}\big)\big|\\
\lesssim \big\{ \lVert \mathbf{X}^{\mathbf{2}}-\mathbf{Y}^{\mathbf{2}}\rVert_{2\al+2;I_{R+1}}+\lVert \mathbf{X}^{\mathbf{1}}-\mathbf{Y}^{\mathbf{1}}\rVert_{\al;I_{R+1}} \lVert \mathbf{X}^{\mathbf{1}} \rVert_{\al;I_{R+1}}+\lVert  \mathbf{Y}^{\mathbf{1}}\rVert_{\al;I_{R+1}} \lVert \mathbf{X}^{\mathbf{1}}-\mathbf{Y}^{\mathbf{1}} \rVert_{\al;I_R}\big\} \\
\cdot \big\{ |y_0-z_0|^{\al+2}+\lVert y\rVert_\scal^{\al+2} \, |y_0-z_0|^{\frac12 (\al+2)} \big\} \ .
\end{multline}
\end{lemma}

\begin{proof}
One has trivially
$$
K^{(2)}_{x,(x_0+y_0,x_1+y_1)}-K^{(2)}_{x,(x_0+z_0,x_1+y_1)}=K^{(2)}_{(x_0+y_0,x_1+y_1),(x_0+z_0,x_1+y_1)}=K^{(1)}_{(x_0+y_0,x_1+y_1),(x_0+z_0,x_1+y_1)} \ ,
$$
and so, by (\ref{k-chen-relation}), we get the identity
\begin{multline*}
\mathbf{X}^{\mathbf{2}}_{x}\big(K^{(2)}_{x,(x_0+y_0,x_1+y_1)}-K^{(2)}_{x,(x_0+z_0,x_1+y_1)}\big)\\
=\mathbf{X}_{(x_0+y_0,x_1+y_1)}^{\mathbf{2}}\big(K^{(2)}_{(x_0+y_0,x_1+y_1),(x_0+z_0,x_1+y_1)}\big)+\mathbf{X}^{\mathbf{1}}\big(K^{(1)}_{x,x+y}\big) \, \mathbf{X}^{\mathbf{1}}\big(K^{(1)}_{(x_0+y_0,x_1+y_1),(x_0+z_0,x_1+y_1)}\big) \ .
\end{multline*}
The desired estimate (\ref{estim-x-2-2-time}) now follows from (\ref{estim-x-1-1}) and (\ref{estim-x-2-1}). The very same arguments can be applied in order to derive (\ref{estim-x-2-2-time-diff}).

\end{proof}

\begin{lemma}\label{lem:x-2-2}
For all $R>0$, $x=(x_0,x_1)\in I_R$, $y=(y_0,y_1)\in [-1,1]^2$ and $z_1\in [-1,1]$, it holds that
\begin{multline}\label{estim-x-2-2}
\big| \mathbf{X}_{x}^{\mathbf{2}}\big(K^{(2)}_{x,(x_0+y_0,x_1+y_1)}-K^{(2)}_{x,(x_0+y_0,x_1+z_1)}\big)\big|\\
 \lesssim \big\{ \lVert \mathbf{X}^{\mathbf{2}}\rVert_{2\al+2;I_{R+1}} + \lVert \mathbf{X}^{\mathbf{1}} \rVert_{\al;I_{R+1}}^2\big\} \cdot \big\{ |y_1-z_1|^{2\al+4}+|y_1-z_1| \cdot \lVert y\rVert_\scal^{2\al+3}+|y_1-z_1|^{\al+2} \cdot \lVert y\rVert_\scal^{\al+2}\big\} \ ,
\end{multline}
and
\begin{multline}\label{estim-x-2-2-diff}
\big| \big\{\mathbf{X}_{x}^{\mathbf{2}}-\mathbf{Y}_{x}^{\mathbf{2}}\big\}\big(K^{(2)}_{x,(x_0+y_0,x_1+y_1)}-K^{(2)}_{x,(x_0+y_0,x_1+z_1)}\big)\big|\\
\lesssim \big\{ \lVert \mathbf{X}^{\mathbf{2}}-\mathbf{Y}^{\mathbf{2}}\rVert_{2\al+2;I_{R+1}}+\lVert \mathbf{X}^{\mathbf{1}}-\mathbf{Y}^{\mathbf{1}}\rVert_{\al;I_{R+1}} \lVert \mathbf{X}^{\mathbf{1}} \rVert_{\al;I_{R+1}}+\lVert  \mathbf{Y}^{\mathbf{1}}\rVert_{\al;I_{R+1}} \lVert \mathbf{X}^{\mathbf{1}}-\mathbf{Y}^{\mathbf{1}} \rVert_{\al;I_{R+1}}\big\} \cdot\\
\big\{ |y_1-z_1|^{2\al+4}+|y_1-z_1| \cdot \lVert y\rVert_\scal^{2\al+3}+|y_1-z_1|^{\al+2} \cdot \lVert y\rVert_\scal^{\al+2}\big\} \ .
\end{multline}

\end{lemma}

\begin{proof}
We will rely on the following readily-checked identity:
\begin{multline*}
\mathbf{X}_{x}^{\mathbf{2}}\big(K^{(2)}_{x,(x_0+y_0,x_1+z_1)}-K^{(2)}_{x,(x_0+y_0,x_1+y_1)}\big)=\mathbf{X}_{x+y}^{\mathbf{2}}\big(K^{(2)}_{x+y,(x_0+y_0,x_1+z_1)}\big)\\
+\mathbf{X}^{\mathbf{1}}\big(K^{(1)}_{x,x+y}\big)\, \mathbf{X}^{\mathbf{1}} \big( K^{(1)}_{x+y,(x_0+y_0,x_1+z_1)}\big)+(z_1-y_1)\cdot V_{x,y}\ ,
\end{multline*}
where
$$
V_{x,y}:= \mathbf{X}^{\mathbf{2}}_{x}\big((D^{0,1}K)(x+y-.)-(D^{0,1}K)(x-.)\big)
-\mathbf{X}^{\mathbf{1}}\big(K^{(1)}_{x,x+y}\big) \, \mathbf{X}^{\mathbf{1}} \big( (D^{0,1}K)(x+y-.)\big) \ .
$$
Using (\ref{estim-x-1-1}), resp. (\ref{estim-x-2-1}), we get that
$$
\big|\mathbf{X}^{\mathbf{1}}\big(K^{(1)}_{x,x+y}\big)\, \mathbf{X}^{\mathbf{1}} \big( K^{(1)}_{x+y,(x_0+y_0,x_1+z_1)}\big)\big| \lesssim \lVert \mathbf{X}^{\mathbf{1}} \rVert_{\al;I_{R+1}}^2\cdot \lVert y\rVert_\scal^{\al+2}\cdot |y_1-z_1|^{\al+2} \ ,
$$
resp.
$$
\big|\mathbf{X}_{x+y}^{\mathbf{2}}\big(K^{(2)}_{x+y,(x_0+y_0,x_1+z_1)}\big) \big| 
\lesssim \big\{ \lVert \mathbf{X}^{\mathbf{2}}\rVert_{2\al+2;I_{R+1}} + \lVert \mathbf{X}^{\mathbf{1}} \rVert_{\al;I_{R+1}}^2\big\} \cdot |y_1-z_1|^{2\al+4} \ .
$$
Therefore, the proof reduces to the estimation of $V_{x,y}$. To this end, and in the same vein as in the proof of Lemma \ref{lem:x-2-1}, pick $i\geq 0$ such that $2^{-i} \leq \lVert y\rVert_\scal \leq 2^{-(i-1)}$. Also, with expansion (\ref{decompo-k}) in mind, denote by $V^{\ell}_{x,y}$ the expression obtained by replacing each occurence of $K$ with $K_\ell:=2^\ell K_0(2^{2\ell}.,2^\ell .)$ in $V_{x,y}$. For $\ell> i$, let us use (\ref{k-chen-relation}) in order to write $V^{\ell}_{x,y}$ as
$$V^{\ell}_{x,y}=\mathbf{X}_{x+y}^{\mathbf{2}}\big((D^{0,1}K_\ell)(x+y-.) \big)-\mathbf{X}^{\mathbf{2}}_{x}\big((D^{0,1}K_\ell)(x-.) \big) \ ,$$
which immediately entails that $\big|V^{\ell}_{x,y} \big| \lesssim \lVert \mathbf{X}^{\mathbf{2}} \rVert_{2\al+2;I_{R+1}}\, 2^{-\ell(2\al+3)}$, and so
\begin{equation}
\sum_{\ell> i} \big|V^{\ell}_{x,y} \big| \lesssim \lVert \mathbf{X}^{\mathbf{2}} \rVert_{2\al+2;I_{R+1}} \, \lVert y\rVert_\scal^{2\al+3} \ .
\end{equation}
For $0\leq \ell\leq i$, observe on the one hand that
\begin{equation}\label{proof-x-2-1}
\big| \mathbf{X}^{\mathbf{1}} \big( (D^{0,1}K_\ell)(x+y-.)\big)\big| \lesssim \lVert \mathbf{X}^{\mathbf{1}}\rVert_{\al;I_{R+1}} \, 2^{-\ell(\al+1)} \ .
\end{equation}
On the other hand, combining basic Taylor expansions with (\ref{k-chen-relation}) gives us that
\begin{eqnarray}
\lefteqn{\big|\mathbf{X}^{\mathbf{2}}_{x}\big((D^{0,1}K_\ell)(x+y-.)-(D^{0,1}K_\ell)(x-.)\big) \big|}\nonumber\\
&\leq & |y_0| \int_0^1 dr\, \Big\{ \big| \mathbf{X}_{(x_0+ry_0,x_1+y_1)}^{\mathbf{2}}\big((D^{1,1}K_\ell)(x_0+ry_0-.,x_1+y_1-.)\big)  \big|\nonumber\\
& &\hspace{3cm}+\big| \mathbf{X}^{\mathbf{1}}\big( K^{(1)}_{x,(x_0+ry_0,x_1+y_1)}\big)\big| \, \big| \mathbf{X}^{\mathbf{1}}\big((D^{1,1}K_\ell)(x_0+ry_0-.,x_1+y_1-.)\big) \big| \Big\}\nonumber\\
& & +|y_1| \int_0^1 dr\, \Big\{ \big| \mathbf{X}_{(x_0,x_1+ry_1)}^{\mathbf{2}}\big((D^{0,2}K_\ell)(x_0-.,x_1+ry_1-.)\big)  \big|\nonumber\\
& &\hspace{3cm}+\big| \mathbf{X}^{\mathbf{1}}\big( K^{(1)}_{x,(x_0,x_1+ry_1)}\big)\big| \, \big| \mathbf{X}^{\mathbf{1}}\big((D^{0,2}K_\ell)(x_0-.,x_1+ry_1-.)\big) \big| \Big\}\nonumber\\
&\lesssim& \big\{ \lVert \mathbf{X}^{\mathbf{2}}\rVert_{2\al+2;I_{R+1}} + \lVert \mathbf{X}^{\mathbf{1}} \rVert_{\al;I_{R+1}}^2\big\}\cdot  \big\{ |y_0|\cdot \big[ 2^{-\ell(2\al+1)} +\lVert y \rVert_\scal^{\al+2}\cdot 2^{-\ell(\al-1)}\big]\nonumber\\
& & \hspace{7cm}+|y_1| \cdot \big[ 2^{-\ell(2\al+2)}+\lVert y\rVert_\scal^{\al+2}\cdot 2^{-\ell\al} \big] \big\}\ .\label{proof-x-2-2}
\end{eqnarray} 
Putting together (\ref{proof-x-2-1}) and (\ref{proof-x-2-2}) easily provides us with the desired estimate, that is
$$\sum_{0\leq \ell\leq i} \big|V^{\ell}_{x,y} \big| \lesssim \big\{ \lVert \mathbf{X}^{\mathbf{2}}\rVert_{2\al+2;I_{R+1}} + \lVert \mathbf{X}^{\mathbf{1}} \rVert_{\al;I_{R+1}}^2\big\} \cdot \lVert y \rVert_\scal^{2\al+3} \ ,$$
which achieves the proof of (\ref{estim-x-2-2}). A similar strategy can then be implemented towards (\ref{estim-x-2-2-diff}).

\end{proof}

\

\end{document}